\documentclass[twoside,a4paper,12pt,centertags]{amsart}
\usepackage{amsmath,amssymb,verbatim,vmargin}
\usepackage[bookmarks=true]{hyperref}   

\theoremstyle{plain}
\newtheorem{thm}{Theorem}[section]
\newtheorem{lem}[thm]{Lemma}
\newtheorem{cor}[thm]{Corollary}
\newtheorem{prop}[thm]{Proposition}

\theoremstyle{definition}
\newtheorem{defn}[thm]{Definition}
\newtheorem{rem}[thm]{Remark}

\title[Maximal regularity for elliptic systems I]
{Weighted maximal regularity estimates and solvability of non-smooth elliptic systems I}
\author{Pascal Auscher} \author{Andreas Axelsson}
\address{Pascal Auscher, Univ. Paris-Sud, laboratoire de Math\'ematiques, UMR 8628, Orsay F-91405; CNRS, Orsay, F-91405; Centre for Mathematics and its Applications, Australian National University, Canberra ACT 0200, Australia}
\email{pascal.auscher@math.u-psud.fr}
\address{Andreas Axelsson, Matematiska institutionen, Link\"opings universitet, 581 83 Link\"oping, Sweden}
\email{andreas.axelsson@liu.se}

\mathchardef\semic="303B
\newcommand{\R}{{\mathbf R}}
\newcommand{\C}{{\mathbf C}}

\newcommand{\mH}{{\mathcal H}}
\newcommand{\mD}{{\mathcal D}}
\newcommand{\mK}{{\mathcal K}}
\newcommand{\mX}{{\mathcal X}}
\newcommand{\mY}{{\mathcal Y}}
\newcommand{\mL}{{\mathcal L}}

\DeclareMathOperator{\re}{Re}

\newcommand{\sett}[2]{ \{ #1 \, \semic \, #2 \} }

\newcommand{\supp}{\text{{\rm supp}}\,}
\newcommand{\dist}{\text{{\rm dist}}\,}

\newcommand{\nul}{\textsf{N}}
\newcommand{\ran}{\textsf{R}}
\newcommand{\dom}{\textsf{D}}

\newcommand{\clos}[1]{\overline{#1}}
\newcommand{\conj}[1]{\overline{#1}}

\newcommand{\sgn}{\text{{\rm sgn}}}
\newcommand{\barint}{\mbox{$ave \int$}}
\newcommand{\divv}{{\text{{\rm div}}}}
\newcommand{\curl}{{\text{{\rm curl}}}}

\newcommand{\tdd}[2]{\tfrac{\partial #1}{\partial #2}}

\newcommand{\ta}{{\scriptscriptstyle \parallel}}
\newcommand{\no}{{\scriptscriptstyle\perp}}
\newcommand{\pd}{\partial}

\newcommand{\oA}{{\overline A}}
\newcommand{\uA}{{\underline A}}

\newcommand{\loc}{\text{{\rm loc}}}
\newcommand{\tN}{\widetilde N_*}
\newcommand{\hE}{\widehat E}
\newcommand{\tE}{\widetilde E}
\newcommand{\tS}{\widetilde S}
\newcommand{\E}{{\mathcal E}}
\newcommand{\bphi}{\varphi}
\newcommand{\area}{\mathcal A}

\newcommand{\essinf}{\mathop{\rm ess{\,}inf}}
\newcommand{\essup}{\mathop{\rm ess{\,}sup}}

\def\barint_#1{\mathchoice
            {\mathop{\vrule width 6pt
height 3 pt depth -2.5pt
                    \kern -8.8pt
\intop}\nolimits_{#1}}%
            {\mathop{\vrule width 5pt height
3 pt depth -2.6pt
                    \kern -6.5pt
\intop}\nolimits_{#1}}%
            {\mathop{\vrule width 5pt height
3 pt depth -2.6pt
                    \kern -6pt
\intop}\nolimits_{#1}}%
            {\mathop{\vrule width 5pt height
3 pt depth -2.6pt
          \kern -6pt \intop}\nolimits_{#1}}}

\usepackage{color}

\definecolor{gr}{rgb}   {0.,   0.8,   0. } 
\definecolor{bl}{rgb}   {0.,   0.5,   1. } 
\definecolor{mg}{rgb}   {0.7,  0.,    0.7}

\begin{document}

\begin{abstract}
  We develop new solvability methods for divergence form second order, real and complex, elliptic systems
  above Lipschitz graphs, with $L_2$ boundary data.
    The coefficients $A$ may depend on all variables, but are assumed to be close to coefficients $A_0$ that are independent of the coordinate transversal to the boundary, in the Carleson sense $\|A-A_0\|_C$ defined by Dahlberg.
  We obtain a number of {\em a priori} estimates and boundary behaviour results under finiteness of $\|A-A_0\|_C$.
  Our methods yield full characterization of weak solutions, whose gradients have $L_2$ estimates of a non-tangential maximal function or of the square function, via an integral representation acting on the conormal gradient, with a singular operator-valued kernel.
   Also, the non-tangential maximal function of a weak solution is controlled in $L_2$ by the square function of its 
  gradient. This estimate is new for systems in such generality, and even for real non-symmetric equations in dimension $3$
  or higher. The existence of a proof {\em a priori} to well-posedness, is also a new fact.

  As corollaries, we obtain well-posedness of the Dirichlet, Neumann and Dirichlet regularity problems under 
  smallness of $\|A-A_0\|_C$ and well-posedness for $A_0$, improving earlier results for real symmetric equations.
  Our methods build on an algebraic reduction to a first order system first made for coefficients $A_0$ by the two authors 
  and A. McIntosh in order to use functional calculus related to the Kato conjecture solution, 
  and the main analytic tool for coefficients $A$ is an operational calculus to prove weighted maximal regularity estimates.
\end{abstract}

\maketitle



\section{Introduction}

In this article, we present and develop new representations and new solvability methods for 
boundary value problems (BVPs) for
divergence form second order, real and complex, elliptic systems. We look here at BVPs in domains Lipschitz 
diffeomorphic to the upper half space $\R^{1+n}_+ := \sett{(t,x)\in\R\times \R^n}{t>0}$, $n\ge 1$.
The same problems on bounded domains Lipschitz diffeomorphic to the unit ball, contain
noticeable differences, such as use of Fredholm theory, which we address in a forthcoming paper \cite{AA2}.
Here, we focus on the fundamental scale-invariant estimates.

The system of  equations  is 
\begin{equation}  \label{eq:divform}
 Lu^\alpha(t,x)=  \sum_{i,j=0}^n\sum_{\beta= 1}^m \pd_i\Big( A_{i,j}^{\alpha, \beta}(t,x) \pd_j u^{\beta}(t,x)\Big) =0,\qquad \alpha=1,\ldots, m
\end{equation}
in $\R^{1+n}_+$,
where $\pd_0= \tdd{}{t}$ and $\pd_i= \tdd{}{x_i}$, $1\le i\le n$.
We assume
\begin{equation}   \label{eq:boundedmatrix}
  A=(A_{i,j}^{\alpha,\beta}(t, x))_{i,j=0,\ldots,n}^{\alpha,\beta= 1,\ldots,m}\in L_\infty(\R^{1+n};\mL(\C^{(1+n)m})),
\end{equation}
and that $A$ is strictly accretive on $\mH$, meaning that
there exists $\kappa>0$ such that
\begin{equation}   \label{eq:accrassumption}
  \sum_{i,j=0}^n\sum_{\alpha,\beta=1}^m \int_{\R^n} \re (A_{i,j}^{\alpha,\beta}(t,x)f_j^\beta(x) \conj{f_i^\alpha(x)}) dx\ge \kappa 
   \sum_{i=0}^n\sum_{\alpha=1}^m \int_{\R^n} |f_i^\alpha(x)|^2dx,
\end{equation}
for all $f\in\mH$ and a.e. $t>0$.
The definition of $\mH$, a subspace of $L_2(\R^n;\C^{(1+n)m})$, will be given in Section \ref{sec:statement}.

We seek to prove well-posedness for (\ref{eq:divform}), \textit{i.e.}\,  unique solvability
in appropriate spaces given Dirichlet data $u|_{t=0}$, Neumann data 
$\pd_{\nu_A}u|_{t=0}$ or Dirichlet regularity data $\nabla_x u|_{t=0}$, assumed to satisfy an $L_2$ condition.
Note that the continuity estimate required for well-posedness in the sense of Hadamard is not included in our
notion of well-posedness, but will be shown to hold.
For the Neumann and Dirichlet regularity problems, we will work in the class of weak solutions whose gradient 
$\nabla_{t,x}u$ has $L_2$ modified non-tangential maximal function $\tN(\nabla_{t,x} u)$ in $L_2$.
(See Definition~\ref{defn:NTandC}.)
Under our assumptions, we shall describe the limiting behaviour of $\nabla_{t,x}u$ at $t=0$ and $\infty$
and obtain a perturbation result for well-posedness in this class.
For the Dirichlet problem,  it is more natural given our method to work in the class of weak solutions with 
square function estimate $\iint_{\R^{1+n}_+} |\nabla_{t,x}u|^2 tdtdx<\infty$.
Under our assumptions, we shall  describe the limiting behaviour of $u$ at $t=0$ and $\infty$ and 
prove a rigidity theorem that shows new \textit{a priori} non-tangential maximal estimates and $L_2$ estimates, and obtain a perturbation result for well-posedness.

Let us begin by pointing out that the coefficients depend on $t$, which makes these problems not always solvable
in such generality. 
In Caffarelli, Fabes and Kenig~\cite{CaFK}, the necessity of a square Dini condition is pointed out.
There has been a wealth of results for real symmetric equations (\textit{i.e.}\,  $m=1$ and $A_{ij}= A_{ji}\in \R$,
$\mH= L_2(\R^n; \C^{1+n})$).
In Fabes, Jerison and Kenig~\cite{FJK}, the $L_2$ Dirichlet problem is solved under the square Dini condition
and continuity.
Dahlberg removed continuity and proved in \cite{D2} that if the discrepancy $A_1-A_2$ of two matrices $A_1$, $A_2$ satisfies a small 
Carleson condition, \textit{i.e.}\,  if $\|A_1-A_2\|_C$ from Definition~\ref{defn:NTandC} is small enough, 
then $L_{p_1}$-solvability of the Dirichlet problem with coefficients $A_1$ implies
$L_{p_2}$-solvability of the Dirichlet problem with coefficients $A_2$ with $p_2=p_1$.
The smallness condition was removed in Fefferman, Kenig and Pipher~\cite{FKP}, but then the value of 
$p_2$ becomes unspecified.
R. Fefferman obtained in \cite{F} the same conclusions as Dahlberg with $p_2=p_1$, under large perturbation 
conditions of different nature.
See also Lim~\cite{Lim}.
Kenig and Pipher~\cite{KP} proved that the $L_p$-Neumann and regularity problems
are uniquely solvable if the discrepancy $A(t,x)-A(0,x)$ satisfies Dahlberg's small Carleson condition,
depending on $p\in(1,2+\epsilon)$.
Moreover, in \cite{KP2} they proved small perturbation results for the Neumann and regularity problems
analogous the result \cite{D2} for the Dirichlet problem, as well as large perturbation results
for the regularity problem analogous to \cite{FKP} for the Dirichlet problem.

Some related results of  Kenig and Pipher \cite{KP3} (going back to questions of Dahlberg \cite{D2}),  Dindos, Petermichl and Pipher~\cite{DPP} and Dindos and Rule~\cite{DR}
are obtained under smallness of a Carleson condition on $t\nabla_{t,x} A(t,x)$. See also Rios' work \cite{R}. Such an hypothesis does not compare to the one on $A(t,x)-A(0,x)$.

We note that these results are obtained for $L_p$ data, for appropriate $p$'s, including $p=2$.
This is using all the available technology for {\em real scalar equations}, starting from the
maximum principle, hence $L$-harmonic measure, and Green's functions.
Moreover, as far as solvability is concerned, the main thrust of these works is to get $p=2$
with non-tangential maximal estimates, using for this the classical variational solutions,
or those obtained via the maximum principle.

Of course, $t$-dependent coefficients incorporate the $t$-independent ones. 
We refer to the book by Kenig~\cite{kenig} and references therein, and to 
Alfonseca, Auscher, Axelsson, Hofmann and Kim~\cite{AAAHK} for more recent results on $L_{\infty}$ perturbation of real symmetric (or complex and constant) equations.
See also below.  

We mention a series of works for two dimensional equations on the upper half-plane with $t$-independent coefficients.  Auscher and Tchamitchian~\cite{AT} study complex coefficients equations with diagonal $A$ (which we call here block form) and describe Dirichlet, regularity and Neumann problems for $L_{p}$ data for $p>1$ and even for  data of Hardy type for $p\le 1$. 
This is a precursor of the work for systems here, as it built upon new proofs relying on Calder\'on-Zygmund operators (which are no longer available here) of the one dimensional Kato conjecture proved earlier by Coifman, McIntosh and Meyer~\cite{CMcM} and its extension  by Kenig and Meyer~\cite{KM}. For real equations of non block forms, Kenig, Koch, Pipher and Toro~\cite{KKPT} show that the Dirichlet problem is well-posed for large enough $p$ (and obtain counterexamples for any given specific $p$) by showing that $L$-harmonic measure is absolutely continuous with respect to Lebesgue measure on
the boundary\footnote{A recent preprint by  Dindos, Kenig and Pipher~\cite{DKP} posted during the revision of this article   shows this is related to well-posedness with $\text{BMO}$ data.}. Kenig and Rule~\cite{KR} then obtained well-posedness for the Neumann and regularity problems with $p-1>0$ small enough (and obtained counterexamples for any specific $p>1$). 
The recent thesis of Barton \cite{B} 
deals with complex, $t$-independent  $L_{\infty}$ perturbations of the situation
in \cite{KKPT,KR}, and she obtains well-posedness of the Neumann and regularity problems in $L_{p}$ for $p-1>0$ small and even at $p=1$ with data in the classical Hardy space.

As the reader has observed, we consider complex systems and we wish to obtain $L_2$ solvability under conditions as general as possible (we mention that $L_{p}$ solvability with our methods when $p\ne 2$ is under  study at this time). 
For this, we need other tools than those mentioned above.
In fact, the tools we develop and that we describe next would not have been conceivable prior to 
the full solution in all dimensions of the Kato problem and its extensions.
In Auscher, Axelsson and McIntosh~\cite{AAM}, a new method was presented for solving BVPs with $t$-independent 
coefficients, following an earlier setup designed in Auscher, Axelsson and Hofmann~\cite{AAH}. 
The main discovery in \cite{AAM} is that the
equation (\ref{eq:divform}) becomes particularly simple when solving for the conormal gradient
defined by \begin{equation}\label{eq:conormalgrad}
 f= \nabla_A u:= \begin{bmatrix} \pd_{\nu_A}u \\ \nabla_x u \end{bmatrix},
\end{equation}
where $ \pd_{\nu_A}u$ denotes the (inward for convenience) conormal derivative (see Section~\ref{sec:difftoint}),
instead of the potential $u$ itself. It is a set of generalized Cauchy--Riemann equations expressed as an
autonomous first order system
\begin{equation}   \label{eq:autonomous}
  \pd_t f + DB f=0,
\end{equation}
where $D$ is a self-adjoint (but not positive) 
first order differential operator with constant coefficients that is elliptic in some sense and $B$ 
is multiplication with a bounded matrix $B(x)$, which is strictly accretive on the space $\mH$ in 
(\ref{eq:accrassumption}) and related to 
$A(x)=A(t,x), t>0$, by an explicit algebraic formula.
The operator $DB$ is a bisectorial operator and can be shown to have an $L_2$-bounded holomorphic
functional calculus for any ($t$-independent) matrix $A$ satisfying (\ref{eq:boundedmatrix})
and (\ref{eq:accrassumption}).
This fact was proved earlier by Axelsson, Keith and McIntosh~\cite[Theorem 3.1]{AKMc} elaborating
on the technology for the solution of the Kato problem by Auscher, Hofmann, Lacey, McIntosh and
Tchamitchian~\cite{AHLMcT}; a more direct proof is proposed in Auscher, Axelsson and McIntosh~\cite{elAAM}.
As explained there, the main difficulty is the non-injectivity of $D$.
The upshot is the possibility of solving (\ref{eq:autonomous}) by a semi-group formula 
$f= e^{-t|DB|}f_0$ with $f_0$ in a suitable trace space, and such $f$ has non-tangential and square 
function estimates.
The BVP can then be solved in an appropriate class if and only if the map from the trace functions to boundary data is invertible.
This is the scheme for the Neumann and regularity problems, for which the boundary data is simply 
the normal or tangential part of $\nabla_A u$. 
For the Dirichlet problem, it turns out that a ``dual'' scheme involving the operator $BD$ can be used 
similarly.
The one-to-one correspondence between trace functions $f_0$ and boundary data may fail, see 
Axelsson~\cite{AxNon}, and it is here that restrictions on $A$ appear.
It is known to hold if $A$ is (complex) self-adjoint or block form (\textit{i.e.}\,  no cross derivatives $\pd_0A\pd_i$
or $\pd_i A \pd_0$, $i\ge 1$, in (\ref{eq:divform})), or constant.
Another consequence of this method, and this is why considering complex coefficients is useful, is that
the set of $t$-independent $A$'s for which solvability holds is open in $L_\infty$. See \cite{AAM}.

Our work for $t$-dependent coefficients takes the algebraic reduction to (\ref{eq:autonomous})
as a starting point, the conormal gradient becoming the central object.  We shall state the main results in Section~\ref{sec:statement} and 
explain the strategy in Section~\ref{sec:roadmap}. It involves in particular study of a highly singular integral operator $S_A$, with an operator-valued kernel.  On a technical level, proper definition and handling of this operator is most efficiently done using
operational calculus rather than the usual maximal regularity treatment originally due to de Simon~\cite{dSim} (see Kunstmann and Weis~\cite[Chapter 1]{KW} for an overview) and this avoids having to assume qualitatively that $A$ is smooth in the calculations.
We use the terminology operational calculus, following the thesis \cite{Alb} of Albrecht, 
for the extension of 
functional calculus when not only scalar holomorphic functions are applied to the underlying
operator (in our case $DB_0$ with $B_{0}(x)=B(0,x)$), but more general operator-valued holomorphic functions.
The Hilbert space theory we use here to prove boundedness on appropriate functional spaces in Section \ref{sec:saest}, surveyed in Section~\ref{sec:abstropcalc},
is a special case of the general theory developed in
Albrecht, Franks and McIntosh~\cite[Section 4]{AFMc}, 
Lancien, Lancien and LeMerdy~\cite{LLleM}, and Lancien and LeMerdy~\cite{LleM}.
For further details and references, we refer to Kunstmann and Weis~\cite[Chapter 12]{KW}.

The Carleson control on the discrepancy $A(t,x)-A(0,x)$ from \cite{D, FKP, KP, KP2} appears in a very natural way in the estimates of $S_A$, and well-posedness of the three BVPs with coefficients $A(t,x)$ will follow under smallness of this control and well-posedness for coefficients $A(0,x)$. We mention that the Dirichlet problem could be obtained by an abstract duality procedure from a regularity problem, in agreement with the results of \cite{KP, KP2} for real symmetric equations. See also Kilty and Shen~\cite{Shen1}, and Shen~\cite{Shen2}. We will formalize this  abstract procedure  in our subsequent work \cite{AA2}.
We remark however that although the hypotheses are the same for each BVP, the perturbation results can be proved independently of one another. For example, one does not need knowledge on well-posedness of regularity for $A$ or of Dirichlet for $A^*$ to obtain well-posedness of Neumann for $A$, in contrast with the results in \cite{KP2}. 

We do not know how to prove well-posedness under the finiteness of $\|A(t,x)-A(0,x)\|_C$ only. (In light of \cite{FKP, KP2}, 
this would first require to extend our methods to solvability for $L_p$ data.)
However, thanks to our representations, we do obtain under this hypothesis a number of {\em a priori} estimates and boundary behaviour on solutions to the equation
(\ref{eq:divform}) without any assumption on well-posedness.  For example, and we concentrate on this to finish this introduction, we show that if $\|A(t,x)-A(0,x)\|_C<\infty$, all weak solutions to (\ref{eq:divform}) with coefficients $A$ such that $\iint_{\R^{1+n}_+} |\nabla_{t,x}u|^2 tdtdx<\infty$
 satisfy, up to a constant,  sup-$L_{2}$ estimates as well as non-tangential maximal functions estimates. More precisely, we obtain
$$
 \max(\sup_{t>0}\|u_t\|_2, \|\tN(u)\|_2) \lesssim \|\nabla_{t,x}u\|_{L_2(tdt; L_2)}
$$
provided  the right hand side is finite and the solution $u$ vanishes at infinity in some sense (see Section \ref{sec:statement} for precise meaning).
Note in particular that this applies when $A$ is $t$-independent and in that case, this is implicit from \cite[Corollary 4.2]{AAM}
when restricted to the class of solutions considered there. 
Domination of the non-tangential maximal function $\|\tN(u)\|_2$ by the square function
$\|\nabla_{t,x}u\|_{L_2(tdt; L_2)}\approx \|S(u)\|_2$, $S(u)(x)= (\int_{|y-x|<t} |\nabla_{t,y} u|^2 dtdy/t^{n-1} )^{1/2}$,
is reminiscent of the result of Dahlberg, Jerison and Kenig~\cite{DJK}, and also of 
Dahlberg, Kenig, Pipher and Verchota~\cite{DKPV}.
But there is a difference. In \cite{DJK} comparability of $\tN(u)$ and $S(u)$ is obtained
for solutions of the equation (\ref{eq:divform}) under (\ref{eq:boundedmatrix}) and (\ref{eq:accrassumption}),
$A$ real and $m=1$,
in all $L_q(\R^n; d\mu)$ spaces, $0<q<\infty$, with $\mu$ a doubling $A_\infty$ weight 
with respect to $L$-harmonic measure.
If the Dirichlet problem in the class $\|\tN(u)\|_p<\infty$ is proved to be 
solvable for one $1<p<\infty$, then Lebesgue measure is $A_\infty$ of $L$-harmonic measure, hence
$\|\tN(u)\|_q\approx \|S(u)\|_q$.
This fact follows in particular from combining \cite{JK1} and \cite{FJK} under $\|A(t,x)-A(0,x)\|_C<\infty$
and $A$ real symmetric. 
In \cite{DKPV}, comparability $\|\tN(u)\|_q\approx \|S(u)\|_q$, $0<q<\infty$, is obtained for
real symmetric constant elliptic (in the sense of Legendre--Hadamard) 
second order systems (and also higher order but the formulation 
becomes different) on  bounded Lipschitz domains owing to the fact that $L_2$ solvability 
of the Dirichlet problem was known (see the introduction of \cite{DKPV}).
This comparability also follows for real non-symmetric scalar equations in two dimensions combining 
the results of \cite{KKPT} and again \cite{DJK}.
Here, although we obtain only one part of the comparison, it is essential to note that this
is an {\em a priori} estimate valid independently of well-posedness. The existence of an {\em a priori} proof is new even for real symmetric scalar equations under the Carleson control (for example for all the $t$-independent ones),
and is permitted by the solution of the Kato square root problem and its extensions.

{\bf Acknowledgments.}

The first author thanks the Mathematical Sciences Institute at the Australian National University for
hospitality where part of this work was done.
The second author was supported by a travel grant from the Swedish Research Council to 
participate in the El Escorial conference 2008, where this work got started.
Both authors also thank the LAMFA in Amiens for hospitality.

Thanks to Chema Martell for a careful reading of a preliminary version of this manuscript, and
to Steve Hofmann for discussions on the comparability between non-tangential maximal function
and square function estimates. 
We are also indebted to the referees for many valuable suggestions, including the writing 
of the informal presentation in Section 3.

%
%
%
%
%
\section{Statement of results}   \label{sec:statement}

In this section we state our results concerning solvability of boundary value problems
on the half space $\R^{1+n}$, and show how they extend to domains $\Omega$ which are 
Lipschitz diffeomorphic to $\R^{1+n}_+$. 

Let us first fix notation.
We write $\{e_0,e_1,\ldots,e_n\}$ for the standard basis for $\R^{1+n}$ with 
$e_0$ ``upward'' pointing into $\R^{1+n}_+$, and write $t=x_0$ for the vertical
coordinate. For the vertical derivative, we write $\partial_0 = \partial_t$.
For an $m$-tuple of vectors $v=(v_i^\alpha)_{0\le i\le n}^{1\le \alpha\le m}$, 
we write $v_\no$ and $v_\ta$ for the
normal and tangential parts of $v$, 
\textit{i.e.}\, ~$(v_\perp)_0^\alpha= v_0^\alpha$ and $(v_\perp)_i^\alpha=0$ when $1\le i\le n$,
whereas $(v_\ta)_i^\alpha= v_i^\alpha$ when $1\le i\le n$ and $(v_\ta)_0^\alpha=0$.
We write $f_t(x):= f(t,x)$ for functions in $\R^{1+n}_+$.
As compared to \cite{AAM}, we here use subscript $0$ to denote restriction to the boundary $\R^n$
at $t=0$, rather than the normal component of $f$. 
We also prefer to use small letters $f,g,\ldots$ to denote functions in $\R^{1+n}_+$, since
this is where we work most of the time, not on the boundary as in \cite{AAM}.

For tuples of functions and vector fields, gradient and divergence act as $(\nabla_{t,x}u)_i^\alpha= \pd_i u^\alpha$ and 
$(\divv_{t,x}f)^\alpha= \sum_{i=0}^n \pd_i f^\alpha_i$, with corresponding
tangential versions $\nabla_x u= (\nabla_{t,x}u)_\ta$
and $(\divv_x f)^\alpha= \sum_{i=1}^n \pd_i f^\alpha_i$.
With $\curl_{t,x} f=0$ we understand that 
$\pd_jf_i^\alpha= \pd_i f_j^\alpha$, for all $i$, $j=0,\ldots,n$.
Similarly, write $\curl_{x} f_\ta=0$ if 
$\pd_jf_i^\alpha= \pd_i f_j^\alpha$, for all $i$, $j=1,\ldots,n$.

We assume that $A$ is {\em strictly accretive on} $\mH$, \textit{i.e.}\,  
$$
 \int_{\R^n} \re(A(t,x) f(x),f(x)) dx \ge \kappa \int_{\R^n} |f(x)|^2 dx,
$$
holds for some constant $\kappa>0$, uniformly for $t>0$ and all $f$ belonging to the closed subspace
\begin{equation}     \label{eq:Hspace}
  \mH:= \nul(\curl_x)=\sett{g\in L_2(\R^n;\C^{(1+n)m})}{\curl_x(g_\ta)=0}.
\end{equation}
For scalar equations, \textit{i.e.}\,  $m=1$, (\ref{eq:accrassumption}) amounts to the pointwise condition
$$
  \re(A(t,x)\zeta,\zeta) \ge \kappa |\zeta|^2, \qquad \text{for all } \zeta\in \C^{1+n}, \text{a.e. } (t,x)\in \R^{1+n}_+.
$$ 
For systems, (\ref{eq:accrassumption}) is stronger than a strict G\aa rding inequality 
on $\R^{1+n}_+$ (\textit{i.e.}\,  integration would be on $\R^{1+n}_+$ and $f$ such that $\curl_{t,x}f=0$);
still (\ref{eq:accrassumption}) is natural given the type of perturbation we consider here.

The boundary value problems we consider are to find $u\in \mD'(\R_+^{1+n};\C^m)$ solving the
divergence form second order elliptic system
$$
   \divv_{t,x} A \nabla_{t,x} u=0 \text{ in } \R^{1+n}_+
$$
in distributional sense, with appropriate interior estimates of $\nabla_{t,x} u$ and
satisfying one of the following three natural boundary conditions.
\begin{itemize}
\item The Dirichlet condition $u= \bphi$ on $\R^n$, 
given $\bphi\in L_2(\R^n;\C^m)$.
\item The Dirichlet regularity condition $\nabla_x u= \bphi$ on $\R^n$, given
$\bphi\in L_2(\R^n;\C^{nm})$ satisfying $\curl_x \bphi=0$.
\item The Neumann condition $(e_0, A\nabla_{t,x} u)= \bphi$ on $\R^n$, given
$\bphi\in L_2(\R^n;\C^{m})$.
\end{itemize}
Note that since we shall impose distributional $\nabla_{t,x} u \in L_2^\loc$, $u$ 
can be identified with a function in  $W^{1,\loc}_2$, \textit{i.e.} with a weak solution.

\begin{defn}    \label{defn:NTandC}
The {\em modified non-tangential maximal function} of a function $f$ in $\R^{1+n}_+$ is
$$
  \tN(f)(x):= \sup_{t>0}  t^{-(1+n)/2} \|f\|_{L_2(W(t,x))}, \qquad x\in \R^n,
$$
where $W(t,x):= (c_0^{-1}t,c_0t)\times B(x;c_1t)$, for some fixed constants $c_0>1$, $c_1>0$.
The {\em modified Carleson norm} of a function $g$ in $\R^{1+n}_+$ is
$$
  \|g\|_C := \left( \sup_Q \frac 1{|Q|} \iint_{(0, l(Q))\times Q} \sup_{W(t,x)}|g|^2 \frac {dtdx}{t} \right)^{1/2},
$$
where the supremum is taken over all cubes $Q$ in $\R^n$, with $l(Q)$ denoting their side lengths.
\end{defn}

Note that different choices for $c_0, c_1$ will give different, but equivalent norms $\|\tN(f)\|_2$,
as well as equivalent norms $\|g\|_C$.
Furthermore, this maximal function is really non-tangential since $\tN(f)$ and the closely 
related maximal function
$\sup_{|y-x|<t}  t^{-(1+n)/2} \|f\|_{L_2(W(t,y))}$ have equivalent $L_2$ norms.
The latter was introduced in \cite{KP}.
The modified Carleson norm originates from Dahlberg~\cite{D2}.

We will use the modified Carleson norm to measure the size of perturbations 
of $t$-independent coefficients $A_0$. 
(In fact we shall use a possibly weaker quantity $\|\cdot\|_*$; see Section~\ref{sec:functionspace}.) 
Intuitively, $\|A-A_0\|_C < \infty$ means that in a certain sense $A(t, x)=A_0(x)$ 
at $t = 0$, but also that $A(t,x)$ is close to $A_0(x)$
at all scales since we are dealing with a scale invariant norm. 
Also, given $A$, such a $A_0$ is unique and has controlled bounds thanks to
the following lemma.  We remark here that  the modified Carleson norm there 
can be weakened to the usual one. See Section~\ref{sec:functionspace} for proofs.

\begin{lem}   \label{lem:carlesontrace}
  Fix $A: \R^{1+n}_+\to \mL(\C^{(1+n)m})$ with $\|A\|_\infty<\infty$ and strictly accretive on $\mH$, with
  constant of accretivity $\kappa>0$.
  Assume that $A_0$ are $t$-independent measurable coefficients such that $\|A-A_0\|_C<\infty$.
  Then $A_0$ is uniquely determined by $A$, \textit{i.e.} if 
   $A'_0$ are $t$-independent coefficients such that $\|A-A'_0\|_C<\infty$,
  then $A'_0=A_0$ almost everywhere.
  Furthermore $A_0$ is bounded and strictly accretive on $\mH$, with
$$
  \kappa\le \kappa_0\le \|A_0\|_\infty\le \|A\|_\infty,
$$
where $\kappa_0$ denotes the constant of accretivity of $A_0$.
\end{lem}

For the Neumann and Dirichlet regularity problems, our result is the following.

\begin{thm}   \label{thm:NeuLip}
   Consider coefficients $A\in L_\infty(\R^{1+n}_+; \mL(\C^{(1+n)m}))$ which are strictly accretive on $\mH$.
\begin{itemize}
\item[{\rm (i)}]
{\rm {\em A priori} estimates:}
 Assume that $u\in W^{1,\loc}_2(\R^{1+n}_{+}, \C^m)$ has gradient $\nabla_{t,x} u$ with 
 estimate $\|\tN(\nabla_{t,x}u)\|_2 <\infty$, and where $u$ satisfies
 (\ref{eq:divform}) in $\R^{1+n}_+$ distributional sense.
 If there exists $t$-independent measurable coefficients $A_0$ such that
 $\| A-A_0 \|_C <\infty$, then $\nabla_{t,x}u$ has limits 
$$
  \lim_{t\to 0} t^{-1} \int_t^{2t} \| \nabla_{s,x} u_s - g_0 \|_2^2 ds =0=
   \lim_{t\to \infty} t^{-1} \int_t^{2t} \| \nabla_{s,x} u_s \|_2^2 ds,
$$
for some function $g_0 \in L_2(\R^n ;\C^{(1+n)m})$, with estimate 
$\|g_0\|_2\lesssim\|\tN(\nabla_{t,x}u)\|_2$.

\item[{\rm (ii)}]
{\rm Well-posedness:}
  By the Neumann problem with coefficients $A$ (or $A_0$) being well-posed, we mean that
  given $\bphi\in L_2(\R^n;\C^m)$, there is a function $u\in W^{1,\loc}_2(\R^{1+n}_{+}, \C^m)$,
  unique modulo constants, 
  solving  (\ref{eq:divform}), with coefficients $A$ (or $A_0$), and having estimates as in (i)
  and trace $g_0= \lim_{t\to 0}\nabla_{t,x} u$ such that $(A_0 g_0)_\no= \bphi$.
  
  The following perturbation result holds.
  If the Neumann problem for $A_0$ is well-posed,
  then there exists $\epsilon>0$ such that if $\| A-A_0 \|_C <\epsilon$, then the Neumann
  problem is well-posed for $A$. 
  
The corresponding result holds when the Neumann problem is replaced by the regularity problem
and the boundary condition $(A_0 g_0)_\no= \bphi$ is replaced by
 $(g_0)_\ta= \bphi\in L_2(\R^n; \C^{nm})$, where $\bphi$ satisfies $\curl_x \bphi=0$.
Moreover, for both BVPs the solutions $u$ have estimates
 $$
   \|\tN(\nabla_{t,x} u)\|_2 \approx \|g_0\|_2 \approx \|\bphi\|_2.
 $$

\item[{\rm (iii)}]
{\rm Further regularity:}
  Assume that $A_0$ is as in (i), with $\| A-A_0 \|_C$ sufficiently small
  and consider solutions $u$ as in (i).
  
  If $A$ satisfy the $t$-regularity condition $\|t\pd_t A\|_C<\infty$,
 then
 $$
   \int_0^\infty \| \pd_t \nabla_{t,x} u \|_2^2 tdt\lesssim \sup_{t>0}\|\nabla_{t,x} u_t\|_2^2
   \approx \|\tN(\nabla_{t,x}u)\|^2_2,
 $$
 $t\mapsto \nabla_{t,x}u_t\in L_2$ is continuous and $\lim_{t\to 0} \|\nabla_{t,x}u_t -g_0\|_2 =0=\lim_{t\to\infty}\|\nabla_{t,x}u_t\|_2$.
The converse estimate $\|\tN(\nabla_{t,x}u)\|^2_2\lesssim\int_0^\infty \| \pd_t \nabla_{t,x} u \|_2^2 tdt$ 
holds provided $\|t\pd_t A\|_C$ is sufficiently small.

If $\max(\|t\pd_i A\|_C, \|t\pd_t A\|_C)<\infty$ holds for some $i=1,\ldots,n$, then 
$$
\int_0^\infty \| \pd_i \nabla_{t,x} u \|_2^2 tdt\lesssim \|\tN(\nabla_{t,x}u)\|^2_2.
$$
The estimate $ \|\tN(\nabla_{t,x}u)\|^2_2\lesssim \sum_{i=1}^n \int_0^\infty \| \pd_i \nabla_{t,x} u \|_2^2 tdt$
holds provided $\|t\nabla_{t,x}A\|_C$ is sufficiently small.
\end{itemize}
\end{thm}
Under the hypothesis $\|A-A_0\|_C<\infty$, 
the implicit constant in (i) depends on $n$, $m$, $\|A\|_\infty$, $\kappa$.
In (ii) and (iii), under the smallness hypothesis on $\| A-A_0 \|_C$, 
which depends on $n$, $m$, $\|A\|_\infty$, $\kappa$ 
(and in (ii) also on the ``well-posedness constant'' for $A_0$),
the implicit constants depend on $n$, $m$, $\|A\|_\infty$, $\kappa$
(and in (iii) also on the regularity assumptions on $A$).
However, some inequalities are true without smallness and/or well-posedness as the reader can check on reading details in Section \ref{sec:neu}.
Note that in (ii), the uniqueness holds in the class defined by $\|\tN(\nabla_{t,x}u)\|_2<\infty$.

For the Dirichlet problem, our main result is the following, 
including a rigidity result of solutions with square function estimates.

\begin{thm}   \label{thm:DirLip}
   Consider coefficients $A\in L_\infty(\R^{1+n}_+; \mL(\C^{(1+n)m}))$ which are strictly accretive on $\mH$.
\begin{itemize}
\item[{\rm (i)}]
  {\rm {\em A priori} representation and estimates:} 
  Assume that $u\in W^{1,\loc}_2(\R^{1+n}_{+}, \C^m)$ has gradient bounds
  $\int_0^\infty \|\nabla_{t,x}u_t\|_2^2 tdt<\infty$  and satisfies
 (\ref{eq:divform}) in $\R^{1+n}_+$ distributional sense.
 If there exists $t$-independent measurable coefficients $A_0$ 
 such that $\| A-A_0 \|_C <\infty$, then $u= \hat u+c$ almost everywhere,
 for a unique $\hat u\in C(\R_+; L_2(\R^n;\C^m))$ and constant $c\in \C^m$.
 Identifying the functions $u$ and $\hat u+c$, it has 
 $L_2$ limits 
$$
  \lim_{t\to 0} \| u_t -\hat u_0-c \|_2 =0=
   \lim_{t\to \infty} \| u_t -c\|_2,
$$
for some $\hat u_0\in L_2(\R^n;\C^m)$, and we have estimates
$$
\max(\|\tN(\hat u)\|^2_2, \sup_{t>0}\|\hat u_t \|^2_2)\lesssim \int_0^\infty \|\nabla_{t,x}u\|_2^2 tdt.
$$

\item[{\rm (ii)}]
  {\rm Well-posedness:}
  By the Dirichlet problem with coefficients $A$ (or $A_0$) being well-posed, we mean that 
  given $\bphi\in L_2(\R^n;\C^m)$, there is a unique $u\in W^{1,\loc}_2(\R^{1+n}_{+}, \C^m)$
  having gradient bounds and solving (\ref{eq:divform}), with coefficients $A$ (or $A_0$), 
  as in (i), and trace $u_0= \bphi$.
  
  The following perturbation result holds.
  If the Dirichlet problem for $A_0$ is well-posed,
  then there exists $\epsilon>0$ such that if $\| A-A_0 \|_C <\epsilon$, then the Dirichlet
  problem is well-posed for $A$. 
  Moreover, these solutions $u$ have estimates
 $$
   \|\tN(u)\|_2^2 \approx \sup_{t>0}\|u_t \|_2^2 \approx
   \int_0^\infty \|\nabla_{t,x} u\|_2^2 tdt \approx \|\bphi\|_2^2.
 $$
\end{itemize}
\end{thm}

Say that  a function $w\in L_{2}^{\loc}(\R_{+}, L_{2})$ vanishes at $\infty$ in $L_{2}$ sense if $\lim_{t\to \infty}\|w_{t}\|_{2}=0$. For the solutions $u$ as in (i), we see that the following three statements are equivalent:   $u$ vanishes at $\infty$ in $L_{2}$ sense,  $u\in C(\R_{+},L_{2}(\R^n;\C^m))$,  $u_{0}\in L_2(\R^n;\C^m)$. So
 in (ii), the boundary condition $u_0\in L_2(\R^n;\C^m)$ forces $c=0$ and $u=\hat u$.
 
Under the hypothesis $\|A-A_0\|_C<\infty$, 
the implicit constants in (i) depend on $n$, $m$, $\|A\|_\infty$, $\kappa$.
In (ii), under the smallness hypothesis on $\| A-A_0 \|_C$, 
which depends on $n$, $m$, $\|A\|_\infty$, $\kappa$ and the ``well-posedness constant'' for $A_0$,
the implicit constants depend on $n$, $m$, $\|A\|_\infty$, $\kappa$.
However, some inequalities are true without smallness and/or well-posedness as the reader can check on reading details in Section \ref{sec:dirichlet}.
Note that in (ii), uniqueness holds in the class defined by $\int_0^\infty\|\nabla_{t,x}u\|_2^2 tdt<\infty$.

As mentioned briefly in the introduction, 
the hypothesis on well-posedness with $t$-independent
coefficients $A_0$ is satisfied, for all three BVPs, for 
Hermitean coefficients, \textit{i.e.}\,  $A_0(x)^*= A_0(x)$, 
for block form coefficients, \textit{i.e.}\,  $(A_0)_{\no\ta}=0= (A_0)_{\ta\no}$, and for 
constant coefficients, \textit{i.e.}\,  $A_0(x)=A_0$, as well as for
sufficiently small $t$-independent $L_\infty(\R^n; \mL(\C^{(1+n)m}))$ perturbations thereof.
This was proved in \cite[Theorem 2.2]{AAM}.
That the notions of well-posedness of these BVPs used in \cite{AAM} coincide with the ones here, for 
$t$-independent coefficients,
follows from Corollaries~\ref{eq:WPcompatNeu} and \ref{eq:WPcompatDir}.

Note that we do not assume pointwise bounds on the solutions, hence we use $\tN$ instead 
of the usual non-tangential maximal function.

When $m=1$ and $A$, $A_0$ are real symmetric (and $\R^{1+n}_+$ 
replaced by the unit ball),
Theorem~\ref{thm:NeuLip}(ii) is in \cite{KP}, and Theorem~\ref{thm:DirLip}(ii) is in \cite{D}
(and \cite{DJK} for the square function estimate).
The rest of Theorems~\ref{thm:NeuLip} and \ref{thm:DirLip} is mostly new.
In Section~\ref{sec:roadmap}, a more detailed road map to the proofs is given.

\begin{proof}[Proof of Theorems~\ref{thm:NeuLip} and \ref{thm:DirLip}]
  For the Neumann and regularity problems in $\R^{1+n}_+$, part (i) follows from Theorems~\ref{thm:inteqforNeu} and \ref{cor:inteqforNeu},
  part (ii) follows from Corollary~\ref{cor:neuregmain}, and part (iii) is proved in Theorem~\ref{thm:neuQE}.
  
 For the Dirichlet problem in $\R^{1+n}_+$, part (i) follows from Theorems~\ref{thm:inteqforDir}
 and \ref{thm:Ysols},
 and part (ii)
  follows from Corollary~\ref{cor:dirmain}, except for the estimate of the non-tangential maximal function, 
  which is proved in Theorem~\ref{thm:dirNT}.
\end{proof}

We end this section with a remark on the Lipschitz invariance of the above results.
Let $\Omega\subset\R^{1+n}$ be a domain which is Lipschitz diffeomorphic to $\R^{1+n}_+$,
and let $\rho: \R^{1+n}_+\to \Omega$ be the Lipschitz diffeomorphism. Denote the boundary
by $\Sigma:= \partial \Omega$ and the restricted boundary Lipschitz diffeomorphism
by $\rho_0: \R^n\to \Sigma$.

Given a function $\tilde u: \Omega\to \C^m$, we pull it back to 
$u:= \tilde u\circ \rho: \R^{1+n}_+\to \C^m$.
By the chain rule, we have $\nabla_{t,x} u= \rho^* (\nabla_{t,x} \tilde u)$, where
the pullback of an $m$-tuple of vector fields $f$, is defined as
$\rho^*(f)(x)^\alpha:= \underline{\rho}^t (x) f^\alpha(\rho(x))$, with $\underline{\rho}^t$
denoting the transpose of Jacobian matrix $\underline{\rho}$.
If $\tilde u$ satisfies $\divv_{t,x} \tilde A \nabla_{t,x} \tilde u=0$ in $\Omega$, 
with coefficients $\tilde A\in L_\infty(\Omega;\mL(\C^{(1+n)m}))$, then $u$ will satisfy 
$\divv_{t,x} A \nabla_{t,x} u=0$ in $\R^{1+n}_+$, where 
$A\in L_\infty(\R^{1+n}_+;\mL(\C^{(1+n)m}))$ is defined as
\begin{equation}   \label{eq:pullbackA}
  A({\bf x}):= |J(\rho)({\bf x})| (\underline{\rho}({\bf x}))^{-1} \tilde A(\rho({\bf x})) (\underline{\rho}^t({\bf x}) )^{-1}, 
  \qquad {\bf x}\in \R^{1+n}_+.
\end{equation}
Here $J(\rho)$ is the Jacobian determinant of $\rho$.

The boundary conditions on $\tilde u$ on $\Sigma$ translate in the following way to 
boundary conditions on $u$ on $\R^n$.
\begin{itemize}
\item The Dirichlet condition $\tilde u= \tilde \bphi$ on $\Sigma$
is equivalent to the Dirichlet condition $u= \bphi$ on $\R^n$, 
where $\bphi := \tilde \bphi\circ \rho_0 \in L_2(\R^n;\C^m)$.
\item The Dirichlet regularity condition $\nabla_\Sigma \tilde u= \tilde \bphi$ on $\Sigma$ ($\nabla_\Sigma$
denoting the tangential gradient on $\Sigma$),
is equivalent to $\nabla_x u= \bphi$ on $\R^n$, where 
$\bphi := \rho_0^*(\tilde \bphi) \in L_2(\R^n;\C^{nm})$.
\item The Neumann condition $(\nu, \tilde A\nabla_\Omega \tilde u)= \tilde \bphi$ on $\Sigma$
(contrary to tradition, $\nu$ being the inward unit normal vector field on $\Sigma$)
is equivalent to $(e_0, A\nabla_{t,x} u)= \bphi$ on $\R^n$, where 
$\bphi := |J(\rho_0)| \tilde \bphi\circ \rho_0 \in L_2(\R^n;\C^{m})$.
\end{itemize}

In this way the Dirichlet/regularity/Neumann problem with coefficients $\tilde A$ in 
the Lipschitz domain $\Omega$ is equivalent to the Dirichlet/regularity/Neumann problem 
with coefficients $A$ in the half space $\R^{1+n}_+$,
and it is straightforward to extend Theorems~\ref{thm:NeuLip} and \ref{thm:DirLip} to
the Lipschitz domain $\Omega$.

%
%
%
%
%
\section{Road map to the proofs}   \label{sec:roadmap}

For the reader's convenience, we give in this section an informal explanation of the main ideas behind the methods and
the proofs of Theorems~\ref{thm:NeuLip} and \ref{thm:DirLip}. In particular, the precise definitions of classes of solutions will be given later. 
Our basic idea for constructing solutions $u$ to the divergence form
equation (\ref{eq:divform}) in $\R^{1+n}_+$ is to consider it as a first order system with the 
gradient $\nabla_{t,x} u$ as the unknown function.
In fact, solving for the $t$-derivatives in the equation, the divergence form equation for $u$ becomes
a vector-valued ODE 
$$
  \pd_t (\nabla_{t,x} u) + T_A (\nabla_{t,x} u) =0,
$$
where $T_A$ is an operator only involving the first order derivatives along $\R^n$ and multiplication by entries of $A(t,x)$.
The divergence form equation was first studied through this ODE in \cite{AAH}.
However, it turns out that if one instead of $\nabla_{t,x} u$ takes the conormal gradient $\nabla_A u$ defined by \eqref{eq:conormalgrad} as the unknown,
then the corresponding operator $T_A$ has a simpler structure; the ODE reads
\begin{equation}   \label{eq:ODEroadmap}
  \pd_t  f+ D B f =0,\qquad \text{with } f:= \nabla_A u,
\end{equation}
where 
$
  D:= 
    \begin{bmatrix} 0 & \divv_x  \\ 
     -\nabla_x & 0 \end{bmatrix}
$
and $B$ 
is a second strictly accretive coefficient matrix determined by $A$.
This was the key discovery in \cite{AAM} when coefficients $A$ do not depend on $t$. 
This carries over to $t$-dependent coefficients and as this result is central to us, we give full proof
of it in Proposition~\ref{prop:divformasODE}.

\subsection{The Neumann and regularity problems}   \label{sec:roadNeu}

The first order approach is most natural for solving the Neumann and regularity BVPs,
since these boundary conditions are conditions on the conormal gradient $f$, not on the potential $u$.
Indeed, the Neumann BVP means that the normal part $(f_0)_\no= \pd_{\nu_A} u|_{\R^n}$ is given at the boundary $t=0$,
whereas the regularity condition is that the tangential part $(f_0)_\ta= \nabla_x u|_{\R^n}$ is given.
Note that both BVPs mean that ``one half'' of the function $f_{0}$ is prescribed.
This is natural for a first order elliptic equation.

On the other hand, the set of all traces $f_0= f|_{\R^n}$ of solutions $f$ to the equation 
$\pd_t f+ DB f=0$ in $\R^{1+n}_+$, with appropriate estimates, is a subspace 
of $L_2=L_2(\R^n; \C^{(1+n)m})$ which we denote $E_A^+ \mH$.
Here the reader should have the classical situation in mind, where (\ref{eq:ODEroadmap}) 
is the 
Cauchy--Riemann equations and $E_A^+\mH$ is the upper Hardy subspace of $L_2$.
Just like in this classical situation, it turns out that for $t$-independent coefficients
and small perturbations thereof, $E_A^+\mH$ is a closed proper subspace,
being ``one half'' of $L_2$, and there is a bounded Hardy type projection $E_A^+$ onto
$E_A^+\mH$. 
Moreover, there is a Cauchy type reproducing formula for the solution $f$ to the ODE, given 
$f_0\in E_A^+\mH$. Hence there is a one-to-one correspondence between solutions $f$ and their traces
$f_0\in E_A^+\mH$.

Proving these facts for small perturbations of $t$-independent coefficients is the main work in this paper. 
For $t$-independent coefficients, this result is in \cite{AAM}.
Before explaining the proofs in more detail,
assume for a moment the stated properties of $E_A^+$ and $E_A^+\mH$, in order to explain
the implications for BVPs.
The unique solvability of the Neumann BVP means that for each boundary data
$\bphi\in L_2(\R^n;\C^{m})$ there is a unique $f\in E_A^+\mH$ such that $f_\no=\bphi$.
Hence well-posedness of the Neumann BVP is equivalent to 
$$
   E_A^+ \mH\to L_2(\R^n;\C^m): f\mapsto f_\no
$$
being an isomorphism. Similarly well-posedness of the regularity BVP is equivalent to 
$$
   E_A^+ \mH\to \sett{g\in L_2(\R^n;\C^{nm})}{\curl_x g=0}: f\mapsto f_\ta
$$
being an isomorphism.
Even for $t$-independent coefficients, these maps are not always invertible.
For counter examples   based on  \cite{KR} in this  context,  see \cite{AxNon}. However, we mentioned  three important classes of $t$-independent coefficients
where techniques are available to prove invertibility.

Concerning this approach to solving BVPs, it should be noted that the problem is divided
into two parts: (i) understanding representation formulas and the trace space $E_A^+\mH$ for 
solutions to the differential equation,
and (ii) understanding the relation between the full trace space $E_A^+ \mH$ and the boundary condition
(Neumann or regularity).
Carefully note  that (i) only  involves the differential equation and not the boundary conditions.
This is of great importance, since it means that it suffices to study the ODE (\ref{eq:ODEroadmap}), 
and explains why the trace spaces $E_A^+\mH$ may be 
well behaved even when the BVPs fail to be well-posed.
The main harmonic analysis work goes into proving that the projections $E_A^+$ are bounded
for all complex $t$-independent coefficients, and small $t$-dependent perturbations thereof. 
From this it is deduced that the projections $E_A^+$ depend continuously on $A$
in a certain Carleson sense, and 
as a consequence in Corollary~\ref{cor:neuregmain} 
that well-posedness of the Neumann and regularity problems is stable 
under small perturbations of the coefficients.

We next explain our methods for solving the ODE, \textit{i.e.}\,  solving problem (i). 
For this we study (\ref{eq:ODEroadmap}), where we first consider $t$-independent
coefficients $B= B_0$, and we write $E_A^+= E_0^+$.
In this case, we view $DB_0$ as an unbounded operator in $L_2(\R^n;\C^{(1+n)m})$,
and at a first glance the solution to (\ref{eq:ODEroadmap}) with initial/boundary data $f_0$ seems
to be $f_t= e^{-tDB_0} f_0$. 
However, the problem is that $DB_0$ is not a sectorial operator, but instead bisectorial, \textit{i.e.}\,  its spectrum 
is contained in a double sector around the real axis. 
This indefiniteness means that the operators $e^{-tDB_0}$ are not well defined on 
$L_2(\R^n;\C^{(1+n)m})$ for any $t\ne 0$.
Another technical problem is that $DB_0$ has an infinite
dimensional null space.
The fact is that there are topological splittings 
$$
  L_2= \mH\oplus \nul(DB_0)= \Big( E_0^+ \mH \oplus E_0^-\mH \Big) \oplus \nul(DB_0),
$$
where $\mH=\clos{\ran(DB_0)}= \clos{\ran(D)}$ is the closure of the range.
The splitting of $\mH$ into the spectral subspace $E_0^+\mH$ for the sector in the right half plane
and the spectral subspace $E_0^-\mH$ for the sector in the left half plane is a deep result, and builds on
the Kato square root problem as discussed in the introduction.
This proof also shows that $DB_0$ has square function estimates, 
which in particular shows that $-DB_0$ generates a bounded holomorphic semigroup in $E_0^+\mH$,
and that $DB_0$ generates a bounded holomorphic semigroup in $E_0^-\mH$.

Given any $f_0\in E_0^+\mH$, differentiation as seen in \cite{AAM} shows that the generalized 
Cauchy reproducing formula  
$$
f=C_{0}^+ f_{0},
$$
with $C_{0}^+f_{0}(t,x)= (e^{-t DB_0}E_0^+ f_0)(x)$, 
yields a solution to (\ref{eq:ODEroadmap}) with trace $f_0$. 
Conversely, given a solution $f$ to (\ref{eq:ODEroadmap}), it is a fact that $f_t$ belongs
to the range $\mH$ for  any conormal gradient $f$ and $t>0$.
(Note that it follows from (\ref{eq:ODEroadmap}) that $\pd_t f_t \in \mH$.)
We apply the projections $E_0^\pm$ and suitable exponentials to the equation, giving
$$
  \begin{cases}
     \pd_s(e^{-(t-s) DB_0} E_0^+ f_s) =0, & \qquad  s\in (0,t), \\
     \pd_s(e^{(s-t) DB_0} E_0^- f_s) =0, & \qquad s\in (t,\infty).
  \end{cases}
$$
Integration with limits $\lim_{s\to\infty} f_s=0$ and $\lim_{s\to 0} f_s= f_0$ indicates that 
the trace $f_0$ belongs to the subspace $E_0^+\mH$ and that the generalized 
Cauchy reproducing formula 
$
f=C_{0}^+ f_{0}$ holds. This converse and existence of the limits are   shown in Corollary \ref{eq:WPcompatNeu}. 
For further details concerning the relation between $C_0^+$ and the classical Cauchy integral,
see \cite[Ex. 1.2]{AAH} and \cite[Thm. 2.1]{AxNon}.
Note that $C_0^+ f_0=0$ for all $f_0\in E_0^- \mH$.

For  $t$-dependent perturbations $B$ of a given $t$-independent coefficient matrix
$B_0$, we write the ODE as $\pd_t f + DB_0 f_t= D\E_t f_t$, with $\E:= B_0-B$.
The above argument now gives 
$$
  \begin{cases}
     \pd_s(e^{-(t-s) DB_0} E_0^+ f_s) = e^{-(t-s)DB_0}E_0^+ D\E_s f_s, &  \qquad s\in (0,t), \\
     \pd_s(e^{(s-t) DB_0} E_0^- f_s) = e^{(s-t)DB_0}E_0^- D\E_s f_s, & \qquad s\in (t,\infty),
  \end{cases}
$$
and integration and subtraction of the equations give the integral equation
\begin{equation}   \label{eq:inteqroadmap}
  f= C_{0}^+ h^+ + S_Af, 
\end{equation}
for some $h^+ \in E_{0}^+\mH$ 
(which will be shown to be $h^+= E_0^+ f_0$)
and where $S_A$ is the integral operator given by
\begin{equation}   \label{eq:firstformalSAdefn}
    S_A f_t :=  \int_0^t e^{-(t-s)DB_0} E_0^+ D \E_s f_s ds - \int_t^\infty e^{(s-t)DB_0} E_0^- D\E_s f_s ds,  \qquad t>0.
\end{equation}
To construct solutions to (\ref{eq:inteqroadmap}), it is therefore natural to think of the Picard fixed point theorem. For this, we need an appropriate function space of functions in $\R^{1+n}_{+}$ that contains the free evolution 
$g=C_{0}^+ h^+$ and on which $S_A$ is bounded.  
By \cite{AAM}, the non-tangential maximal function in Definition \ref{defn:NTandC} of $g$ belongs to $L_{2}(\R^n)$. Thus the space $\mX$ defined by $\|\tN(f)\|_{2}<\infty$ is a natural candidate (see Section~\ref{sec:functionspace}).

We turn to a closer look at $S_{A}$,
and that allows us to justify the equivalence between (\ref{eq:inteqroadmap}) and the ODE.
This operator involves forward and backward maximal regularity operators. Usual treatment of maximal regularity is not sufficient for our needs.
A convenient  way  to study $S_A$, in the spirit of functional calculus, is as follows.
Define, for $\lambda\in\C$ with $\re\lambda>0$, the bounded integral operator 
$F(\lambda): L_2(\R_+;\C)\to L_2(\R_+;\C): u=(u_{s})_{s>0} \mapsto F(\lambda)u$ by 
$$
  F(\lambda)u_t:= \int_0^t \lambda e^{-(t-s)\lambda} u_s ds.
$$
By letting $F(\lambda)$ act pointwise in $x\in\R^n$, it defines a bounded operator on 
$L_2(\R_+; \mH)$.
In this space we also have the operator 
$$
 |DB_0|f:= 
  \begin{cases}
     DB_0f, & f\in E_0^+\mH, \\
     -DB_0 f, & f\in E_0^-\mH,
  \end{cases}
$$
by letting it act pointwise in $t\in\R_+$.
The operator $|DB_0|$ is a sectorial operator in $\mH$, hence in $L_2(\R_+; \mH)$,
and on the sector containing its spectrum, $\lambda\mapsto F(\lambda)$
defines an operator-valued holomorphic function.
Similar to Dunford functional calculus, we can apply $(F(\lambda))_\lambda$ to $|DB_0|$,
since $|DB_0|$ commutes with each $F(\lambda)$, and we obtain an operator
$F(|DB_0|)$ on $L_2(\R_+;\mH)$.
An advantage of this method, referred to as {\em operational calculus}, is that boundedness
of the singular integral $f=(f_{s})_{s>0} \mapsto   F(|DB_0|)f$ defined by for all $t>0$,
$$
F(|DB_0|)f_t:= \int_0^t |DB_0|e^{-(t-s)|DB_0|}f_s ds, 
$$
can be easily derived from the square function estimates for $DB_0$, in exactly the same
way as the boundedness of $E_0^\pm$, $e^{-t|DB_0|}$, or more general operators in the 
functional calculus of $DB_0$ is proved.

The upshot of all this is a representation 
\begin{equation}    \label{eq:inteqopcalcroadmap}
   S_A= F(|DB_0|) \hE_0^+\E + F^*(|DB_0|) \hE_0^- \E,
\end{equation}
where $\hE_0^\pm$ are bounded operators such that $E_0^\pm D= (DB_0)\hE_0^\pm$
and $F^*(\lambda)u_{t}:= \int_t^\infty \lambda e^{-(s-t)\lambda} u_s ds$.
Since all operators on the right hand side in (\ref{eq:inteqopcalcroadmap}) 
are bounded on $L_2(\R_+;\mH)$, so is $S_A$.
More generally, this boundedness proof goes through on weighted spaces
$L_2(\R_+, t^\alpha dt;\mH)$ for $|\alpha|<1$. Details are in Section~\ref{sec:holoopcalc}.
However, for $\alpha=-1$, which is the natural scaling for  BVPs with $L_2$ data, boundedness
of $F^*(|DB_0|)$ fails.  Recall nonetheless that the free evolution belongs to $\mX$, 
which contains $L_2(\R_+,  dt/t;L_2)$ as a subspace.
(As compared to the space $\mX$,  $L_2(\R_+, dt/t ;L_2)$ consists of functions that must 
vanish in a certain sense at $t=0$.)  
Thus we can replace $L_2(\R_+, dt/t ;L_2)$ by the larger space $\mX$. To ensure boundedness of $S_A$ on  $\mX$, we still need additional control on $\E$ such as finiteness of $\|\E\|_{C}$ of Definition \ref{defn:NTandC} (in fact a possibly smaller quantity suffices) and our fundamental estimate is
\begin{equation}   
  \|\tN(S_A f)\|_{2} \lesssim \| A-A_0 \|_C \|\tN(f)\|_2,  \label{eq:Carlesoncontrol} 
  \end{equation}
obtained  from the chain
$$
  \mX\stackrel\E\longrightarrow L_2(\R_+, dt/t ;L_2) \stackrel{\hE_0^-}\longrightarrow
  L_2(\R_+, dt/t ;\mH) \stackrel{F^*(|DB_0|)}\longrightarrow \mX.
$$
See Lemma~\ref{lem:Carleson} for this modified Carleson
embedding theorem in the first arrow and 
Theorem~\ref{thm:modelendpointbdd} for remaining details. This allows us to prove that 
the trace of $f$ at $t=0$ exists in a certain sense and one sees that $E_{0}^+f_{0}=h^+$.  Details of this representation of solutions $f$ are in Theorem~\ref{thm:inteqforNeu}.

Now,  smallness of this Carleson norm  implies smallness of $\|S_A\|_{\mX\to \mX}$,
in which case  (\ref{eq:inteqroadmap}) rewrites as a Cauchy reproducing formula
$$
  f= (I-S_A)^{-1} C_{0}^+  h^+,
$$
by inverting $I-S_A$ in $\mX$.
Conversely, for any $h^+\in E_{0}^+\mH$, the Picard iteration scheme applies to produce a solution
$f= (I-S_A)^{-1} C_{0}^+ h^+$ to (\ref{eq:inteqroadmap}) whose trace $f_{0}$ is determined by a  linear relation  $E_A^+ h^+= f_0$.

This provides us with the Hardy type projections $E_A^+$ needed to solve the Neumann and 
regularity problems as described above.
Details are in Corollary~\ref{cor:neuregmain}.
Additional \textit{a priori} square function estimates on solutions, under further regularity assumption on the coefficients, can be shown. They are proved in Section~\ref{sec:sqfcnests}.

\subsection{The Dirichlet problem} \label{sec:roadDir}

For the Dirichlet problem with $L_2$ boundary data, it is not obvious that the above first
order approach applies. Nevertheless it is possible to adapt the arguments and we describe 
this now. 
Instead of (\ref{eq:ODEroadmap}) for the conormal gradient $f$, we want to work with the
potential $u$ at the $L_2$ level.
The  heuristic is that $u$ solves the divergence form equation if and only if there
is a vector-valued potential $v$ such that
\begin{equation}  \label{eq:BDeqroadmap}
   \pd_t v + BDv=0
\end{equation}
and $u= -v_\no$.
(The minus sign is just for convenience.)
On one hand, we have seen that $u$ is a divergence form equation
if and only if $(\pd_t +DB)(\nabla_A u)=0$.
On the other hand, applying $D$ to the ODE (\ref{eq:BDeqroadmap})
gives $(\pd_t +DB)(Dv)=0$.
Thus the claim amounts to rewrite the conormal gradient $f= \nabla_A u$ as
$f=Dv$. (Note that the latter equation is pointwise in $t$.)
Evaluating the tangential part $f_\ta$ shows the relation $u= -v_\no$.
Once again the reader should compare this to the classical situation of 
Cauchy--Riemann's equations,
in which case $u$ would be a harmonic function, and $v$ the analytic function having $u$ and
its harmonic conjugate function as imaginary and real parts. 

We mention that the $v$ we construct via the conormal gradient again does not quite satisfy \eqref{eq:BDeqroadmap} because of the null space of $D$, but that is not a problem as we only need its normal part which does not depend on this null space.  Having this heuristic in mind we proceed as before in two steps: (i) understanding representation of solutions, and (ii)  understanding its trace and the relation to solvability.  We mostly concentrate on (i) as (ii) will follow right away.

For $t$-independent coefficients $B=B_0$, we note that $B_0D$ is another bisectorial
operator, just like $DB_0$. It is not injective (the null space is $\mH^\perp$) and
the spectral projections $\tE_0^\pm$ for the sectors in the right and left half planes
split the range $B_0\mH$ into two closed subspaces.
Similar to the argument for $DB_0$, we have that solutions obeying a certain square function estimate to 
the ODE (\ref{eq:BDeqroadmap}) all are of the form 
$v= \widetilde C_{0}^+ v_{0} + c$ for a unique $v_{0}\in \tE_0^+ L_2$ and some $c\in \C^{(1+n)m}$, 
where $\widetilde C_{0}^+$ is defined by $ (\widetilde C_{0}^+ v_0)(t,x)=(e^{-t B_0D}\tE_0^+ v_0)(x)$.
Thus, for $t$-independent coefficients, we have the representation
$$
  u= c  -(\widetilde C_{0}^+ v_{0})_\no, \qquad v_0\in \tE_0^+ L_2, c\in \C^m,
$$
for solutions $u$ to \eqref{eq:divform} obeying a square function estimate. This  is in Corollary \ref{eq:WPcompatDir} and improves \cite{AAM} where this was shown for a smaller class of solutions.  Note that the conormal gradient of $u$ can be calculated as $f=D\widetilde C_{0}^+v_{0}$.

For  $t$-dependent perturbations $B$ of a given $t$-independent coefficient matrix
$B_0$, suppose we are given a solution $u$. Since we  do not know $v$ yet, we go via $f= \nabla_{A}u$ (which in the end will be $Dv$) to solve (\ref{eq:BDeqroadmap}).
Since $f$ satisfies $(\pd_t +DB)f=0$,
as in Section~\ref{sec:roadNeu}, we look for  a functional setting in which we obtain an equation of the form
$
f_t = e^{-tDB_{0}} h^+ + S_A f_t
$
for some $h^+$  in a positive spectral space
so as to again apply the Picard fixed point theorem to construct solutions. The main difference here is in the free evolution term $g_{t}= e^{-tDB_{0}} h^+$. Indeed, $f_{0}$, which we should relate to $h^+$, may only be defined in a space of Sobolev type with regularity index -1 (and in fact, we do not care about $f_{0}$). 

Square function estimates   (see \cite{AAM}) lead us to the solution. Indeed, we have 
$\int_0^\infty \|e^{-tDB_{0}} h^+\|_2^2 t dt\approx \|D^{-1} h^+\|_2^2$ provided  $D^{-1} h^+\in  L_2$.  Here $D^{-1}$ is defined as a closed operator as we assumed $h^+$ in a positive spectral space.  Hence  the natural (considering the method) space of conormal gradients  for Dirichlet problems is $\mY:= L_2(\R_+, t dt ;L_2)$. 
Indeed, imposing the free evolution $g$ to be in   $ \mY$ allows us to  obtain  $g=D\widetilde C_{0}^+\tilde h^+$ for  $\tilde h^+\in \tE_0^+ L_2$ determined by $\tilde h^+=D^{-1}h^+$. (Note  that $D$ and $B_{0}$ are interchanged.) This is consistent with the $t$-independent case.

The next step  is thus to bound $S_A$ in $\mY$. We use again the operational calculus representation \eqref{eq:inteqopcalcroadmap}. In the space $\mY$,   it is now the operator $F(|DB_0|)$  that fails to be bounded,
but again the additional Carleson control ensures boundedness of $\|S_A\|_{\mY\to\mY}$  and our second fundamental estimate reads
 \begin{equation}   
   \| S_A f \|_{L_2(tdt; L_2)}  \lesssim  \| A-A_0 \|_C \| f \|_{L_2(tdt;L_2)}. \label{eq:Carlesoncontrol2}
\end{equation}
Details are in Proposition~\ref{prop:TAboundsonXY}. This allows us to justify the formal manipulations and to obtain an \textit{a priori} representation of conormal gradients $f=\nabla_{A}u$ in $\mY$ of  solutions $u$  by
\begin{equation}   \label{eq:inteqroadmapDir}   
f  =  D\widetilde C_{0}^+ \tilde h^+ + S_{A}f , \qquad \tilde h^+\in \tE_0^+ L_2.
 \end{equation}

We next want to exhibit the announced vector-valued potential $v$ which must satisfy $f=Dv$. Remark that $D$ being non-injective, there is some freedom in the choice but we basically want to factor out $D$ in \eqref{eq:inteqroadmapDir}.
This is granted for the free evolution and for $S_{A}$ we obtain starting from 
\eqref{eq:firstformalSAdefn}
that $S_A f= D\tS_A f$, where 
\begin{equation}    \label{eq:tildeSA}
  \tS_A f_t:=  \int_0^t e^{-(t-s)B_0D} \tE_0^+ \E_s f_s ds - \int_t^\infty e^{(s-t)B_0D} \tE_0^- \E_s f_s ds.
\end{equation}
All this follows from the intertwining $b(DB_0)D= Db(B_0D)$ of the functional
calculi of $DB_0$ and $B_0D$.
Thus we can set
$$
   v  := \widetilde C_{0}^+ \tilde h^+ + \tS_A f, 
   $$
   and then show  that there exists a constant $c\in \C^m$ such that
   $$
   u =c  -(v)_\no,
$$
for solutions $u$ to \eqref{eq:divform} obeying an initial square function estimate with $f$ being the conormal gradient of $u$.
Again, this is not the only possible choice for $v$ but any other choice  has  identical normal part  and this is what we need
to recover $u$. Moreover,  this choice has good estimates.
Details of this representation of solutions $u$ to the divergence form equation
are in Theorems~\ref{thm:inteqforDir} and \ref{thm:Ysols}. We note that this proves that all solutions
with gradient satisfying a square function estimates are, up to constants, continuous in $t$ with values in $L_{2}(\R^n, \C^m)$.
This representation also enables us to show existence of the trace $v_0=\lim_{t\to 0} v_t$ in 
$L_2(\R^n;\C^{(1+n)m})$ and $\tilde h^+$ is determined by $\tilde h^+= \tE_{0}^+ v_{0}$.  Most importantly, the representation allows us to prove non-tangential maximal function estimates in Section~\ref{sec:maxest}. We remark they are \textit{a priori} estimates. 

Upon smallness of the Carleson control, these representations above rewrite 
\begin{align*}
   f & = (I-S_{A})^{-1} D\widetilde C_{0}^+ \tilde h^+,  \qquad \tilde h^+\in \tE_0^+ L_2, \\
   v & = \widetilde C_{0}^+ \tilde h^+ + \tS_A f, \\
   u & =c  -(v)_\no,
\end{align*}
and, conversely, this can serve (via Picard's iteration scheme to obtain the first equation)  as an ansatz to produce a solution $u$ starting from a given $\tilde h^+ \in \tE_0^+ L_2$ and constant $c=0$, and $v_{0} \in L_{2}$ is determined by a linear relation $v_{0}=\tE_{A}^+ \tilde h^+$.  With this ansatz, 
 well-posedness of the Dirichlet problem is seen to be equivalent to 
$$
  \tE_0^+ L_2\to L_2(\R^n;\C^m): \tilde h^+\mapsto u_0= -(v_0)_\no
$$
being an isomorphism.
This allows us to prove stability of well-posedness of the Dirichlet 
problem under small perturbations in Corollary~\ref{cor:dirmain}.

%
%
%
%
%
\section{Integration of the differential equation}    \label{sec:difftoint}

Following \cite{AAM}, we construct solutions $u$ to the divergence form system (\ref{eq:divform}), 
by replacing $u$ by its gradient $g$ as the unknown function.
Consequently (\ref{eq:divform}) for $u$ is replaced by (\ref{eq:firstorderdiv}) below for $g$.
Proposition~\ref{prop:divformasODE} reformulates this first order system (\ref{eq:firstorderdiv})
further, by solving for the $t$-derivatives, as the vector-valued ODE (\ref{eq:firstorderODE}) for the conormal gradient 
$$
  f=\nabla_A u= [ \pd_{\nu_A}u, \nabla_x u ]^t, \qquad \text{where }
  [ \alpha, v]^t := \begin{bmatrix} \alpha \\ v \end{bmatrix}
$$ 
for $\alpha\in\C^m$ and $v\in \C^{nm}$,
and
$\pd_{\nu_A}u:= (A\nabla_{t,x}u)_\perp$ denotes the (inward!) conormal derivative of $u$.

According to the decomposition of $m$-tuples into normal and tangential parts as introduced in
Section~\ref{sec:statement}, we split the matrix as
$$
  A(t,x)= \begin{bmatrix} A_{\no\no}(t,x) & A_{\no\ta}(t,x) \\ A_{\ta\no}(t,x) & A_{\ta\ta}(t,x) \end{bmatrix}.
$$
Note that with our assumption that $A$ be strictly accretive on $\mH$ for a.e. $t>0$, 
the matrix $A_{\no\no}$ is invertible.

\begin{prop}  \label{prop:divformasODE}
  The pointwise transformation 
$$
   A\mapsto \hat A:=    \begin{bmatrix} A_{\no\no}^{-1} & -A_{\no\no}^{-1} A_{\no\ta}  \\ 
     A_{\ta\no}A_{\no\no}^{-1} & A_{\ta\ta}-A_{\ta\no}A_{\no\no}^{-1}A_{\no\ta} \end{bmatrix}
$$
is a self-inverse bijective transformation of the set of bounded matrices which are strictly accretive on $\mH$.

For a pair of coefficient matrices $A= \hat B$ and $B= \hat A$, 
the pointwise map $g\mapsto f=[ (Ag)_\no, g_\ta ]^t$ gives 
a one-one correspondence, with inverse $g= [ (B f)_\no, f_\ta ]^t$,
between solutions $g\in L_2^\loc(\R_+;L_2(\R^n;\C^{(1+n)m}))$ to the equations
\begin{equation}  \label{eq:firstorderdiv}
     \begin{cases}
     \divv_{t,x} (Ag)=0, \\
     \curl_{t,x} g=0
     \end{cases}
\end{equation}
and solutions $f\in L_2^\loc(\R_+;\mH)$ to the generalized Cauchy--Riemann equations
\begin{equation}  \label{eq:firstorderODE}
  \pd_t f+ DB f=0,
\end{equation}
where the derivatives are taken in $\R^{1+n}_+$ distributional sense, and
$
  D:= 
    \begin{bmatrix} 0 & \divv_x  \\ 
     -\nabla_x & 0 \end{bmatrix}.
$
\end{prop}

This was proved in \cite[Section 3]{AAM} for $t$-independent coefficients.
The proof goes through without changes for $t$-dependent coefficients, 
but for completeness we give the proof of this important result.
Note that $\clos{\ran(D)}=\mH$.

\begin{proof} We first look at the correspondence $A\mapsto \hat A=B$. Fix an arbitrary $t>0$ and we write $A$ for $A(t,\cdot)$. 
  From the accretivity of $A$ on $\mH$, it follows that the component $A_{\no\no}$ is pointwise strictly accretive,
  hence invertible, and therefore so is 
  $\oA:= \begin{bmatrix} A_{\no\no} & A_{\no\ta} \\ 0 & I \end{bmatrix}$. Thus, multiplication by $\oA$ is an isomorphism on $\mH$ and,  letting $\uA:= \begin{bmatrix} I & 0 \\ A_{\ta\no} & A_{\ta\ta} \end{bmatrix}$, 
   $B=  \uA\oA^{-1}$ is bounded if $A$
  is so.
  We calculate, for any fixed $g\in \mH$ and $f=\oA g$,
\begin{multline*}
  \re(Bf,f)= \re(\hat A\oA g,\oA g)= 
  \re\left( \begin{bmatrix} I & 0 \\ A_{\ta\no} & A_{\ta\ta} \end{bmatrix}  \begin{bmatrix} g_\no \\ g_\ta \end{bmatrix},
  \begin{bmatrix} A_{\no\no} & A_{\no\ta} \\ 0 & I \end{bmatrix} \begin{bmatrix} g_\no \\ g_\ta \end{bmatrix} \right) \\
  = 
    \re\left( \begin{bmatrix} A_{\no\no} & A_{\no\ta} \\ A_{\ta\no} & A_{\ta\ta} \end{bmatrix}  \begin{bmatrix} g_\no \\ g_\ta \end{bmatrix}, \begin{bmatrix} g_\no \\ g_\ta \end{bmatrix} \right) = \re(Ag,g). 
\end{multline*}
  This shows that $B=\hat A$ is strictly accretive if $A$ is so.
  That $\hat{\hat A}=A$ is straightforward to verify, and this shows that $A$ and $\hat A$ are in one-to-one correspondence.

  Next consider a pair of functions $g$ and $f$ in $L_2^\loc(\R_+;L_2(\R^n;\C^{(1+n)m}))$ such that $f=\oA g$.
  Equations (\ref{eq:firstorderdiv}) for $g$ are equivalent to
\begin{equation}
\begin{cases}
  \pd_t (Ag)_\no + \divv_x( A_{\ta\no} g_\no + A_{\ta\ta}g_\ta) =0, \\
  \pd_t g_\ta - \nabla_x g_\no =0, \\
  \curl_x g_\ta =0.
\end{cases}
\end{equation}
  The last equation is equivalent to $f_t\in \mH$.
  Moreover, using that $(Ag)_\no= f_\no$, $g_\ta= f_\ta$ and
  $g_\no= (Bf)_\no= A_{\no\no}^{-1}(f_\no- A_{\no\ta}f_\ta)$, 
  the first two equations are seen to be equivalent to the equation $\pd_t f+ DB f=0$.
  This proves the proposition.
\end{proof}

\begin{rem}
In terms of the second order equation where $g$ is a gradient and $f$ is the corresponding conormal gradient, the identity $\re(Bf,f)= \re (Ag,g)$ rewrites $ \re(B\nabla_{A} u,\nabla_{A} u)= \re (A\nabla_{t,x} u ,\nabla_{t,x} u)$ for any appropriate $u$ (not necessarily a solution).
\end{rem}

We now want to construct solutions to \eqref{eq:firstorderODE}. 
Let us first recall the situation when $B(t,x)= B_0(x)$ does not depend on the $t$-variable.
In this case, we view $B_0$ as a multiplication operator in the boundary function space
$L_2(\R^n; \C^{(1+n)m})$.
Define closed and open sectors and double sectors in the complex plane by
\begin{alignat*}{2}
    S_{\omega+} &:= \sett{\lambda\in\C}{|\arg \lambda|\le\omega}\cup\{0\}, 
    & \qquad
    S_{\omega} &:= S_{\omega+}\cup(-S_{\omega+}), \\    
    S_{\nu+}^o &:= \sett{\lambda\in\C}{ \lambda\ne 0, \, |\arg \lambda|<\nu},  
    & \qquad  
    S_{\nu}^o &:= S_{\nu+}^o\cup(- S_{\nu+}^o),
\end{alignat*}
and define the {\em angle of accretivity} of $B_0$ to be
$$
   \omega:= \sup_{f\not = 0,f\in\mH} |\arg(B_0f,f)|  <\pi/2.
$$
The method for constructing solutions to the elliptic divergence form system, developed 
in \cite{AAH, AAM}, uses holomorphic functional calculus of the {\em infinitesimal generator}
$DB_0$ appearing in the ODE (\ref{eq:firstorderODE}), and the following was proved.
\begin{itemize}
\item[{\rm (i)}]
The operator $DB_0$ is a closed and densely defined $\omega$-bisectorial operator, 
\textit{i.e.}\, ~$\sigma(DB_0)\subset S_\omega$, where $\omega$ is the angle of accretivity of $B_0$.
Moreover, there are resolvent bounds 
$\|(\lambda- DB_0)^{-1}\| \lesssim 1/ \dist(\lambda, S_\omega)$ when $\lambda\notin S_\omega$.
\item[{\rm (ii)}]
The function space splits topologically as
$$
   L_2(\R^n; \C^{(1+n)m}) = \mH \oplus \nul(DB_0),
$$
and the restriction of $DB_0$ to $\mH=\clos{\ran(D)}$ 
is a closed, densely defined and injective operator with dense range in $\mH$, with same estimates on spectrum and resolvents as in (i).
\item[{\rm (iii)}]
The operator $DB_0$ has a bounded holomorphic functional calculus in $\mH$, 
\textit{i.e.}\, ~for each bounded holomorphic function $b(\lambda)$ on a double sector 
$S_\nu^o$, $\omega<\nu<\pi/2$, the operator $b(DB_0)$ in $\mH$ is bounded with estimates
$$
   \|b(DB_0)\|_{\mH \rightarrow \mH} \lesssim \| b\|_{L_\infty(S_\nu^o)}.
$$
\end{itemize}
For background material on sectorial operators (which is straightforward to adapt to 
bi-sectorial operators) and their holomorphic functional calculi, see \cite{ADMc}.
The construction of the operators $b(DB_0)$ is explained in detail in Section~\ref{sec:abstropcalc},
in the more general case of operational calculus.
The two most important functions $b(\lambda)$ here are the following.
\begin{itemize}
\item
The characteristic functions
$\chi^+(\lambda)$ and $\chi^-(\lambda)$ for the right and left half planes,
which give the generalized {\em Hardy projections} $E_0^\pm:=\chi^\pm(DB_0)$.
\item
The exponential functions $e^{-t|\lambda|}$, $t>0$, which give the operators $e^{-t|DB_0|}$.
Here $|\lambda|:= \lambda\sgn(\lambda)$ and $\sgn(\lambda):= \chi^+(\lambda)-\chi^-(\lambda)$.
\end{itemize}
A key result that we make use of frequently, is that the boundedness of the projections $E_0^\pm$
shows that there is a topological splitting
\begin{equation}     \label{eq:hardysplit}
   \mH= E_0^+ \mH \oplus E_0^-\mH
\end{equation}
of $\mH=\clos{\ran(D)}= \clos{\ran(DB_0)}$ into complementary closed subspaces $E_0^\pm\mH:= \ran(E_0^\pm)$.

We also recall the definition of the generalized Cauchy extension $C_0^+$ from 
Section~\ref{sec:roadmap}.

\begin{prop}\label{prop:Cauchyextension}
The generalized Cauchy extension
$$
f_t=(C_{0}^+f_{0})(t,\cdot):= e^{-t|DB_0|}E_0^+f_0
$$ 
of $f_{0}\in E_0^+ \mH$ gives a solution
  to $\pd_t f+ DB_0 f=0$, in the strong sense $f\in C^1(\R_+; L_2)\cap C^0(\R_+; \dom(DB_0))$,
  with $L_2$ bounds $\sup_{t>0}\|f_t\|_2\approx \|f_0\|_2$ and $L_2$
  limits $\lim_{t\to 0} f_t =f_0$ and $\lim_{t\to\infty} f_t =0$.
\end{prop}

Now consider more general $t$-dependent coefficients $B(t,x)$.
Fix some $t$-independent coefficients $B_0$, strictly accretive on $\mH$.
(This $B_0$ should be thought of as the boundary trace of $B$, acting in $\R^{1+n}_+$ independently of $t$.)
To construct solutions to the ODE, we rewrite it as
\begin{equation}   \label{eq:pertODE}
  \pd_t f + D B_0 f = D \E f,\qquad \text{where } \E_t:= B_0-B_t.
\end{equation}
However, while $\pd_t f + D B_0 f =0$ can be interpreted in the strong sense,
(\ref{eq:pertODE}) will be understood in the sense of distributions.
The following proposition rewrites this equation in integral form.
It uses operators $\hE_0^\pm$, defined as 
\begin{equation}   \label{eq:hatEdefn}
\hE_0^\pm := E_0^\pm B_0^{-1} P_{B_0\mH},
\end{equation}
where $P_{B_0\mH}$ denotes the projection onto $B_0\mH$ in the topological splitting $L_2= B_0 \mH\oplus \mH^\perp$
and $B_0^{-1}$ is the inverse of $B_0: \mH\to B_0\mH$.
Beware that $B_0^{-1}$ is not necessarily a multiplication operator and is only defined on 
the subspace $B_0\mH$. Note also that unlike $E_0^\pm$, $\hE_0^\pm$ are not projections.

\begin{prop}   \label{prop:preinteqs}
  If $f\in L_2^\loc(\R_+; \mH)$ satisfies $\pd_t f+ DB f=0$ in $\R^{1+n}_+$ distributional sense,
  then 
\begin{align*}
   -\int_0^t \eta_+'(s) e^{-(t-s)|DB_0|} E_0^+ f_s ds = \int_0^t \eta_+(s) DB_0 e^{-(t-s)|DB_0|} \hE_0^+ \E_s f_s ds, \\
   -\int_t^\infty \eta_-'(s) e^{-(s-t)|DB_0|} E_0^- f_s ds = \int_t^\infty \eta_-(s) DB_0 e^{-(s-t)|DB_0|} \hE_0^- \E_s f_s ds, 
\end{align*}
for all $t>0$ and smooth bump functions $\eta_\pm(s)\ge 0$, where $\eta_+$ is compactly supported in $(0,t)$,
and $\eta_-$ is compactly supported in $(t,\infty)$.
\end{prop}

\begin{proof}
In Section~\ref{sec:roadNeu}, we showed formally how to integrate the differential equation and 
arrived at (\ref{eq:inteqroadmap}).
To make this rigorously, we proceed as follows.
By assumption
\begin{equation}    \label{eq:tstfcnforinteq}
  \int_0^\infty \Big( (-\pd_s \phi_s, f_s)+ (D\phi_s, B_0 f_s)\Big) ds = \int_0^\infty(D\phi_s, \E_s f_s)ds,
\end{equation}
for all $\phi\in C^\infty_0(\R^{1+n}_+; \C^{(1+n)m})$.
To prove the identity on $(0,t)$, let $\phi_0\in \mH$ be any boundary function and define
$\phi_s:= \eta_+(s) (e^{-(t-s)|DB_0|}E_0^+)^* \phi_0\in C_0^\infty(\R_+; \dom(D))$. 
To show that we can use this as test function, take $\eta\in C_0^\infty(\R^n)$ with $\eta=1$ 
in a neighbourhood of $x=0$ and $\int_{\R^n}\eta=1$
and write $\eta_\epsilon:= \epsilon^{-n} \eta(x/\epsilon)$ and
$$
   \phi^{R,\epsilon}_s(x):= \eta_+(s)\eta(x/R)
   \left( \eta_\epsilon *  ((e^{-|(t-s)DB_0|}E_0^\pm)^* \phi_0) \right)(x).
$$
It is straightforward to verify that $\pd_s \phi^{R,\epsilon}\to \pd_s \phi$
and $D \phi^{R,\epsilon}\to D \phi$ in $L_2(\supp\eta_+ \times \R^n;\C^{(1+n)m})$ when 
$R\to\infty$, $\epsilon\to 0$.

From (\ref{eq:tstfcnforinteq}) we obtain
\begin{multline*}
   \int_0^t (-\eta_+'(s) (e^{-(t-s)|DB_0|}E_0^+)^* \phi_0 - \eta_+(s) (DB_0 e^{-(t-s)|DB_0|}E_0^+)^* \phi_0, f_s )ds \\
  + \int_0^t ( \eta_+(s) D(e^{-(t-s)|DB_0|}E_0^+)^* \phi_0, B_0 f_s )ds \\
  = \int_0^t  ( \eta_+(s) D(e^{-(t-s)|DB_0|}E_0^+)^* \phi_0, \E_s f_s ) ds.
\end{multline*}
Since $B_0^* D(e^{-(t-s)|DB_0|}E_0^+)^* = (e^{-(t-s)|DB_0|}E_0^+DB_0)^*= (DB_0e^{-(t-s)|DB_0|}E_0^+)^*$,
the last two terms on the left hand side cancel.
Using that
$E_0^+ D= E_0^+ D P_{B_0\mH}= E_0^+ (DB_0)B_0^{-1} P_{B_0\mH}= DB_0 \hE_0^+$ 
on the right hand side, we have proved that 
$$
  -\left( \phi_0, \int_0^t \eta'_+(s) e^{-(t-s)|DB_0|} E_0^+ f_s ds\right)=
  \left( \phi_0, \int_0^t \eta_+(s) e^{-(t-s)|DB_0|} E_0^+ D \E_s f_s ds \right).
$$
Since this holds for all $\phi_0$, the $(0,t)$ integral formula follows.
The proof for the $(t,\infty)$ integral formula is similar.
\end{proof}

Our goal is to take limits to arrive at an integrated equation. 
Formally, if we let $\eta_\pm$ approximate the characteristic functions for $(0,t)$ and $(t,\infty)$ respectively,
we obtain in the limit from Proposition~\ref{prop:preinteqs} that
\begin{align*}
   E_0^+ f_t- e^{-t |DB_0|} E_0^+ f_0 = \int_0^t DB_0 e^{-(t-s)|DB_0|} \hE_0^+ \E_s f_s ds, \\
   0- E_0^- f_t = \int_t^\infty DB_0 e^{-(s-t)|DB_0|} \hE_0^- \E_s f_s ds, 
\end{align*}
if $\lim_{t\to 0} f_t =f_0$ and $\lim_{t\to\infty} f_t=0$ in appropriate sense (and yet to be proved).
Subtraction yields $f= C_{0}^+ f_0 + S_A f$, which we wish to solve as
\begin{equation}  \label{eq:tdepCauchy}
  f = (I- S_A)^{-1} C_0^+ f_0,
\end{equation}
where the integral operator $S_A$ is
\begin{equation}    \label{eq:TAdefn}
  S_A f_t =  \int_0^t DB_0 e^{-(t-s)|DB_0|} \hE_0^+ \E_s f_s ds - \int_t^\infty DB_0 e^{-(s-t)|DB_0|} \hE_0^- \E_s f_s ds.
\end{equation}
and  $C_{0}^+$ is the generalized Cauchy integral defined via the semigroup $e^{-t|DB_0|}$.  

The equation (\ref{eq:tdepCauchy}) can also be viewed as a generalized Cauchy integral formula, for
$t$-dependent coefficients $A$,
and we shall see that, given any $f_0\in L_2(\R^n; \C^{(1+n)m})$, 
it constructs a solution $f_t$ to the elliptic equation.
However, for this one  and also justification of the limiting arguments, one needs a  suitable functional setting we now introduce.

%
%
%
%
%
\section{Natural function spaces}    \label{sec:functionspace}

It is well known that solutions $g$ to (\ref{eq:firstorderdiv}) with $L_2$ boundary data typically
satisfy certain square function estimates, as well as non-tangential maximal function estimates.
In this section, we study the basic properties of some natural function spaces related to BVPs
with $L_2$ boundary data.

\begin{defn}    \label{defn:XY}
   In $\R^{1+n}_+$, define the Banach/Hilbert spaces 
\begin{align*}
  \mX & := \sett{ f: \R^{1+n}_+\to \C^{(1+n)m} }{ \tN(f)\in L_2(\R^n) }, \\
  \mY &:= \sett{ f: \R^{1+n}_+\to \C^{(1+n)m} }{\int_0^\infty \| f_t \|^2_{2} tdt < \infty}, 
\end{align*} 
with the obvious norms. Here $\tN$ denotes the modified non-tangential maximal function
from Definition~\ref{defn:NTandC}.
By $\mY^*= L_2(\R^{1+n}_+, dt/t;\C^{(1+n)m})$ we denote the dual space of 
$\mY$, relative to $L_2(\R^{1+n}_+;\C^{(1+n)m})$.
\end{defn}

In Sections~\ref{sec:neu} and \ref{sec:dirichlet} we demonstrate that 
the maximal function space $\mX$ is the natural space to solve
the Neumann and regularity problems in, whereas $\mY$ is natural for the Dirichlet problem.
Natural is meant with respect to the method. 
That the spaces $\mY$ and $\mX$ are relevant for $L_2$ BVPs 
with $t$-independent coefficients is clear from the following theorem. 
For proofs, we refer to \cite[Proposition 2.3]{AAM} and \cite[Proposition 2.56]{AAH}.

\begin{thm}    \label{thm:QEandNT}
  Let $f_0$ belong to the spectral subspace $E_0^+ \mH$. Then its generalized Cauchy extension $f=C_{0}^+f_{0}$ as in Proposition \ref{prop:Cauchyextension}
    has estimates
$$
 \| \pd_t f \|_{\mY}\approx \|f\|_\mX\approx \|f_0\|_2.
$$
\end{thm}
 We will show in Corollary~\ref{eq:WPcompatNeu} that any distributional solution
 $f\in\mX$ to $\pd_t f+ DB_0 f=0$ is  the generalized Cauchy extension of  some $f_0\in E_0^+ \mH$. 

Clearly $\mY\subset L_2^\loc(\R_+; L_2)$.
The following lemma shows that $\mX$ is locally $L_2$ inside $\R^{1+n}_+$ as well,
and is quite close to $\mY^*$.

\begin{lem}      \label{lem:XlocL2}
  There are estimates
$$
  \sup_{t>0} \frac 1t\int_t^{2t} \| f_s \|_2^2 ds \lesssim \| \tN(f) \|^2_2 \lesssim \int_0^\infty \| f_s \|^2_2 \frac {ds}s.
$$
In particular $\mY^*\subset \mX$.
\end{lem}
\begin{proof}
  The second inequality follows by integrating the pointwise estimate
$$
  \tN(f)(x)^2 \approx \sup_{t>0} \iint_{W(t,x)} |f(s,y)|^2\frac{dsdy}{s^{1+n}}  
  \le  \iint_{|y-x|< c_0c_1 s} |f(s,y)|^2\frac{dsdy}{s^{1+n}}.
$$
For the lower bound on $\|\tN(f)\|_2$, it suffices to estimate $t^{-1}\int_t^{c_0 t} \|f_s\|_2^2 ds$, uniformly for $t>0$.
To this end, split $\R^n=\bigcup_k Q_k$, where $Q_k$ all are disjoint cubes with diagonal lengths $c_1 t$.
Then 
$$
  t^{-1} \int_t^{c_0 t} \int_{Q_k} |f(s,y)|^2 dsdy\lesssim |Q_k| \inf_{x\in Q_k}|\tN(f)(x)|^2\lesssim \int_{Q_k}|\tN(f)(x)|^2 dx.
$$
Summation over $k$ gives the stated estimate.
\end{proof}

The space $\mY^*$ is a subspace of $\mX$ of functions with zero trace at the boundary $\R^n$,
in the square $L_2$-Dini sense $\lim_{t\to 0} t^{-1} \int_t^{2t} \|f_s\|^2_2 ds=0$. A fundamental quantity is the norm  of multiplication operators mapping $\mX$ into $\mY^*$.

\begin{defn} \label{defn:starnorm}  For functions $\E: \R^{1+n}_+\to \mL(\C^{(1+n)m})$, we denote $$\|\E\|_* :=  \|\E\|_{\mX\to \mY^*}= \sup_{\|f\|_\mX=1}\|\E f\|_{\mY^*}$$  the norm of pointwise multiplication by $\E$. 
\end{defn}

The following lemma gives a sufficient Carleson condition for a multiplication operator to map into this subspace.

\begin{lem}   \label{lem:Carleson}
  For functions $\E: \R^{1+n}_+\to \mL(\C^{(1+n)m})$, we have estimates
$$
  \|\E\|_\infty \lesssim \|\E\|_* \lesssim \|\E\|_C, 
$$
where  $ \|\E\|_C$ denotes the modified Carleson norm from Definition~\ref{defn:NTandC}.
\end{lem}

\begin{proof}
  For the first estimate, fix $t$ and consider only $f$ supported on $(t,2t)$ in the definition of $\|\E\|_{\mX\to \mY^*}$.
  Lemma~\ref{lem:XlocL2} shows that
$$
   \sup \|\E f\|_{\mY^*}/ \|f\|_\mX \approx \sup (t^{-1/2} \|\E f\|_2) / ( t^{-1/2}\|f\|_2)= \essup_{t<s<2t}\|\E_s\|_\infty,
$$
where the first two suprema are over all $0\ne f\in L_2((t,2t)\times \R^n;\C^{(1+n)m})$.
Taking supremum over $t$ shows the estimate $\|\E\|_\infty \lesssim \|\E\|_*$.

For the second estimate, we calculate
\begin{multline*}
   \|\E f\|^2_{\mY^*}\approx \iint_{\R^{1+n}_+} \left( \frac 1{t^{1+n}}\iint_{W(t,x)} dsdy\right) |\E(t,x)f(t,x)|^2\frac {dtdx}t \\
    \approx \iint_{\R^{1+n}_+} \left( \frac 1{s^{1+n}}\iint_{W(s,y)} |\E(t,x)f(t,x)|^2\frac {dtdx}t \right) dsdy \\
    \lesssim \iint_{\R^{1+n}_+} \left( \frac 1s \sup_{W(s,y)} |\E|^2  \right) \left( \frac 1{s^{1+n}}\iint_{W(s,y)} |f(t,x)|^2 dtdx \right) dsdy 
    \lesssim \|\E\|^2_C \|f\|^2_\mX,
\end{multline*}
where the final estimate is by Carleson's theorem.
\end{proof}

We have not been able to prove that the $\|\cdot\|_*$ norm is equivalent to the modified Carleson norm, that is to prove the appropriate lower bound. It is however easy to see that the $\|\cdot\|_*$ norm dominates
the standard Carleson norm $\|\cdot\|_c$. 
Indeed, choosing $f$ as the characteristic function for the Carleson box $(0,l(Q))\times Q$ (times 
a unit vector field) in 
the estimate $\|\E f\|_{\mY^*}\le \|\E\|_* \|f\|_\mX$, shows that
$$
  \|\E\|_c:= \sup_Q \frac 1{|Q|} \iint_{(0, l(Q))\times Q} |\E(t,x)|^2 \frac {dtdx}{t} \lesssim \|\E\|_*^2.
$$
Furthermore, it is straightforward to see that the modified Carleson norm is dominated by the corresponding modified square Dini norm
$$
  \|\E\|_C^2\lesssim \int_0^\infty \sup_{c_0^{-1}t <s< c_0t} \|\E_s\|_\infty^2 \frac {dt}t.
$$

\begin{proof}[Proof of Lemma~\ref{lem:carlesontrace}]
We shall prove the lemma assuming only  finiteness of the standard 
Carleson norm (hence the lemma also holds for the star norm). Thus we assume
$\|A-A_0\|_c<\infty$ and do the rest of the proof replacing $\|\cdot \|_C$ by $\|\cdot \|_c$. 

To prove uniqueness, we use 
$\|A_0'-A_0\|_c\le \|A-A'_0\|_c+ \|A-A_0\|_c<\infty$ to obtain
$$
 \int_0^{l(Q)}  \left( \frac 1{|Q|}\int_Q |A'_0(x)-A_0(x)|^2 dx \right) \frac {dt}t <\infty
$$
for all cubes $Q\subset \R^n$,
which only is possible if $A'_0=A_0$ almost everywhere, by $t$-independence.

Similarly for $A-A_0$, we have
$\int_0^{l(Q)}( |Q|^{-1}\int_Q |A(t,x)-A_0(x)|^2 dx) dt/t <\infty$
for any cube $Q$, and it follows that we have essential infimum 
$$
  \essinf_{0<t<l(Q)} \frac 1{|Q|}\int_Q |A(t,x)-A_0(x)|^2 dx =0.
$$
To prove $\|A_0\|_\infty\le \|A\|_\infty$, assume $\epsilon>0$ and pick
$Q$ such that 
$(|Q|^{-1}\int_Q |A_0|^2 dx)^{1/2}\ge \|A_0\|_\infty-\epsilon$.
Then choose $t\in (0,l(Q))$
such that $|Q|^{-1}\int_Q |A(t,x)-A_0(x)|^2 dx\le \epsilon^2$.
We assume that $t$ is a Lebesgue point of 
$t\mapsto |Q|^{-1}\int_Q |A(t,x)-A_0(x)|^2 dx$ and of 
$t\mapsto |Q|^{-1}\int_Q |A(t,x)|^2 dx$.
This yields
$\|A_0\|_\infty-\epsilon \le (|Q|^{-1}\int_Q |A_0|^2 dx)^{1/2}\le \|A\|_\infty + \epsilon$.
Letting $\epsilon\to 0$ proves the claim.

To prove $\kappa_0\ge \kappa$, assume $\epsilon>0$ and pick $f\in\mH\cap C^\infty_0(\R^n;\C^{(1+n)m})$.
Pick $Q$ such that $\supp f\subset Q$. Then
\begin{multline*}
  \re (A_0f,f) = \re(A_t f,f)+ \re((A_0-A_t)f,f)\ge \\
  \kappa \|f\|_2^2-|Q| \|f\|_\infty^2 
  \left(|Q|^{-1}\int_Q |A(t,x)-A_0(x)|^2 dx\right)^{1/2},
\end{multline*}
for all $t\in (0,l(Q))$. Taking the essential supremum over such $t$ gives
$\re (A_0f,f)\ge \kappa \|f\|_2^2$.
Since $\mH\cap C^\infty_0(\R^n;\C^{(1+n)m})$ is dense in $\mH$,
taking infimum over $f$, this proves the claim.
\end{proof}

%
%
%
%
%
\section{Holomorphic operational calculus}     \label{sec:holoopcalc}

Throughout this section $\Lambda$ denotes a closed, densely defined 
$\omega$-sectorial operator in an arbitrary Hilbert space $\mH$, \textit{i.e.}\,  $\sigma(\Lambda)\subset S_{\omega+}$,
and we assume resolvent bounds $\| (\lambda- \Lambda)^{-1} \|_{\mH\to \mH} \lesssim 1/\dist(\lambda, S_{\omega+})$.
For simplicity, we assume throughout that $\Lambda$ is injective, and therefore has dense range.
In our applications $\Lambda$ will be $|DB_0|$, and $\mH$ will be the Hilbert space from (\ref{eq:Hspace}).  See Section~\ref{sec:saest}.

The goal in this section is to develop the theory needed to make rigorous the limiting argument following 
Proposition~\ref{prop:preinteqs}. 
To this end, we study uniform boundedness and convergence of model operators
\begin{align}
  S^+_\epsilon f_t := \int_0^t \eta_\epsilon^+(t,s) \Lambda e^{-(t-s)\Lambda} f_s ds, \label{eq:Tplus} \\
  S^-_\epsilon f_t := \int_t^\infty \eta_\epsilon^-(t,s) \Lambda e^{-(s-t)\Lambda} f_s ds,  \label{eq:Tminus}
\end{align}
acting on functions $f_t(x)= f(t,x)$ in a Hilbert space $L_2(\R_+,d\mu(t); \mH)$.
For uniform boundedness issues, it suffices that
the bump functions $\eta_\epsilon^+(t,s)$ and $\eta_\epsilon^-(t,s)$ are uniformly bounded and compactly supported 
within $\sett{(s,t)}{0<s<t}$ and $\sett{(s,t)}{0<t<s}$ respectively.
For convergence issues and to link to the ODE, they should approximate the characteristic functions
of the above sets.
A convenient choice which we shall use systematically is the following.
Define $\eta^0(t)$ to be the piecewise linear continuous function with support
$[1,\infty)$, which equals $1$ on $(2,\infty)$ and is linear on $(1,2)$.
Then let $\eta_\epsilon(t):= \eta^0(t/\epsilon)(1- \eta^0(2\epsilon t))$ and
$$
  \eta_\epsilon^\pm(t,s):= \eta^0(\pm (t-s)/\epsilon) \eta_\epsilon(t)\eta_\epsilon(s).
$$

We study the operators $S^\pm_\epsilon$ from the point of view of operational calculus. This means for example that
we view $S^+_\epsilon= F(\Lambda)$ as obtained from the underlying operator $\Lambda$ 
(acting horizontally, \textit{i.e.}\,  in the variable $x$) by applying the operator-valued function $\lambda \mapsto F(\lambda)$, where
$$
  (F(\lambda)f)_t:=  \int_0^t \eta_\epsilon^+(t,s) \lambda e^{-(t-s)\lambda} f_s ds,
$$
which depends holomorphically on $\lambda$ in a sector $S_{\nu+}^o$ containing the spectrum of $\Lambda$.
Note that each of these vertically acting, \textit{i.e.}\,  acting in the $t$-variable, operators $F(\lambda)$ commute with $\Lambda$.

\subsection{Operational calculus in Hilbert space}   \label{sec:abstropcalc}

Consider $\Lambda$ as above. Let  $\mK:=L_2(\R_+, d\mu(t);\mH)$
for some Borel measure $\mu$. 
We extend the resolvents $(\lambda-\Lambda)^{-1}\in\mL(\mH)$, $\lambda\notin S_{\omega+}$, to
bounded operators on $\mK$
(and we use the same notation, letting
$((\lambda-\Lambda)^{-1}f)_t := (\lambda-\Lambda)^{-1}(f_t)$ for all $f\in \mK$ and a.e. $t>0$).
These extensions of the resolvents to $\mK$ clearly inherit the bounds from $\mH$.
We may think of them as being the resolvents of an $\omega$-sectorial operator
$\Lambda= \Lambda_\mK$, although this extended unbounded operator $\Lambda_\mK$
is not needed below.

Define the commutant of $\Lambda$ to be
$$
  \Lambda'  := \sett{ T\in \mL(\mK) }{ (\lambda- \Lambda)^{-1} T= T(\lambda- \Lambda)^{-1} \text{ for } \lambda\notin S_{\omega+}}.
$$
Fix $\omega<\nu< \pi/2$, and consider classes of operator-valued holomorphic functions
\begin{align*}
  H(S_{\nu+}^o; \Lambda' ) &:= \{ \text{holomorphic } F: S^o_{\nu+}\to \Lambda'  \}, \\
  \Psi(S^o_{\nu+}; \Lambda' ) &:= \sett{F\in H(S^o_{\nu+}; \Lambda' ) }{\|F(\lambda)\| \lesssim \min(|\lambda|^a, |\lambda|^{-a}),
  \text{some }  a>0}, \\
  H_\infty(S^o_{\nu+}; \Lambda' ) &:= \sett{F\in H(S^o_{\nu+}; \Lambda' )}{\sup_{\lambda\in S^o_{\nu+}}\|F(\lambda)\|<\infty  }.
\end{align*}
Through Dunford calculus, we define for $F\in \Psi(S^o_{\nu+}; \Lambda' )$ the operator
\begin{equation}   \label{eq:Dunford}
  F(\Lambda):= \frac 1{2\pi i}\int_{\gamma} F(\lambda) (\lambda- \Lambda)^{-1} d\lambda,
\end{equation}
where $\gamma$ is the unbounded contour $\sett{r e^{\pm i\theta}}{r>0}$, $\omega<\theta<\nu$,
parametrized counter clockwise around $S_{\omega+}$.
This yields a bounded operator $F(\Lambda)$, since the bounds on $F$ and the resolvents guarantee
that the integral converges absolutely.

\begin{rem}
Functional calculus of the operator $\Lambda$ is a special case of this operational calculus (\ref{eq:Dunford}).
Applying a scalar holomorphic function $f(\lambda)$ to $\Lambda$ with functional calculus is the same
as applying the operator-valued holomorphic function $F(\lambda)= f(\lambda)I$ to $\Lambda$ with operational
calculus.
For the functional calculus, we write $\Psi(S^o_{\nu+})$ and $H_\infty(S^o_{\nu+})$ for the corresponding 
classes of scalar symbol functions.

We also remark that a more general functional and operational calculi for bisectorial operators like $DB_0$
are developed entirely similar to those of sectorial operators $\Lambda$, 
replacing the sector $S_{\omega+}$ by the bisector $S_\omega$.
\end{rem}
The following three propositions contain all the theory of operational calculus that we need.
To be self-contained and illustrate their simplicity, we give full proofs, although the
propositions are proved in exactly the same way as for functional calculus,
and can be found in \cite{ADMc}.

\begin{prop}    \label{prop:opcalchomom}
  If $F, G\in \Psi(S^o_{\nu+}; \Lambda' )$, then
$$
  F(\Lambda) G(\Lambda)= (FG)(\Lambda).
$$
\end{prop}

Note that we need not assume that $F(\lambda)$ and $G(\mu)$ commute for any $\lambda$, $\mu\in S^o_{\nu+}$.

\begin{proof}
We use contours $\gamma_1$ and $\gamma_2$, with angles $\omega<\theta_1<\theta_2<\pi/2$,
so that $\gamma_2$ encircles $\gamma_1$.
Cauchy's theorem now yields
\begin{multline*}
  (2\pi i)^2 F(\Lambda) G(\Lambda)= \left( \int_{\gamma_1} \frac{F(\lambda)}{\lambda-\Lambda} d\lambda \right)  \left( \int_{\gamma_2} \frac{G(\mu)}{\mu-\Lambda} d\mu \right) \\
   =\int_{\gamma_1} \int_{\gamma_2} F(\lambda)G(\mu) \frac 1{\mu-\lambda} \left( \frac 1{\lambda-\Lambda}- \frac 1{\mu-\Lambda} \right) d\lambda d\mu \\
   =\int_{\gamma_1} \frac{F(\lambda)}{\lambda-\Lambda}\left( \int_{\gamma_2} \frac{G(\mu)}{\mu-\lambda} d\mu \right) d\lambda
   -  \int_{\gamma_2} \left( \int_{\gamma_1} \frac{F(\lambda)}{\mu-\lambda} d\lambda\right)   \frac{G(\mu)}{\mu-\Lambda} d\mu \\
   = \int_{\gamma_1} \frac{F(\lambda)}{\lambda-\Lambda} 2\pi i G(\lambda) d\lambda-  0= (2\pi i)^2 (FG)(\Lambda),
\end{multline*}
using the resolvent equation.
\end{proof}

\begin{prop}  \label{prop:opcalcboundedness}
  Assume that $\Lambda$ satisfies square function estimates, \textit{i.e.}\,  assume that 
$$
  \int_0^\infty \| \psi(t\Lambda) u \|_\mH^2 \frac{dt}t \approx \| u \|_\mH^2, \qquad\text{for all } u\in\mH
$$
and some fixed $\psi\in\Psi(S^o_{\nu+})$.
Then there exists $C<\infty$ such that
$$
   \| F(\Lambda) \| \le C \sup_{\lambda\in S^o_{\nu+}} \| F(\lambda) \|,\qquad \text{for all } F\in \Psi(S^o_{\nu+}; \Lambda' ).
$$
\end{prop}

We remark that if square function estimates for $\Lambda$ hold with one such $\psi$, then they hold for any non-zero 
$\psi\in\Psi(S^o_{\nu+})$.

\begin{proof}
Note that the square function estimates extend to $u\in \mK$, with $\|\cdot\|_\mK$ instead of $\|\cdot\|_\mH$.
We drop $\mK$ in $\|\cdot\|_\mK$.
Using the resolution of identity $\int_0^\infty \psi^2(s\Lambda)u ds/s = cu$, where $0<c<\infty$ is a constant, and
the square function estimates, we calculate
\begin{multline*}
  \|F(\Lambda)u \|^2 \approx \int_0^\infty \| \psi(t\Lambda) F(\Lambda) u \|^2 \frac{dt}t \\
  \approx \int_0^\infty \left\| \int_0^\infty (\psi(t\Lambda) F(\Lambda) \psi(s\Lambda)) (\psi(s\Lambda)u) \frac{ds}s \right\|^2 \frac{dt}t  \\
   \lesssim \sup_{S^o_{\nu+}} \|F(\lambda)\|^2 
   \int_0^\infty \left( \int_0^\infty \eta(t/s)\frac{ds}s \right) \left( \int_0^\infty \eta(t/s) \|\psi(s\Lambda)u\|^2 \frac{ds}s \right) \frac{dt}t \\
   \lesssim \sup_{S^o_{\nu+}} \|F(\lambda)\|^2 \int_0^\infty \|\psi(s\Lambda)u\|^2 \frac{ds}s 
    \lesssim \sup_{S^o_{\nu+}} \|F(\lambda)\|^2 \|u\|^2.
\end{multline*}
We have used the estimate 
$$
\| \psi(t\Lambda) F(\Lambda) \psi(s\Lambda) \|\lesssim \int_\gamma \|F(\lambda)\| |\psi(t\lambda)\psi(s\lambda) \lambda^{-1}d\lambda|\lesssim 
 \sup_{\lambda\in S^o_{\nu+}} \|F(\lambda)\| \eta(t/s),
$$
where $\eta(x):= \min \{x^a, x^{-a}\} ( 1+ |\log x |)$ for some $a>0$.

\end{proof}

\begin{prop}    \label{prop:opcalcconv}
  Assume that $\Lambda$ satisfies square function estimates as in Proposition~\ref{prop:opcalcboundedness}.
  Let $F_n\in\Psi(S^o_{\nu+};\Lambda' )$, $n=1,2,\ldots$, satisfy $\sup_{n,\lambda}\|F_n(\lambda)\|<\infty$, 
  and let $F\in H_\infty(S^o_{\nu+};\Lambda' )$.
  Assume that for each fixed $v\in\mK$ and $\lambda\in S^o_{\nu+}$, we have strong convergence
  $\lim_{n\to\infty}\|F_n(\lambda)v- F(\lambda)v\|=0$.
  Then the operators $F_n(\Lambda)$ converge strongly to a bounded operator $F(\Lambda)$,
  \textit{i.e.}\, 
$$
  F_n(\Lambda)u\rightarrow F(\Lambda)u,\qquad \text{for all } u\in \mK,\text{ as } n\rightarrow \infty.
$$
\end{prop}

\begin{proof}
  Since $\sup_n\|F_n(\Lambda)\|<\infty$ by Proposition~\ref{prop:opcalcboundedness},
  it suffices to consider $u= \psi(\Lambda)v$ for some fixed $\psi\in \Psi(S^o_{\nu+})\setminus\{0\}$, since $\ran(\psi(\Lambda))$ is dense in $\mK$.
  From (\ref {eq:Dunford}), we get
$$
  \| F_n(\Lambda)u- F_m(\Lambda)u \| \lesssim \int_\gamma \| (F_n(\lambda)-F_m(\lambda)) v \| \, |\psi(\lambda)\lambda^{-1} d\lambda|, 
$$
where  $\| (F_n(\lambda)-F_m(\lambda)) v \|\lesssim \|v\|$ and $|\psi(\lambda)|/|\lambda|$ is integrable.
The dominated convergence theorem applies and proves the proposition.
\end{proof}

Propositions~\ref{prop:opcalchomom}, \ref{prop:opcalcboundedness} and 
\ref{prop:opcalcconv} show that we have a continuous Banach algebra homomorphism
$$
  H_\infty(S^o_{\nu+};\Lambda' )\rightarrow \mL(\mK): F\mapsto F(\Lambda),
$$
provided that $\Lambda$ satisfies square function estimates as in Proposition~\ref{prop:opcalcboundedness}.
This is the operational calculus that we need.
Note that with $F(\Lambda)$ defined in this way for all $F\in H_\infty(S^o_{\nu+};\Lambda' )$,
Proposition~\ref{prop:opcalcconv} continues to hold for any $F_n\in H_\infty(S^o_{\nu+};\Lambda' )$.

\subsection{Maximal regularity estimates}   

  Here,
  we apply the operational calculus from Section~\ref{sec:abstropcalc} to prove weighted bounds on the operators 
   $S_\epsilon^\pm$ from (\ref{eq:Tplus}) and (\ref{eq:Tminus}).

\begin{thm}    \label{thm:weakmaxreg}
   The operators $S^+_\epsilon$ are uniformly bounded and converge
   strongly as $\epsilon\to 0$ on the weighted space 
   $L_2(t^{\alpha}dt; \mH)$ if $\alpha<1$.
    The operators $S^-_\epsilon$ are uniformly bounded and converge 
   strongly as $\epsilon\to 0$ on the weighted space $L_2(t^{\alpha}dt; \mH)$ if $\alpha>-1$.
\end{thm}

Note that the case $\alpha=0$ is the usual maximal regularity result in $L_2(dt;\mH)$.
The methods here provide a proof of it.

To establish boundedness of the integral operators $F(\lambda)$, we rely on the following
version of Schur's lemma.
The proof is straightforward using Cauchy--Schwarz' inequality.
\begin{lem}   \label{lem:Schur}
  Consider the integral operator $f_t\mapsto \int_0^\infty k(t,s) f_s ds$, with $\C$-valued kernel $k(t,s)$. 
  If the kernel has the bounds
$$
  \sup_t \frac 1{t^{\beta_2-\alpha}}\int_0^\infty |k(t,s)| s^{\beta_1} ds = C_1<\infty,\qquad 
  \sup_s \frac 1{s^{\beta_1+\alpha}}\int_0^\infty |k(t,s)| t^{\beta_2} dt = C_2<\infty,
$$
for some $\beta_1,\beta_2\in \R$,
then the integral operator is bounded on $L_2(t^\alpha dt;\mH)$ 
with norm at most $\sqrt{C_1C_2}$.
\end{lem}

The second result that we need shows that when the integral operators $F(\lambda)$ define
a holomorphic function in $\Psi(S^o_{\nu+};\mL(\mK))$, then the resulting operator $F(\Lambda)$
can be represented as an integral operator with operator-valued kernel.

\begin{lem}    \label{lem:psiasintop}
  Consider a family of integral operators $F(\lambda)f_t= \int_0^\infty k_\lambda(t,s) f_s ds$ such that 
  the $\C$-valued kernels have the bounds
$$
  \sup_t\frac 1{t^{\beta_2-\alpha}}\int_0^\infty |k_\lambda(t,s)| s^{\beta_1} ds \le \eta(\lambda) ,\qquad 
   \sup_s \frac 1{s^{\beta_1+\alpha}}\int_0^\infty |k_\lambda(t,s)| t^{\beta_2} dt   \le \eta(\lambda).
$$
If $\sup_{\lambda\in S^o_{\nu+}} \eta(\lambda)<\infty$,
if $\lambda\mapsto k_\lambda(t,s)$ is holomorphic in $S^o_{\nu+}$ for a.e. $(t,s)$, 
and if $\iint_K |\pd_\lambda k_\lambda(t,s)|dtds$
is locally bounded in $\lambda$ for each compact set $K$, then 
$F\in H_\infty(S^o_{\nu+};\mL(L_2(t^\alpha dt;\mH)))$.

If furthermore 
$\eta(\lambda)\lesssim \min(|\lambda|^a, |\lambda|^{-a})$ for $\lambda\in S^o_{\nu+}$ and some $a>0$, then 
$F\in \Psi(S^o_{\nu+};\mL(L_2(t^\alpha dt;\mH)))$, and
$$
  F(\Lambda) f_t = \int_0^\infty k_\Lambda(t,s) f_s ds,\qquad\text{for all } f\in L_2(t^\alpha dt;\mH) \text{ and a.e. } t,
$$
where the operator-valued kernels $k_\Lambda(t,s)$ are defined through (\ref{eq:Dunford}) for a.e. $(t,s)$.
\end{lem}

\begin{proof}
  Schur's lemma~\ref{lem:Schur} provides the bounds on $F(\lambda)$.
  To show that the operator-valued function $F$ is holomorphic, by local boundedness it suffices 
  to show that the scalar function
$$
  \lambda\mapsto \iint (h_t, k_\lambda(t,s) f_s) dsdt
$$
is holomorphic, for all bounded and compactly supported $f,h$.
The hypothesis on $\pd_\lambda k_\lambda(t,s)$ guarantees this.

To prove the representation formula for $F(\Lambda)$, it suffices to show that
for each $f\in L_2(t^\alpha dt; \mH)$, $v\in\mH$, and a.e. $t$, changing order 
of integration is possible in 
$$
  \iint (v, k_\lambda(t,s) (\lambda-\Lambda)^{-1} f_s) ds d\lambda.
$$
Thus, by Fubini, one needs to show
$$
  \iint |k_\lambda(t,s)| \|f_s\| ds \frac{|d\lambda|}{|\lambda|}<\infty,\qquad \text{for a.e. } t.
$$
The bounds on $k_\lambda(t,s)$ in the hypothesis guarantee this.
\end{proof}

\begin{proof}[Proof of Theorem~\ref{thm:weakmaxreg}]
  Since $S^+_\epsilon$ in $L_2(t^\alpha dt;\mH)$ and $S^-_\epsilon$ in $L_2(t^{-\alpha} dt;\mH)$,
  with $\Lambda$ replaced by $\Lambda^*$, are adjoint operators, it suffices to consider $S^+_\epsilon$.
  Let 
$$
  F_\epsilon(\lambda)f_t := \int_0^t \eta_\epsilon^+(t,s) \lambda e^{-(t-s)\lambda} f_s ds.
$$
  Uniform boundedness of the integral operators $F_\epsilon(\lambda)$ follows from
  Lemma~\ref{lem:Schur} with $\beta_1= -\alpha$, $\beta_2=0$, 
  using the estimate $\int_0^y e^x x^{-\alpha} dx\lesssim e^y y^{-\alpha}$, which holds
  if and only if $\alpha\in (-\infty,1)$.
  Indeed, since $\lambda\in S^o_{\nu+}$ with $\nu <\pi/2$, we have $\lambda_1:= \re\lambda\approx |\lambda|$ and
$$
  \int_0^t |\lambda e^{-\lambda(t-s)}|s^{-\alpha} ds\approx \int_0^t \lambda_1 e^{-\lambda_1(t-s)} s^{-\alpha}ds
  = \lambda_1^{\alpha} e^{-\lambda_1 t} \int_0^{\lambda_1 t} e^x x^{-\alpha} dx \lesssim t^{-\alpha},
$$
Similarly, $\int_s^\infty |\lambda e^{-\lambda(t-s)}| dt\lesssim e^{\lambda_1 s}\int_{\lambda_1 s}^\infty e^{-x}dx=1$.

  Again using Lemma~\ref{lem:Schur}, we note for fixed $\epsilon>0$ the crude estimate 
  $\|F_\epsilon(\lambda)\|\lesssim |\lambda|e^{-\epsilon\re\lambda}$, and with 
  Lemma~\ref{lem:psiasintop} we verify that $F_\epsilon\in \Psi(S^o_{\nu+};\mL(L_2(t^\alpha dt;\mH)))$, and
$$
  F_\epsilon(\Lambda) f_t = \int_0^t \eta_\epsilon^+(t,s) \Lambda e^{-(t-s)\Lambda}  f_s ds= S^+_\epsilon f_t,
  \qquad\text{for a.e. } t.
$$

  To prove strong convergence, by Proposition~\ref{prop:opcalcconv} it suffices to show 
  strong convergence of $F_\epsilon(\lambda)$ as $\epsilon\to 0$, for fixed $\lambda\in S^o_{\nu+}$.
  By uniform boundedness of $F_\epsilon(\lambda)$, it suffices to show
  that $F_\epsilon(\lambda)f$ converges in $L_2(t^\alpha dt;\mH)$ as $\epsilon\to 0$ for each
  $f$ in the dense set $\bigcup_{\delta>0} L_2((\delta,\delta^{-1}),t^\alpha dt;\mH)$.
  This will follow from norm convergence of $F_\epsilon(\lambda)$
  in $\mL(L_2((\delta,\delta^{-1}),t^\alpha dt;\mH), L_2(t^\alpha dt;\mH) )$ for each
  fixed $\delta>0$. 
  To see this, we use Lemma~\ref{lem:Schur} with $\beta_1= -\alpha$ and $\beta_2=0$.
  As above $C_1$ is uniformly bounded.   
  One verifies decay to $0$
  as $\epsilon\to 0$ of
$$
 \sup_{s\in (\delta,\delta^{-1})} \int_{(2\epsilon)^{-1}}^\infty \lambda_1 e^{-(t-s)\lambda_1} dt \qquad\text{and}\qquad
 \sup_{s\in (\delta,\delta^{-1})} \int_{s}^{s+2\epsilon} \lambda_1 e^{-(t-s)\lambda_1} dt.
$$
  This shows that $C_2\to 0$ as $\epsilon\to 0$, which proves the strong convergence and the theorem.
\end{proof}

\subsection{Endpoint cases}     \label{sec:endpoint}

The operators $S^-_\epsilon$ are not uniformly bounded on $L_2(t^\alpha dt; \mH)$
when $\alpha\le -1$, and therefore no limit operator $S^-$ exists in these spaces.
Indeed, if $\eta(t)$ is a smooth approximation to the Dirac delta at $t=1$ and $f\in \mH$,
then $S^-_\epsilon (\eta f)_t$ is independent of $\epsilon$ for $\epsilon<t/2$,
with non-zero value $\approx\Lambda e^{-\Lambda} f\in \mH$ for $t\approx 0$.
Thus  $\sup_{\epsilon>0}\int_0^\infty\|S^-_\epsilon(\eta f)_t\|_\mH^2 t^\alpha dt=\infty$ if $\alpha\le -1$.
By duality
$S^+_\epsilon$ cannot be uniformly bounded when $\alpha\ge 1$.

In this section we study the endpoint cases $\alpha=\pm 1$.
It is convenient here (and as we apply these abstract results in the subsequent paper \cite{AA2}) 
 to introduce  the abstract spaces  $Y:= L_2(tdt;\mH)$ and $Y^*:= L_2(dt/t;\mH)$.  Note that they differ  from $\mY$, $\mY^*$ by the target space $\mH$ being here an arbitrary Hilbert space. 
To obtain a uniform boundedness result for $S^-_\epsilon$,
assume there exists an auxiliary Banach space $X$ with continuous embeddings
\begin{equation}    \label{eq:abstrXembedd}
  Y^* \subset X\subset L_2^\loc(dt;\mH),
\end{equation}
\textit{i.e.}\,  $\int_a^b \|f_t\|_\mH^2 dt\lesssim \|f\|^2_X\lesssim\int_0^\infty \|f_t\|_\mH^2 dt/t$ hold
for each fixed $0<a<b<\infty$, and such that the map $u\mapsto (e^{-t\Lambda}u)_{t>0}$
is bounded $\mH\to X$, \textit{i.e.}\, 
\begin{equation}    \label{eq:semigroupXabstr}
  \|e^{-t\Lambda} u\|_X \lesssim \|u\|_\mH, \qquad\text{for all } u\in \mH.
\end{equation}

\begin{thm}    \label{thm:modelendpointbdd}
   Consider the model operators $S^+_\epsilon$ and $S^-_\epsilon$ from (\ref{eq:Tplus}-\ref{eq:Tminus})
   and $Y$, $Y^*$ and $X$ as above.
      
   The operators $S^+_\epsilon$ are uniformly bounded on $Y^*$ and 
    converge strongly to a limit operator $S^+\in \mL(Y^*,Y^*)$
    as $\epsilon\to 0$.
   
   The operators $S^-_\epsilon$ are uniformly bounded $Y^* \to X$, and there is a limit
   operator $S^-\in \mL(Y^*, X)$ such that 
   $\lim_{\epsilon\to 0} \|S^-_\epsilon f- S^- f\|_{L_2(a,b;\mH)}= 0$ 
   for any fixed $0<a<b<\infty$ and $f\in Y^*$.
 \end{thm}

For the proof, we shall need the first part of the following lemma. The second part will
be required in Propositions~\ref{prop:TAboundsonXY} and \ref{prop:endpointdir} below.

\begin{lem}    \label{lem:princpart}
  The operators 
$$
  \int_0^\infty \eta_\epsilon(s) \Lambda e^{-s\Lambda}  f_s ds : Y^*\to \mH
$$
are bounded, uniformly in $\epsilon$, and converge strongly as $\epsilon\to 0$.
Let $U_s: \mH\to \mH$ be bounded operators such that $ \mH\to Y^*; h\mapsto (U_s^* e^{-s\Lambda^*}h)_{s>0}$
is bounded.
Then the operators 
$$
  \int_0^\infty \eta_\epsilon(s) e^{-s\Lambda} U_s f_s ds : Y \to \mH
$$
are bounded, uniformly in $\epsilon$, and converge strongly as $\epsilon\to 0$.
\end{lem}

\begin{proof}
  For the first operator, square function estimates for $\Lambda^*$ give
$$
   \left\| \int_0^\infty \eta_\epsilon(s) \Lambda e^{-s\Lambda} f_s ds \right\|_\mH= \sup_{\|h\|_2=1}
  \left| \int_0^\infty  (s\Lambda^* e^{-s\Lambda^*} h, f_s) \eta_\epsilon(s)\frac{ds}s\right| \lesssim \|\eta_\epsilon f\|_{Y^*}
  \lesssim \|f\|_{Y^*}.
$$
For the second operator
\begin{multline*}
  \left\|   \int_0^\infty \eta_\epsilon(s)  e^{-s\Lambda} U_s f_s ds \right\|_\mH 
  \lesssim \sup_{\|h\|_2=1}\left| \int_0^\infty (U_s^* e^{-s \Lambda^*} h, f_s) \eta_\epsilon(s)ds\right| \\
  \lesssim \sup_{\|h\|_2=1} \| U_s^* e^{-s \Lambda^*} h\|_{Y^*} 
  \|\eta_\epsilon f\|_Y  
  \lesssim \|\eta_\epsilon f\|_Y  \lesssim \|f\|_Y,
\end{multline*}
where in the second last estimate the hypothesis is used.
(Note that the $\mH$-bound on $U_s$ is not used quantitatively.)

To see the strong convergence, replace $\eta_\epsilon$ by $\eta_\epsilon-\eta_{\epsilon'}$ and 
use the dominated convergence theorem.
\end{proof}
 
\begin{proof}[Proof of Theorem~\ref{thm:modelendpointbdd}]
  The result for $S^+_\epsilon$ is contained in Theorem~\ref{thm:weakmaxreg}, so it suffices to consider 
  $S^-_\epsilon$.
  Write
\begin{multline}   \label{eq:splitofbadtinftyterm}
  S^-_\epsilon f_t = \int_t^\infty \eta_\epsilon^-(t,s) \Lambda e^{-(s-t)\Lambda} f_s ds 
   =\int_t^\infty \eta_\epsilon^-(t,s) \Lambda ( e^{-(s-t)\Lambda}- e^{-(s+t)\Lambda}) f_s ds \\
   -  \int_0^{t+2\epsilon} (\eta_\epsilon(t)\eta_\epsilon(s)- \eta^-_\epsilon(t,s)) \Lambda e^{-(s+t)\Lambda} f_s ds \\
   + \eta_\epsilon(t)e^{-t \Lambda} \int_0^\infty \eta_\epsilon(s) \Lambda e^{-s\Lambda} f_s ds= I-II+III.
\end{multline}
We show that it is only the last term which is singular in the sense that it is not uniformly bounded on $Y^*$.
Consider the term $I$ and the symbol
$F^I_\epsilon(\lambda)u_t= \int_t^\infty \eta_\epsilon^-(t,s) k_\lambda(t,s) u_s ds$,
where $k_\lambda(t,s):= \lambda e^{-(s-t)\lambda}(1- e^{-2t\lambda})$.
Boundedness of $F^I_\epsilon(\lambda)$ on $Y^*$, uniformly in $\lambda$ and $\epsilon$
follows from Lemma~\ref{lem:Schur} and the estimates
$\int_t^\infty |k_\lambda(t,s)| sds\lesssim t$
and $\int_0^s |k_\lambda(t,s) | dt\lesssim 1$.
For example 
$$
  \int_t^\infty |k_\lambda(t,s)| sds \lesssim \min(1, t\lambda_1) e^{t\lambda_1} \int_t^\infty \lambda_1 e^{-s\lambda_1} sds
  = t \min(1,t\lambda_1) (1+ 1/(t\lambda_1))\lesssim t,
$$
with $\lambda_1:= \re\lambda$.
On the other hand, for fixed $\epsilon>0$, it is straightforward to verify with Lemma~\ref{lem:Schur}
that  $\|F^I_\epsilon(\lambda)\|_{Y^*\to Y^*}\lesssim |\lambda| e^{-\epsilon\re\lambda}$, and with 
  Lemma~\ref{lem:psiasintop} that $F^I_\epsilon\in \Psi(S^o_{\nu+};\mL(Y^*))$ and
$$
  F^I_\epsilon(\Lambda) f_t = 
  \int_t^\infty \eta_\epsilon^-(t,s)  \Lambda ( e^{-(s-t)\Lambda}- e^{-(s+t)\Lambda}) f_s ds,
  \qquad\text{for a.e. } t.
$$
  To prove strong convergence, as in the proof of Theorem~\ref{thm:weakmaxreg}, 
  by uniform boundedness it suffices to show norm
  convergence of $F^{I}_\epsilon(\lambda)$ in $\mL(L_2((\delta,\delta^{-1}),t^{-1} dt;\mH), Y^* )$ for each fixed
  $\delta>0$. This follows from Lemma~\ref{lem:Schur}, where one verifies decay to $0$
  as $\epsilon\to 0$ of
  $\sup_{s\in (\delta,\delta^{-1})} \int_0^{2\epsilon} |k_\lambda(t,s)| dt$ and
  $\sup_{s\in (\delta,\delta^{-1})} \int_{s-2\epsilon}^{s} |k_\lambda(t,s)| dt$,
  and hence $C_2\to 0$, for fixed $\lambda\in S^o_{\nu+}$.
  Together with the uniform bound $\sup_t t^{-1}\int_t^\infty |k_\lambda(t,s)| sds<\infty$,
  this proves the strong convergence for the term $I$.
  
  Consider next the term $II$ and the symbol
$$F^{II}_\epsilon(\lambda)u_t= 
\int_0^{t+2\epsilon} (\eta_\epsilon(t)\eta_\epsilon(s)- \eta^-_\epsilon(t,s)) \lambda e^{-(s+t)\lambda} u_s ds.
$$
Boundedness of $F^{II}_\epsilon(\lambda)$ on $Y^*$, uniformly in $\lambda$ and $\epsilon$
follows from Lemma~\ref{lem:Schur} and the estimates
$\int_0^{3t} |\lambda e^{-(s+t)\lambda} | sds\lesssim t$
and $\int_{s/3}^\infty |\lambda e^{-(s+t)\lambda} | dt\lesssim 1$.
On the other hand, for fixed $\epsilon>0$, we verify with Lemma~\ref{lem:Schur}
that  $\|F^{II}_\epsilon(\lambda)\|_{Y^*\to Y^*}\lesssim |\lambda| e^{-\epsilon\re\lambda}$, and with 
  Lemma~\ref{lem:psiasintop} that $F^{II}_\epsilon\in \Psi(S^o_{\nu+};\mL(Y^*))$ and
$$
  F^{II}_\epsilon(\Lambda) f_t = 
 \int_0^{t+2\epsilon} (\eta_\epsilon(t)\eta_\epsilon(s)- \eta^-_\epsilon(t,s)) \Lambda e^{-(s+t)\Lambda} u_s ds,
  \qquad\text{for a.e. } t.
$$
  With the same technique as for the term $I$, the strong convergence of the term $II$ follows from the decay to $0$
  as $\epsilon\to 0$ of
  $\sup_{s\in (\delta,\delta^{-1})} \int_{s-2\epsilon}^{s} |\lambda e^{-(s+t)\lambda}| dt$.

  It remains to estimate the principal term $III$. Since the variables $t$ and $s$ separate, we can factor this
  term through the boundary space $\mH$ as a composition $Y^*\to\mH\to X$, where Lemma~\ref{lem:princpart}
  and the assumed bounds $e^{-t\Lambda}:\mH\to X$ prove boundedness, uniform in $\epsilon$,
  as well as strong convergence as maps $Y^*\to  \mH\to L_2(a,b;\mH)$.
  This completes the proof.
\end{proof}

%
%
%
%
%
\section{Estimates of the integral operators $S_A$ and $\tilde S_A$}  \label{sec:saest}

Let us come back to our concrete situation. Consider the operator $DB_0$ from Section~\ref{sec:difftoint}.
We set $\Lambda=|DB_0|:= DB_0 \sgn(DB_0)$
on $\mH=\clos{\ran(D)}$, and see that $\Lambda$ satisfies the assumptions of Section~\ref{sec:holoopcalc}.
It is a closed, densely defined, injective operator with $\sigma(\Lambda)\subset S_{\omega+}$ and
$\| (\lambda- \Lambda)^{-1} \|_{\mH\to\mH} \lesssim 1/\dist(\lambda, S_{\omega+})$
(this follows from the resolvent bounds on $DB_0$).
We apply the abstract theory from Section~\ref{sec:endpoint} to this $\Lambda$
and spaces 
$Y^*:= \mY^*\cap L_2^\loc(\R_+;\mH)$, $X:= \mX\cap L_2^\loc(\R_+;\mH)$ and 
$Y:= \mY\cap L_2^\loc(\R_+;\mH)$.
Note that the continuous embeddings (\ref{eq:abstrXembedd}) follow from Lemma~\ref{lem:XlocL2}
and the boundedness hypothesis (\ref{eq:semigroupXabstr}) on $\mH\to X: h \mapsto (e^{-t|DB_0|} h)_{t>0}$ 
follows from Theorem~\ref{thm:QEandNT} (and the analogous result for the lower half space $\R^{1+n}_-$,
\textit{i.e.}\,  $f_0\in E_0^-\mH$ giving a solution of $\pd_t f+ DB_0 f=0$ for $t<0$.). We shall use the operational calculus of $\Lambda$ to rigorously define and estimate the operator  $S_A$ in (\ref{eq:TAdefn}).

The strategy for the Dirichlet problem described in Section \ref{sec:roadDir} leads us
to consider the functional calculus of $B_0D$ and
the integral operator $\tS_{A}$ from (\ref{eq:tildeSA}). 
If $B_0$ were invertible on all $L_2$, then $DB_0$ and $B_0 D$ would be similar operators, but
this is not the case in general.
Still, whenever $B_0$ is strictly accretive on $\mH$, it is true that $B_0D$ is an $\omega$-bisectorial 
operator with resolvent bounds. Furthermore,
the $L_{2}$ space splits as
$$
   L_2= B_0 \mH\oplus \mH^\perp
$$ 
(cf.~(\ref{eq:hatEdefn})) and $B_0D$ restricts to an injective closed operator with dense range in $B_0\mH$.
This operator has square function estimates, and therefore bounded functional and operational calculus in $B_0\mH$, as in Section~\ref{sec:abstropcalc}. 
For proofs and further details, see \cite{elAAM}. We set $\tilde\Lambda:= |B_0D|$ and $\tE_0^\pm:= \chi^\pm(B_0D)$. 
We extend an operator $b(B_0 D)$ in the functional calculus to an operator on all $L_2$ by letting $b(B_0D)=0$ on $\mH^\perp=\nul(B_0D)$. With this notation $\tE_0^\pm (B_0 \mH)= \tE_0^\pm L_2$, and we shall prefer the
latter to ease the notation.

A important relation between the functional calculus of $DB_0$ and $B_0D$ is
\begin{equation}    \label{eq:DBBDintertw}
  B_0 b(DB_0)= b(B_0D)B_0,
\end{equation}
where we also extend operators $b(DB_0)$ to all $L_2$, letting $b(DB_0)|_{\nul(DB_0)}:= 0$.
The equation (\ref{eq:DBBDintertw}) clearly holds for resolvents $b(z)= (\lambda-z)^{-1}$. The general
case follows from Dunford integration (\ref{eq:Dunford}) and taking strong limits as in
Proposition~\ref{prop:opcalcconv} (adapted to bisectorial operators).
Note that (\ref{eq:DBBDintertw}) in particular shows that for appropriate $b$ and $u$
$$
  b(DB_0)Du = Db(B_0 D)u.
$$

A final observation is that with $\Lambda= |DB_0|$ and $\tilde\Lambda= |B_0D|$, 
then $\tilde \Lambda^*= |DB_0^*|$ and $\Lambda^*= |B_0^*D|$. So $\Lambda$ and $\tilde \Lambda^*$ are of the same type, and the same holds for $\tilde\Lambda$ and $\Lambda^*$.

The boundedness result for the operator $S_{A}$ is as follows.

\begin{prop}      \label{prop:TAboundsonXY}
Assume that $\E: \R^{1+n}_+\to \mL(\C^{(1+n)m})$ satisfies $\|\E\|_*<\infty$, and define operators
$$
     S_A^\epsilon f_t :=  \int_0^t \eta_\epsilon^+(t,s)\Lambda e^{-(t-s)\Lambda} \hE_0^+ \E_s f_s ds +
      \int_t^\infty\eta_\epsilon^-(t,s)\Lambda e^{-(s-t)\Lambda} \hE_0^- \E_s f_s ds.
$$
Then $\|S_A^\epsilon\|_{\mX\to \mX}\lesssim \|\E\|_*$ and
$\|S_A^\epsilon\|_{\mY\to \mY}\lesssim \|\E\|_*$, uniformly for $\epsilon>0$.
In the space $\mX$ there is a limit operator $S_A= S_A^\mX\in \mL(\mX; \mX)$
such that 
$$
  \lim_{\epsilon\to 0}\|S_A^\epsilon f- S_A^\mX f\|_{L_2(a,b;L_2)}=0,\qquad \text{for any } f\in \mX, 0<a<b<\infty.
$$
In the space $\mY$, there is a limit operator $S_A= S_A^\mY\in \mL(\mY; \mY)$ such that
$$
\lim_{\epsilon\to 0}\|S_A^\epsilon f- S_A^\mY f\|_{\mY}= 0,\qquad\text{for any }f\in \mY.
$$
\end{prop}

\begin{proof}
The result on $\mX$ is a direct consequence of Theorem~\ref{thm:modelendpointbdd}, since
$\|\E f\|_{\mY^*}\le \|\E\|_* \|f\|_{\mX}$ and $S_A^\epsilon= S^+_\epsilon \hE_0^+ \E + S^-_\epsilon \hE_0^- \E$.
Note that $\ran(\hE^\pm_0)\subset \mH\subset L_2$.

Consider now the space $\mY$.
The second term $S^-_\epsilon \hE_0^- \E$ is bounded on $\mY$, uniformly in $\epsilon$,
  and converges strongly on $\mY$. This follows from Theorem~\ref{thm:weakmaxreg} and the boundedness 
  $\| \hE_0^-\E\|_{\mY\to\mY}\lesssim \|\E\|_\infty \lesssim \|\E\|_*<\infty$.
  The term $S^+_\epsilon \hE_0^+ \E$ we split as
\begin{multline*}
   \int_0^t \eta_\epsilon^+(t,s) \Lambda e^{-(t-s)\Lambda}  \hE_0^+ \E_s f_s ds 
   =\int_0^t \eta_\epsilon^+(t,s) \Lambda ( e^{-(t-s)\Lambda}- e^{-(t+s)\Lambda})  \hE_0^+\E_s f_s ds \\
   -  \int_{t-2\epsilon}^\infty (\eta_\epsilon(t)\eta_\epsilon(s)- \eta^+_\epsilon(t,s)) 
           \Lambda e^{-(t+s)\Lambda}  \hE_0^+ \E_s f_s ds \\
   + \eta_\epsilon(t)  \Lambda e^{-t \Lambda}
          \int_0^\infty \eta_\epsilon(s)  e^{-s\Lambda} \hE_0^+ \E_s f_s ds.
\end{multline*}
The result for the first two terms follows from the proof of Theorem~\ref{thm:modelendpointbdd} by duality, only using the
boundedness of $\E$ on $\mY$. 
For the last term, as the variables $t$ and $s$ split, it suffices to show uniform boundedness and convergence of 
$$
  L_2\to \mY: h\mapsto \eta_\epsilon(t)  \Lambda e^{-t \Lambda} h
$$
and
$$
  \mY\to L_2: f_t\mapsto \int_0^\infty  \eta_\epsilon(s)  e^{-s\Lambda} \hE_0^+ \E_s f_s ds
$$
separately.
For the first operator, this follows directly from the square function estimates for $\Lambda$.
To handle the second, it suffices to estimate
$B_0\int_0^\infty  \eta_\epsilon(s)  e^{-s\Lambda} \hE_0^+ \E_s f_s ds= 
\int_0^\infty  \eta_\epsilon(s)  e^{-s\tilde\Lambda} \tE_0^+ \E_s f_s ds$, since 
$B_0$ is strictly accretive on $\mH\supset \ran(e^{-s\Lambda} \hE^+_0)$.
To this end, we apply Lemma~\ref{lem:princpart} with $U_s:= \tE_0^+ \E_s P_\mH$, where $P_\mH$ is orthogonal projection
onto $\mH$, and $\Lambda$ replaced by $\tilde\Lambda$.
The hypothesis there on boundedness of
$$
  \mH \to Y^*: h\mapsto U_s^* e^{-s\tilde\Lambda^*}h= P_\mH \E_s^* e^{-s|DB_0^*|} \chi^+(DB_0^*)h,
$$
follows from the maximal estimate in Theorem~\ref{thm:QEandNT} (with $B_0$ replaced by $B_0^*$),
the assumed boundedness of $\E^*:\mX\to \mY^*$ and $L_2$ boundedness of $\chi^+(DB_0^*)$ and $P_\mH$.
This completes the proof.
\end{proof}

By inspection of the proofs above, the limit operator $S_A$, both for $f\in \mX$ and $f\in \mY$, is seen to be 
$$
     S_A f_t = \lim_{\epsilon\to 0}\left( \int_\epsilon^{t-\epsilon} \Lambda e^{-(t-s)\Lambda} \hE_0^+ \E_s f_s ds +
      \int_{t+\epsilon}^{\epsilon^{-1}}\Lambda e^{-(s-t)\Lambda} \hE_0^- \E_s f_s ds \right),
$$
with convergence in $L_2(a,b;L_2)$ for any $0<a<b<\infty$.
This holds since we may equally well choose to work with the characteristic function
$\eta^0(t)= \chi_{(1,\infty)}(t)$ instead of the piecewise linear function $\eta^0$ defined below
(\ref{eq:Tplus}-\ref{eq:Tminus}). 
The only places we need the continuity of $\eta^0$ are in Theorems~\ref{thm:inteqforNeu}
and \ref{thm:inteqforDir} below.

We now turn to defining rigorously the second  integral operator needed  for the Dirichlet problem. 
Write $C_b(X,V)$ for the space of bounded and continuous functions on $X$ with values in $V$.

\begin{prop}    \label{prop:endpointdir}
  The operators
 $$
     \tS_A^\epsilon f_t := \int_0^t \eta_\epsilon^+(t,s) e^{-(t-s)\tilde\Lambda} \tE_0^+ \E_s f_s ds -
      \int_t^\infty\eta_\epsilon^-(t,s) e^{-(s-t)\tilde\Lambda} \tE_0^- \E_s f_s ds 
 $$
   are bounded $\mY\to C_b(\clos\R_+; L_2)$,  
   with $\sup_{t>0}\|\tS_A^\epsilon f_t\|_2\lesssim \|\E\|_* \|f\|_\mY$, uniformly
    for $\epsilon>0$,
   and there is a limit operator $\tS_A\in \mL(\mY, C_b(\clos\R_+; L_2))$ such that 
   $\lim_{\epsilon\to 0} \|\tS_A^\epsilon f_t- \tS_A f_t\|_2= 0$  locally uniformly for $t\in(0,\infty)$,
   for any $f\in\mY$.
   The limit operator satisfies $S_Af= D\tS_Af$ in $\R^{1+n}_+$ distributional sense, 
   where $S_A= S_A^\mY$ is the operator from Proposition~\ref{prop:TAboundsonXY},
   and has limits 
 $$
   \lim_{t\to 0} \|\tS_A f_t - \tilde h^-\|_2 =0= \lim_{t\to \infty}  \|\tS_A f_t\|_2,
 $$
 where 
 $\tilde h^-:= -\int_0^\infty e^{-s\tilde\Lambda}\tE_0^- \E_s f_s ds\in \tE_0^-L_2$,  for any $f\in \mY$.
 \end{prop}
 
Note that $\tS_A^\epsilon f_t= 0$ when $t\notin (\epsilon, \epsilon^{-1})$, so convergence 
$\tS_A^\epsilon f_t\to \tS_A f_t$ is not uniform up to $t=0$.
 By inspection of the proof below, the limit operator is seen to be
\begin{equation}  \label{def:tildeSA}
     \tS_A f_t = \int_0^t e^{-(t-s)\tilde\Lambda} \tE_0^+ \E_s f_s ds -
      \int_t^\infty e^{-(s-t)\tilde\Lambda} \tE_0^- \E_s f_s ds,
 \end{equation}
  where the integrals are weakly convergent in $L_2$ for all $f\in \mY$ and $t>0$.
 
 \begin{proof}
The estimates for $\tS_A^\epsilon$ are more straightforward than those for $S_A^\epsilon$ 
since there is no singularity at $s=t$.
For the $(0,t)$-integral, split it as
$$
   \int_0^t \eta_\epsilon^+(t,s) e^{-(t-s)\tilde\Lambda} (I- e^{-2s\tilde\Lambda}) \tE_0^+ \E_s f_s ds 
   +   e^{-t\tilde\Lambda}\int_0^{t} \eta_\epsilon^+(t,s) e^{-s\tilde\Lambda} \tE_0^+ \E_s f_s ds.
$$
For the first term, we write $e^{-(t-s)\tilde\Lambda} (I- e^{-2s\tilde\Lambda})= 
\tfrac s{t-s} ((t-s)\tilde \Lambda e^{-(t-s)\tilde\Lambda} )((I- e^{-2s\tilde\Lambda})/(s\tilde\Lambda))$
to obtain the estimate
$\|e^{-(t-s)\tilde\Lambda} (I- e^{-2s\tilde\Lambda})\| \lesssim s/t$.
From this uniform boundedness and convergence, locally uniformly in $t$, as $\epsilon\to 0$ follows
by Cauchy--Schwarz inequality.
For the second term we use uniform boundedness of $e^{-t\tilde\Lambda}$ and duality 
to estimate it by
$$
   \sup_{\|h\|_2=1}\left| \int_0^t (\E_s^* e^{-s\tilde \Lambda^*} (\tE_0^+)^* h, f_s) \eta^+_\epsilon(t,s)ds\right|
   \lesssim \|\E^*\|_* \|f\|_\mY,
$$
using Lemma~\ref{lem:princpart} as in the proof of Proposition~\ref{prop:TAboundsonXY}. 
Moreover, the $L_2$ difference between the integral at $\epsilon$ and $\epsilon'$ is bounded by
$\int_0^t \|f_s\|_2^2 |\eta_\epsilon^+(t,s)- \eta_{\epsilon'}^+(t,s)|^2 sds\to 0$ as $\epsilon, \epsilon'\to 0$
for fixed $t$, which proves the convergence.

The proof for the $(t,\infty)$-integral in $\tS_A^\epsilon$ is similar, splitting it as 
$$
   \int_t^\infty \eta_\epsilon^-(t,s) e^{-(s-t)\tilde\Lambda} (I- e^{-2t\tilde\Lambda}) \tE_0^- \E_s f_s ds 
   +   e^{-t\tilde\Lambda}\int_t^\infty \eta_\epsilon^-(t,s) e^{-s\tilde\Lambda} \tE_0^- \E_s f_s ds,
$$
and using the estimate
$\|e^{-(s-t)\tilde\Lambda} (I- e^{-2t\tilde\Lambda})\| \lesssim t/s$ for the first term
and Lemma~\ref{lem:princpart} for the second.

Since clearly $\tS_A^\epsilon f\in C_b(\R_+; L_2)$, its locally uniform limit $\tS_A f$ also belongs to
$C_b(\R_+; L_2)$.
To find the limits of $\tS_A f_t$ at $0$ and $\infty$,
since $\tS_A:\mY\to C_b(\R_+; L_2)$ is bounded
it suffices to consider $f\in\mY$ such that $f_t=0$ for $t\notin(a,b)$, with $0<a<b<\infty$
fixed but arbitrary.
In this case,
 $$
     \tS_A f_t = \int_{a<s<\min(t,b)} e^{-(t-s)\tilde\Lambda} \tE_0^+ \E_s f_s ds -
      \int_{\max(t,a)<s<b} e^{-(s-t)\tilde\Lambda} \tE_0^- \E_s f_s ds
 $$
 satisfies
$\tE_0^+ \tS_A f_t=0$  for $t<a$ and $\tE_0^- \tS_A f_t=0$ when $t>b$, from which the two limits 
$\lim_{t\to 0}\tE_0^+ \tS_A f_t=0=\lim_{t\to \infty}\tE_0^- \tS_A f_t$ follow.
For the remaining two limits $\lim_{t\to \infty}\tE_0^+ \tS_A f_t$ and $\lim_{t\to 0}\tE_0^- \tS_A f_t$, we use that
$$
\lim_{t\to \infty}\int_a^b \|e^{-(t-s)\tilde\Lambda}\tE_0^+ \E_s f_s\|_2 ds=0=
\lim_{t\to 0}\int_a^b \|(e^{-(s-t)\tilde\Lambda}-e^{-s\tilde\Lambda})\tE_0^- \E_s f_s\|_2 ds
$$
by dominated convergence.

To verify the identity $S_A= D\tS_A$, note that 
$\int_0^\infty (\phi_t, S_A^\epsilon f_t)dt = \int_0^\infty (D\phi_t, \tS_A^\epsilon f_t) dt$
for all $f\in \mY$ and $\phi\in C_0^\infty(\R^{1+n}_+;\C^{(1+n)m})$. This uses the relations
$B_0 \hE_0^\pm= B_0 E_0^\pm B_0^{-1}P_{B_0\mH}= \tE_0^\pm P_{B_0 \mH}= \tE_0^\pm$
by (\ref{eq:DBBDintertw}).
Let $\epsilon\to 0$ and use $S_A^\epsilon$ and $\tS_A^\epsilon$ convergence.
This completes the proof.
\end{proof}

%
%
%
%
%
\section{The Neumann and regularity problems}    \label{sec:neu}

  Throughout this section, $A$ denotes $t$-dependent coefficients satisfying (\ref{eq:boundedmatrix}) and (\ref{eq:accrassumption}),
  and $A_0\in L_\infty(\R^n; \mL(\C^{(1+n)m}))$ denotes $t$-independent coefficients which are 
  strictly accretive on $\mH$. 
  We let $B:= \hat A$ and $B_0:= \hat A_0$ be the transformed strictly accretive coefficients
  from Proposition~\ref{prop:divformasODE}, and define $\E:= B_0-B$.

For the Neumann and regularity problems, one seeks estimates of the gradient $g=\nabla_{t,x}u$
rather than the potential $u$. With a slight abuse of notation, we say below that $g$ solves
the divergence form equation when $u$ does so.

\begin{defn}    \label{defn:Xsol}
  By an $\mX$-{\em solution to the divergence form equation}, with coefficients $A$, we mean a function
  $g\in L_2^\loc(\R_+; L_2(\R^n; \C^{(1+n)m}))$, with estimate $\|\tN(g)\|_2<\infty$,
  which satisfies
$$
\begin{cases}
  \divv_{t,x} Ag=0, \\
  \curl_{t,x} g=0,
\end{cases}
$$
in $\R^{1+n}_+$ distributional sense.
\end{defn}

Note that the boundary behaviour of $g$ is not specified in this definition; we show existence
of a limit in appropriate sense (see also Section~\ref{sec:misc}). 
This will allow us to formulate in what sense the boundary data is prescribed.
Our representation and the boundary behavior  of $\mX$-solutions bears on the following result.
Recall that $\Lambda= |DB_0|$.

\begin{thm}   \label{thm:inteqforNeu}
  Assume that $\|\E\|_*<\infty$. Let $f \in \mX$. Then  $f\in   L_2^\loc(\R_+;\mH)$ and solves
  $\pd_t f + DB f=0$ in $\R^{1+n}_+$ distributional sense 
  if and only if
    $f$ satisfies the equation
 $$
    f_t= e^{-t\Lambda} h^+ + S_Af_t,\qquad \text{for some } h^+\in E_0^+\mH.
 $$
In this case, $f$ has limits 
\begin{equation}  \label{eq:neuavlim}
\lim_{t\to 0} t^{-1} \int_t^{2t} \| f_s -f_0 \|_2^2 ds =0=
\lim_{t\to \infty} t^{-1} \int_t^{2t} \| f_s \|_2^2 ds,
\end{equation}
where $f_0:= h^++h^-$ and $h^-:= \int_0^\infty \Lambda e^{-s\Lambda}\hE_0^- \E_s f_s ds\in E_0^-\mH$,
with estimates 
$$
 \max (\|h^+\|_2, \|h^-\|_2) \approx \|f_0\|_2 \lesssim \|f\|_\mX.
$$
If furthermore $\|\E\|_*$ is sufficiently small, then there are estimates
$$
\|h^-\|_2\lesssim  \|h^+\|_2 \approx \|f_0\|_2\approx\| f \|_\mX.
$$
\end{thm}

\begin{proof}
 (i) 
 Assume that $f\in \mX \cap   L_2^\loc(\R_+;\mH)$ and satisfies the differential equation.
  To show that $f_t= e^{-t\Lambda} h^+ + S_Af_t$,
  we choose $\eta_\epsilon^\pm$ for  $\eta^\pm$ in Proposition~\ref{prop:preinteqs}
  and subtract the equations to obtain
\begin{multline}    \label{eq:identityforlimits}
   -\int_0^t (\pd_s\eta^+_\epsilon)(t,s) e^{-(t-s)\Lambda} E_0^+ f_s ds 
   +\int_t^\infty (\pd_s\eta_\epsilon^-)(t,s) e^{-(s-t)\Lambda} E_0^- f_s ds \\ 
   =\int_0^t \eta_\epsilon^+(t,s) \Lambda e^{-(t-s)\Lambda} \hE_0^+ \E_s f_s ds 
    + \int_t^\infty \eta_\epsilon^-(t,s) \Lambda e^{-(s-t)\Lambda} \hE_0^- \E_s f_s ds.
\end{multline}
  Note that $DB_0=\pm |DB_0|= \pm \Lambda$ on $E^\pm_0 \mH$.
  We fix $0<a<b<\infty$ and consider the equation in $L_2(a,b;\mH)$.
  By Proposition~\ref{prop:TAboundsonXY}, the right hand side converges to $S_A f$ in $L_2(a,b;\mH)$.
  When $t\in(a,b)$ and $\epsilon$ is small, the left hand side equals
\begin{multline}  \label{eq:neubdyterms}
  \epsilon^{-1}\int_\epsilon^{2\epsilon} e^{-s\Lambda} (E_0^+ f_{t-s}+ E_0^- f_{t+s}) ds \\ -
  \epsilon^{-1} \int_\epsilon^{2\epsilon} e^{-(t-s)\Lambda} E_0^+ f_s ds
  -2\epsilon \int_{(2\epsilon)^{-1}}^{\epsilon^{-1}} e^{-(s-t)\Lambda} E_0^- f_s ds.
\end{multline}
To prove that the first term converges to $f$ in $L_2(a,b;\mH)$, adding and subtracting the term 
$\epsilon^{-1} \int_\epsilon^{2\epsilon} e^{-s\Lambda} f_t ds= 
e^{-\epsilon\Lambda}(\epsilon\Lambda)^{-1}(I- e^{-\epsilon\Lambda})f_t$
shows that the square of the $L_2(a,b;\mH)$ norm of the difference is bounded by
$$
  \int_a^b \left\| \left(I- e^{-\epsilon \Lambda}\frac{I- e^{-\epsilon \Lambda}}{\epsilon \Lambda}\right)f_t \right\|_2^2 dt
  + \int_a^b \epsilon^{-1} \int_\epsilon^{2\epsilon} \| f_t- E_0^+ f_{t-s}- E_0^- f_{t+s} \|_2^2 ds dt\rightarrow 0
$$
as $\epsilon\to 0$, using Proposition~\ref{prop:opcalcconv} for the functional calculus,
dominated convergence, and the identity $f_t= E_0^+ f_t + E_0^- f_t$ as $f$ is $\mH$-valued.

Next consider the last term in (\ref{eq:neubdyterms}). For any $\phi\in L_2(a,b;\mH)$, we have
\begin{multline*}
   \int_a^b \left( \epsilon \int_{(2\epsilon)^{-1}}^{\epsilon^{-1}} e^{-(s-t)\Lambda} E_0^- f_s ds, \phi_t \right)dt \\ =
   \epsilon \int_{(2\epsilon)^{-1}}^{\epsilon^{-1}} \left(f_s,  \int_a^b (e^{-(s-t)\Lambda^*}- e^{-s\Lambda^*} ) (E_0^-)^* \phi_t dt +
   e^{-s\Lambda^*} (E_0^-)^* \int_a^b \phi_t dt  \right) ds.
\end{multline*}
From the $\sup-L_2$ estimate in Lemma~\ref{lem:XlocL2} for $f$, the estimate 
$\|e^{-(s-t)\Lambda^*}- e^{-s\Lambda^*} \|\lesssim t/s$ and the strong limit
$\lim_{s\to\infty}e^{-s\Lambda^*}(E^-_0)^* =0$, it follows that the last term in (\ref{eq:neubdyterms})
converges weakly to $0$. Hence the middle term must converge weakly in $L_2(a,b;L_2)$ as well,
and we may replace $e^{-(t-s)\Lambda}$ by $e^{-t\Lambda}$ since $\| e^{-(t-s)\Lambda}- e^{-t\Lambda} \|\lesssim s/t$.
We get that
$$
  \int_a^b \left(e^{-t\Lambda} (\epsilon^{-1} \int_\epsilon^{2\epsilon} E_0^+ f_s ds), \phi_t \right) dt
  = \left(  \epsilon^{-1} \int_\epsilon^{2\epsilon} E_0^+ f_s ds, \int_a^b e^{-t\Lambda^*} \phi_t dt\right)
$$
converges for all $\phi\in L_2(a, b;L_2)$.
Since $\epsilon^{-1} \int_\epsilon^{2\epsilon} E_0^+ f_s ds$ are uniformly bounded in $\mH$ by
Lemma~\ref{lem:XlocL2}, and since functions $\int_a^b e^{-t\Lambda^*} \phi_t dt$ are
dense in $B_0\mH\approx \mH^*$
 (for example $\int_\epsilon^{2\epsilon} e^{-t\Lambda^*}\epsilon^{-1}\phi dt\to P_{B_0\mH}\phi$), 
it follows that $\epsilon^{-1} \int_\epsilon^{2\epsilon} E_0^+ f_s ds$ converges
weakly to a function $h^+\in E_0^+ \mH$, and that the weak limit of the middle term in (\ref{eq:neubdyterms}) is
$e^{-t\Lambda} h^+$.
In total, this proves that $f_t -e^{-t\Lambda} h^+= S_A f_t$.

(ii)
Conversely, assume that $f\in \mX$ and $f_t= e^{-t\Lambda} h^+ + S_Af_t$.  First, $f$ is $\mH$-valued since $e^{-t\Lambda} h^+ \in E_{0}^+\mH$ and $S_{A}f_{t}\in \mH$ for almost every $t$.
To verify that $f$ satisfies the differential equation, 
note that $(\pd_t + DB_0)e^{-t\Lambda}h^+=0$.
It suffices to show
that for $\phi\in C^\infty_0(\R^{1+n}_+; \C^{(1+n)m})$ there is convergence
$$
  \int (-\pd_t \phi_t+ B_0^*D\phi_t, f^\epsilon_t) dt\to \int (D\phi_s, \E_s f_s)ds,\qquad \epsilon\to 0,
$$
where $f^\epsilon_t:=S_A^\epsilon f_t$.
For the term $S^+_\epsilon \hE_0^+ \E f$, Fubini's theorem and integration by parts give
\begin{multline*}
  \int_0^\infty \int_0^t \eta^+_\epsilon(t,s) ((-\pd_t+ \Lambda^*)\phi_t, \Lambda e^{-(t-s)\Lambda} \hE_0^+ \E_s f_s  ) ds dt \\
= - \int_0^\infty \left( \int_s^\infty \eta^+_\epsilon(t,s) \pd_t (e^{-(t-s) \Lambda^*}  \Lambda^*\phi_t)dt, \hE_0^+ \E_s f_s  \right) ds \\
  =\int_0^\infty \left( \int_s^\infty (\pd_t\eta^+_\epsilon)(t,s)  e^{-(t-s) \Lambda^*}  \Lambda^*\phi_t dt, \hE_0^+ \E_s f_s  \right) ds \\
    \to \int_0^\infty ( \Lambda^*\phi_s, \hE_0^+ \E_s f_s  ) ds = \int_0^\infty ( D\phi_s, \tE_0^+ \E_s f_s  ) ds.
\end{multline*}
Adding the corresponding limit for the term $S^-_\epsilon \hE_0^- \E f$ gives the stated result.
Note that $\tE_0^+ + \tE_0^-=P_{B_0\mH}$ and $D P_{B_0\mH}= D$.

(iii)
To show the limits, note that
$E_0^+ f- e^{-t\Lambda}h^+= S^+ \hE_0^+ \E f\in \mY^*$, and
by inspection of the proof of Theorem~\ref{thm:modelendpointbdd} we see that
$E_0^- f- e^{-t\Lambda}\int_0^\infty \Lambda e^{-s\Lambda}\hE_0^- \E_s f_s ds\in\mY^*$.
From this, the limits for $f$ follow.

(iv)
It remains to prove the estimates. Note that (\ref{eq:hardysplit}) and Lemma~\ref{lem:XlocL2} show that 
$$
\max(\|h^+\|^2_2, \|h^-\|^2_2)\approx \|f_0\|^2_2= \lim_{t\to 0}t^{-1}\int_t^{2t}\|f_s\|^2_2 ds\lesssim \|f\|_\mX^2.
$$
Proposition~\ref{prop:TAboundsonXY} shows that $\|S_A\|_{\mX\to \mX}\le 1/2$ if 
$\|\E\|_*$ is sufficiently small. In this case $I-S_A$ is an isomorphism on $\mX$ with $\|(I-S_A)^{-1}\|_{\mX\to\mX}\le 2$.
Using this together with Theorem~\ref{thm:QEandNT}, we get estimates
$\|f\|_\mX=\|(I-S_A)^{-1}e^{-t\Lambda}h^+\|_\mX \approx \|h^+\|_2$. This proves the stated estimates
and completes the proof.
\end{proof}

\begin{thm}\label{cor:inteqforNeu} 
  Assume that $\|\E\|_*<\infty$. 
   Then $g$ is an $\mX$-solution to the divergence form equation with coefficients $A$ if and only if
   the corresponding conormal gradient $f= [(Ag)_\no, g_\ta]^t\in \mX$ satisfies the equation
 $$
    f_t= e^{-t\Lambda} h^+ + S_Af_t,\qquad \text{for some } h^+\in E_0^+\mH.
 $$
In this case, $g$ has limits 
\begin{equation}  \label{eq:neuavlim}
\lim_{t\to 0} t^{-1} \int_t^{2t} \| g_s -g_0 \|_2^2 ds =0=
\lim_{t\to \infty} t^{-1} \int_t^{2t} \| g_s \|_2^2 ds,
\end{equation}
where $g_0:= [ (B_0f_0)_\no, (f_0)_\ta ]^t$ and $ \|g_0\|_2\lesssim \|g\|_\mX$ holds.
If furthermore $\|\E\|_*$ is sufficiently small, then there are estimates
$$
  \|h^+\|_2 \approx \|g_0\|_2\approx \|g\|_\mX.
$$
\end{thm}

Note that these limits  are stronger than $L_2$ convergence 
of Cesaro means $t^{-1}\int_t^{2t} g_s ds$ (see Section \ref{sec:misc}).

\begin{proof}  
The equivalence follows right away from Proposition \ref{prop:divformasODE} and the first part of Theorem \ref{thm:inteqforNeu}. 
Note that  $\|g_t\|_2\approx \|f_t\|_2$ and  $\|g\|_\mX\approx \|f\|_\mX$. 
The limit for $g$ at $t=\infty$ is immediate from that of $f$, so is the limit of the tangential part $g_{\ta}$ of $g$. 
The limit of $g_{\no}$ at $t=0$ follows from that of $Bf$. To see this, write $$ B_t f_t-B_0 f_0= B_0(f_t-f_0) - \E_t f_t.$$
Since $\E_t f_t\in \mY^*$, we have $\lim_{t\to 0} t^{-1}\int_t^{2t}\|\E_sf_s\|^2_2 ds=0$.
The limit of $B_0(f_t-f_0)$ at $t=0$ follows  from the
limit of $f$.  The rest of the proof is immediate.
\end{proof}

We note the following immediate corollary to Theorem~\ref{cor:inteqforNeu}.

\begin{cor}     \label{eq:WPcompatNeu}
  Assume that coefficients $A=A_0$ are $t$-independent. Then $g$ is an $\mX$-solution to the 
  divergence form equation if and only if the associated conormal gradient $f$ is
   a generalized Cauchy extension $C_{0}^+h^+$ of some $h^+\in E_0^+ \mH$, \textit{i.e.}
  $$
  f_t= e^{-t\Lambda} h^+ \qquad \text{for all } t>0.
$$
In this case, $h^+= \lim_{t\to 0}f_t$ in $L_2$ sense. 
In particular, the class of $\mX$-solutions in Definition~\ref{defn:Xsol} coincides  with the class of 
solutions in \cite[Definition 2.1(i-ii)]{AAM} for $t$-independent coefficients.
\end{cor}
That the solutions in \cite{AAM} are of this form was shown in the proof of \cite[Theorem 2.3]{AAM}.
Note that the operator $T_A|_{\mH}$ used in \cite{AAM} is similar to our operator $DB_0|_{\mH}$, 
as in \cite[Definition 3.1]{AAM}. 

\begin{rem} We may ask whether $\mY^*$ could be used as a solution space  for gradients of solutions. The answer is no because we have seen that functions in $\mY^*$ vanish in some sense at the boundary so that free evolutions $e^{-t\Lambda} f_{0}\in \mY^*$ if and only if $f_{0}=0$.   
A second question is then how far gradients of solutions are  from being in $\mY^*$. 
 Inspection of (iii) in the proof of Theorem~\ref{thm:inteqforNeu} reveals that  $f_t- e^{-t\Lambda} f_0 \in \mY^*$,  \textit{i.e.}\,  
  the free evolution $e^{-t\Lambda}f_0 \in \mX$ is the only term responsible for $f$ (hence $g$) to belong to $\mX$ and not to $\mY^*$.
\end{rem}

We are now ready study BVPs. We recall that for the class of solutions used in \cite{AAM}, with $t$-independent coefficients $A_0$,
well-posedness of the Neumann and regularity problems was shown to be equivalent to the maps
\begin{align*}
  E_0^+ \mH \to L_2(\R^n;\C^m) &: h^+\mapsto (h^+)_\no, \\
  E_0^+ \mH \to \sett{f\in L_2(\R^n;\C^{nm})}{\curl_x f=0} &: h^+\mapsto (h^+)_\ta,
\end{align*}
being isomorphisms respectively.
From Corollary~\ref{eq:WPcompatNeu}, it is equivalent to well-posedness in the class of $\mX$-solutions.

We now turn to $t$-dependent perturbations of the coefficients.

\begin{cor}   \label{cor:neuregmain}
   Assume that the Neumann problem for $A_0$ is well-posed. 
   Then there exists $\epsilon>0$ such that for any $t$-dependent coefficient matrix
   $A$ with $\|\E\|_*<\epsilon$, the
   Neumann problem is well-posed for $A$ in the following sense.

   Given any function $\bphi\in L_2(\R^n;\C^m)$, there is a unique 
   $\mX$-solution $g$ to the divergence form equation with coefficients $A$, whose trace $g_0$ satisfies
   $(A_0 g_0)_\no= \bphi$. Moreover, this solution has estimates
 $$
   \|\tN(g)\|_2 \approx \|g_0\|_2 \approx \|\bphi\|_2.
 $$
The same holds true when the Neumann problem is replaced by the regularity problem
and the boundary condition $(A_0 g_0)_\perp= \bphi$ is replaced by
$(g_0)_\ta= \bphi\in L_2(\R^n;\C^{nm})$ such that $\curl_x \bphi=0$.
\end{cor}

\begin{proof}
Throughout the proof, we assume that $\|\E\|_*$ is small enough, 
so that $I-S_A$ is invertible on $\mX$ by Proposition~\ref{prop:TAboundsonXY}.
To solve the Neumann problem, we compute $f$ by  making the ansatz 
\begin{equation}    \label{eq:Cauchyint}
  f= (I-S_A)^{-1} C_{0}^+ h^+,
\end{equation}
for some $h^+\in E_0^+ \mH$ to be determined,
and calculate its full trace 
$$
f_0=h^++ \int_0^\infty \Lambda e^{-s\Lambda} \hE_0^- \E_s f_s  ds, 
$$
using Theorems~\ref{thm:inteqforNeu}  and \ref{cor:inteqforNeu}. 
We see that $f$ satisfies the Neumann boundary condition $(f_0)_\no= \bphi$ if and only if
$h^+$ solves the equation $\Gamma_A h^+= \bphi$, where $\Gamma_A : E_0^+\mH\to L_2(\R^n;\C^m)$ is the operator
$$
  \Gamma_A:  
  h^+\mapsto \left( h^++ \int_0^\infty \Lambda e^{-s\Lambda} \hE_0^- \E_s f_s ds \right)_\no.
$$
with $f$ given by  \eqref{eq:Cauchyint}. Since $\Gamma_{A_0}h^+= (h^+)_\no$, a computation using Lemma~\ref{lem:princpart} and the boundedness of $(I-S_{A})^{-1}$ on $\mX$ implies
 $\|\Gamma_A-\Gamma_{A_0}\|_{L_2\to L_2}\lesssim \|\E\|_*$. 
By assumption $\Gamma_{A_0}$ is an invertible operator, and thus $\Gamma_A$ remains an isomorphism
whenever $\|\E\|_*$ is sufficiently small.
Thus, in this case we can, given $\bphi$, calculate $h^+= \Gamma_A^{-1}\bphi$ with $\|h^+\|_2\approx \|\bphi\|_2$ 
and find a unique solution $g$
to the Neumann problem, with estimates $\|g\|_\mX\approx \|g_0\|_2\approx \|h^+\|_2\approx \|\bphi\|_2$.

For the regularity problem, we proceed as for the Neumann problem, but instead solve for $h^+$ in the equation
$\left( h^++ \int_0^\infty \Lambda e^{-s\Lambda} \hE_0^- \E_s f_s ds \right)_\ta=\bphi$.
\end{proof}

%
%
%
%
%
\section{The Dirichlet problem}    \label{sec:dirichlet}

  Throughout this section, $A$ denotes $t$-dependent coefficients satisfying (\ref{eq:boundedmatrix}) and (\ref{eq:accrassumption}),
  and $A_0\in L_\infty(\R^n; \mL(\C^{(1+n)m}))$ denotes $t$-independent coefficients which are 
  strictly accretive on $\mH$. 
  We let $B:= \hat A$ and $B_0:= \hat A_0$ be the transformed strictly accretive coefficients
  from Proposition~\ref{prop:divformasODE}, and define $\E:= B_0-B$.

\begin{defn}    \label{defn:Ysol}
  By a $\mY$-{\em solution to the divergence form equation}, with coefficients $A$, we mean a
  function $u\in W^{1,\loc}_2(\R^{1+n}_{+}, \C^m)$, 
  with estimate $\int_0^\infty \| g_t \|_2^2 t dt<\infty$
  of its gradient $g:= \nabla_{t,x} u$  which satisfies
$
  \divv_{t,x} A g=0
$
in $\R^{1+n}_+$ distributional sense.
\end{defn}

We will prove in Theorem~\ref{thm:Ysols}, under $\|\E\|_{*}<\infty$,
that any $\mY$-solution belongs to $C(\R_+;L_2)$, modulo constants.
Note also that we do not assume any limits of $u$ at $t=0$ or $t=\infty$, but will prove such below.
This will allow us to formulate in what sense the boundary values are prescribed. 
When discussing $C(\R_+;L_2)$ limits and bounds of $\mY$-solutions $u$, we shall
identify the function $u$ with a $C(\R_+;L_2)$ function modulo constants.

Our representation  of $\mY$-solutions bears on  the following result.
Recall that $\tilde \Lambda= |B_0D|$.

\begin{thm}   \label{thm:inteqforDir}
  Assume that $\|\E\|_*<\infty$. Let $f\in \mY$.
   Then $f\in   L_2^\loc(\R_+;\mH)$ and solves
  $\pd_t f + DB f=0$ in $\R^{1+n}_+$ distributional sense if and only if $f$ satisfies the equation
$$
   f_t = De^{-t\tilde\Lambda}\tilde h^+ + S_A f_t, \qquad \text{for some } \tilde h^+\in \tE_0^+ L_2.
$$
In this case, let $v_t:= e^{-t\tilde\Lambda}\tilde h^+ + \tS_A f_t$.
Then $f=Dv$, $v$ satisfies the equation
\begin{equation}    \label{eq:diffeqv}
  \pd_tv+ B D v= -P\E f
\end{equation}
in $\R^{1+n}_+$ distributional sense,
where $P= I-\tE_0^+-\tE_0^-=I-P_{B_{0}\mH}$ is the projection onto $\mH^\perp$ along $B_0\mH$,  and  
$v_t$ has $L_2$ limits
\begin{equation}    \label{eq:dirlimits}
   \lim_{t\to 0} \| v_t  - v_0\|_2 = 0= \lim_{t\to \infty}  \| v_t \|_2, 
\end{equation}
where $v_0:= \tilde h^+ + \tilde h^-$ and
$\tilde h^-:= -\int_0^\infty e^{-s\tilde\Lambda}\tE_0^- \E_s f_s ds\in \tE_0^- L_2$,
with estimates 
$$
\max( \|\tilde h^+\|_2, \|\tilde h^-\|_2) \approx \|v_0\|_2 \lesssim\sup_{t>0}\|v_t\|_2\lesssim \|f\|_\mY.
$$ 
If furthermore $\|\E\|_*$ is sufficiently small, then there are estimates
$$
   \|\tilde h^-\|_2 \lesssim\|\tilde h^+\|_2\approx \sup_{t>0} \|v_t\|_2 \approx \|f\|_\mY.
$$
\end{thm}

\begin{proof}
  (i)   
  We assume that  $f\in \mY \cap  L_2^\loc(\R_+;\mH)$ and  satisfies the differential equation.
As in the proof of Theorem~\ref{thm:inteqforNeu}, we aim to take limits $\epsilon\to 0$
in equation (\ref{eq:identityforlimits}).
 By Proposition~\ref{prop:TAboundsonXY}, the right hand side converges in $\mY$ to $S_A f$. 
 Fix $0<a<b<\infty$.
 For $t\in (a,b)$ and small $\epsilon$, the left hand side equals
\begin{multline}  \label{eq:dirbdyterms}
  \epsilon^{-1}\int_\epsilon^{2\epsilon} e^{-s\Lambda} (E_0^+ f_{t-s}+ E_0^- f_{t+s}) ds \\ -
  \epsilon^{-1} \int_\epsilon^{2\epsilon} e^{-(t-s)\Lambda} E_0^+ f_s ds
  -2\epsilon \int_{(2\epsilon)^{-1}}^{\epsilon^{-1}} e^{-(s-t)\Lambda} E_0^- f_s ds.
\end{multline}
As in the proof of Theorem~\ref{thm:inteqforNeu}, the first term converges to $f$
in $L_2(a,b;L_2)$.
The $L_2$-norm of the last term is bounded by
$\epsilon \int_{(2\epsilon)^{-1}}^{\epsilon^{-1}}  \|f_s\|_2 ds\lesssim \epsilon
  ( \int_{(2\epsilon)^{-1}}^{\epsilon^{-1}} \|f_s\|^2 sds )^{1/2}$,
and hence converges to $0$, uniformly for $t\in(a,b)$.

We conclude that $\tilde f_t^\epsilon:=  \epsilon^{-1} \int_\epsilon^{2\epsilon} e^{-(t-s)\Lambda} E_0^+ f_s ds$
converges in $L_2(a,b;L_2)$ as $\epsilon\to 0$.
In fact, since $\sup_{t>0} \|e^{-t\Lambda}\|_{L_2\to L_2}<\infty$ we have
$$
  \|\tilde f_{t_0}^\epsilon- \tilde f_{t_0}^{\epsilon'}\|_2 
  \le \frac 1{b-a}\int_a^b \| e^{-(t_0-t)\Lambda} ( \tilde f_t^\epsilon- \tilde f_t^{\epsilon'}) \|_2 dt
  \lesssim \left( \int_a^b \| \tilde f_t^\epsilon- \tilde f_t^{\epsilon'} \|^2_2 dt\right)^{1/2},
$$
when $t_0>b$.
Hence, since $(a,b)$ is arbitrary, $\tilde f_t^\epsilon$ converges in $L_2$, locally uniformly in $t$.
Call the limit $\tilde f$ and note that it coincides with $f- S_A f\in \mY$ for a.e. $t>0$. 
Fix $t_0>0$ and note that 
$\tilde f_{t+t_0} = \lim_{\epsilon\to 0} e^{-t\Lambda} \tilde f_{t_0}^\epsilon  = e^{-t\Lambda}\tilde f_{t_0}$
and that in fact $\tilde f_{t_0}\in E_0^+\mH$ by the definition of $\tilde f_{t_0}$.
The estimate
$$
  \sup_{t_0>0}\int_0^\infty \|e^{-t\Lambda} \tilde f_{t_0}\|_2^2 tdt \le \|\tilde f\|_\mY^2
   \lesssim \|f\|_\mY^2
$$
follows. Consider the restriction $\Lambda_+$ of $\Lambda$ to $E_0^+\mH$,
which is a closed and injective operator with dense domain and range.
We claim that $\tilde f_{t_0}\in \dom(\Lambda_+^{-1})$.
To see this, by duality it suffices to show that
$$
  |((\Lambda_+^{-1})^* \phi, \tilde f_{t_0})|\lesssim \|\phi\|_2,\qquad \text{for all } \phi\in \dom((\Lambda_+^{-1})^*).
$$
As in the proof of Proposition~\ref{prop:opcalcboundedness}, we use an identity
$\int_0^\infty (t\Lambda_+ e^{-t\Lambda_+})^2 \tilde f_{t_0} dt/t= 4^{-1}\tilde f_{t_0}$ to estimate
$$
  |((\Lambda_+^{-1})^* \phi, \tilde f_{t_0})| \approx 
  \left| \int_0^\infty (t \Lambda_+^* e^{-t\Lambda_+^*} \phi, t e^{-t\Lambda_+}\tilde f_{t_0}) \frac {dt}t\right| 
  \lesssim \|\phi\|_2 \|f\|_\mY.
$$
Hence the claim. As $\dom(\Lambda_+^{-1})= \ran(\Lambda_+)\subset \ran(D)$,
this shows that $\tilde f_{t_0}= D \tilde h^+_{t_0}$, where $\tilde h^+_{t_0}\in B_0 E_0^+\mH= \tE_0^+L_2$
has bounds $\|\tilde h^+_{t_0}\|_2\lesssim \|f\|_\mY$, uniformly in $t_0$.
From the identity 
$\tilde f_{t+t_0}= e^{-t\Lambda} \tilde f_{t_0}= e^{-t\Lambda} D\tilde h^+_{t_0}$, 
we get
$$
  \int_a^b(\phi_t, \tilde f_{t+t_0}) dt= \left(\int_a^b De^{-t\Lambda^*} \phi_t dt, \tilde h^+_{t_0}\right),
$$
for any $\phi\in L_2(a,b;L_2)$.
Here the left hand side converges as $t_0\to 0$, and the functions $\int_a^b De^{-t\Lambda^*} \phi_t dt$
are dense in $\mH$.
(For example $\int_\epsilon^{2\epsilon} De^{-t\Lambda^*} \epsilon^{-1}\phi dt\to D\phi$.)
Since $\|\tilde h^+_{t_0}\|_2$ is uniformly bounded, it follows that $\tilde h_{t_0}^+\to \tilde h^+$ weakly in
$\tE_0^+ L_2$ as $t_0\to 0$.
Letting $t_0\to 0$ in $\tilde f_{t+t_0}= e^{-t\Lambda} D \tilde h_{t_0}^+= De^{-t\tilde\Lambda}\tilde h_{t_0}^+$,
we obtain $f_t-S_A f_t= \tilde f_t= D e^{-t\tilde\Lambda} \tilde h^+$ for a.e. $t>0$.

(ii)
Conversely, assume that $f\in \mY$  and $f_t = De^{-t\tilde\Lambda}\tilde h^+ + S_A f_t$. The right hand side is $\mH$-valued, so $f \in   L_2^\loc(\R_+;\mH)$ as well. As in Theorem~\ref{thm:inteqforNeu}, we verify that $f$
satisfies the differential equation, and we omit the details. 

(iii) Introduce  $v_t= e^{-t\tilde\Lambda}\tilde h^++\tS_A f_t$, so that $f=Dv$ and the stated limits on $v$ follow from Propositions~\ref{prop:opcalcconv} and \ref{prop:endpointdir}.
To prove \eqref{eq:diffeqv}, compute
\begin{multline*}
  \pd_t(e^{-t\tilde\Lambda}\tilde h^+ + \tS^\epsilon_A f_t)=
  -B_0De^{-t\tilde\Lambda}\tilde h^+ - B_0D\tS^\epsilon_A f_t\\
  +\epsilon^{-1}\int_\epsilon^{2\epsilon} e^{-s\tilde \Lambda} (\tE_0^+\E_{t-s}f_{t-s}+ \tE_0^-\E_{t+s}f_{t+s})ds,
\end{multline*}
for $\epsilon<<t$.
This uses the result for the operator $B_0D$ analogous to Proposition~\ref{prop:Cauchyextension}.
The claim follows by letting $\epsilon\to 0$, using Proposition~\ref{prop:TAboundsonXY}
for convergence of $D\tS_A^\epsilon= S^\epsilon_A$, and a calculation as in part (i)
of the proof of Theorem~\ref{thm:inteqforNeu} for the last term.

 To prove the estimates, note that the square function estimates for $B_0 D$ and the accretivity of $B_0$ on $\mH$ show
that 
$$
\|\tilde h^+\|_2\approx \|B_0De^{-t\tilde \Lambda}\tilde h^+\|_\mY \approx \|De^{-t\tilde \Lambda}\tilde h^+\|_\mY\lesssim \|f\|_\mY+ \|S_A f\|_\mY\lesssim \|f\|_\mY.
$$
From Proposition~\ref{prop:endpointdir}, we also obtain the estimates $\max(\|\tilde h^+\|_2, \|\tilde h^-\|_2)\approx\|v_0\|_2\le \sup_{t>0}\|v_t\|_2\lesssim \|\tilde h^+\|_2+ \|f\|_\mY\lesssim \|f\|_\mY$,
where we have used the topological splitting $B_0\mH= \tE_0^+L_2\oplus \tE_0^-L_2$ in the first equivalence.

(iv)  Finally, Proposition~\ref{prop:TAboundsonXY} shows that $\|S_A\|_{\mY\to\mY}\le 1/2$ if $\|\E\|_*$ is sufficiently small.
In this case $I-S_A$ is an isomorphism on $\mY$,
giving the estimate 
$$
\|f\|_\mY\lesssim \|De^{-t\tilde \Lambda}\tilde h^+\|_\mY.
$$
As $\|De^{-t\tilde \Lambda}\tilde h^+\|_\mY \approx \|\tilde h^+\|_2$,
this proves the stated estimates and completes the proof.
\end{proof}

We can now prove a rigidity theorem for $\mY$-solutions.

\begin{thm}   \label{thm:Ysols}
   Let $u$ be a $\mY$-solution to the divergence form equation, with coefficients $A$.
   Assume that $\|\E\|_*<\infty$.
   Then there is a constant $c\in \C^m$ such that 
 $$
    u= c-v_\no
 $$
 almost everywhere, where $v\in C(\R_+;L_2(\R^n, \C^{(1+n)m}))$ is the vector-valued potential from
 Theorem~\ref{thm:inteqforDir} obtained from the conormal gradient $f$ of $u$.
Identifying the functions $u$ and $c-v_\no$,
there are $L_2$-limits $\lim_{t\to 0} \|u_t- u_0\|_2=0$, $u_0:= c-(v_0)_\no$,
 and $\lim_{t\to \infty} \|u_t-c\|_2=0$,
and there are bounds
$$
  \|u_0-c\|_2\le \sup_{t>0}\| u_t-c\|_2 \lesssim \|\nabla_{t,x}u\|_\mY.
$$
If furthermore $\|\E\|_*$ is sufficiently small, then with $\tilde h^+$ as in Theorem~\ref{thm:inteqforDir},
$$
  \|\tilde h^+\|_2 \approx \|\nabla_{t,x}u\|_\mY.
$$
\end{thm}

\begin{proof}
Let $f=[(A\nabla_{t,x}u)_\no, \nabla_x u]^t \in \mY \cap  L_2^\loc(\R_+;\mH)$ as in Proposition \ref{prop:divformasODE} and then let
$\tilde h^+$ and $v$ be as in Theorem~\ref{thm:inteqforDir}. For the equality $u=c-v_{\no}$, it is enough to show that $\nabla_{t,x}u= -\nabla_{t,x}v_\no$ in $\R^{1+n}_+$ distributional sense. 
 It is clear that $\nabla_{x}u=f_{\ta}=-\nabla_{x}v_{\no}$. 
Using \eqref{eq:diffeqv}, we have
 $
   \pd_t v_\no+ (BDv)_\no= 0,
 $
 because normal parts of functions in $\mH^\perp$ are zero.
 Since $(BDv)_\no= (Bf)_\no= (\nabla_{t,x}u)_\no = \pd_t u$, we conclude that
 $\pd_{t}u= -\pd_{t}v_\no$.

 The stated limits and bounds now follow from Theorem~\ref{thm:inteqforDir} and $\|f\|_{\mY}\approx \|\nabla_{t,x}u\|_{\mY}$.
\end{proof}

The constant $c$ in Theorem~\ref{thm:Ysols} can also be calculated as the limit 
$$
   c= \lim_{d\to\infty}(u, \tau_d\phi), \qquad \tau_d\phi(t,x):= \phi(t-d, x),
$$
for any $\phi\in C_0^\infty(\R^{1+n};\C^m)$ with $\int \phi dtdx=1$.
In particular this limit does not depend on $\phi$. 
So if this limit is zero, we obtain a solution that vanishes at $\infty$ in averaged sense.
(This is thus equivalent to vanishing at $\infty$ in $L_{2}$ sense as defined in Section \ref{sec:statement}.)
This is akin to the classical pointwise limit at infinity required to eliminate constants for representing harmonic functions in the upper half-space. Here the averages replace the pointwise values.

We also note the following corollary to Theorems~\ref{thm:Ysols} and \ref{thm:inteqforDir}.

\begin{cor}     \label{eq:WPcompatDir}
   Assume that coefficients $A=A_0$ are $t$-independent. Then $u$ is a $\mY$-solution to the 
  divergence form equation if and only if it is the normal part  of a generalized Cauchy 
  extension $\widetilde C_{0}^+ \tilde h^+$ of some $\tilde h^+\in \tE_0^+ L_2$, modulo constants, 
  \textit{i.e.}  
$$
   u_t= (e^{-t\tilde\Lambda} \tilde h^+)_\perp +c\qquad \text{for all } t>0 \text{ and some } c\in \C^m.
$$
  In particular, the class of $\mY$-solutions in Definition~\ref{defn:Ysol} that vanish at $\infty$ 
  in $L_{2}$ sense
  coincides  with the class of 
  solutions in \cite[Definition 2.1(iii)]{AAM} for $t$-independent coefficients.
\end{cor}
That the solutions considered in \cite{AAM} are of this form follows from \cite[Lemma 4.2]{AAM}
and the proof of \cite[Theorem 2.3]{AAM}.
Note that the operator $T_A|_{\mH}$ used in \cite{AAM} is similar to our operator $B_0D|_{B_0 \mH}$, as in \cite[Definition 3.1]{AAM}. 

We are now ready to study BVPs. 
We recall that for the class of solutions used in \cite{AAM}, with $t$-independent coefficients $A_0$,
well-posedness of the Dirichlet problem was shown to be equivalent to the map
$$
  \tE_0^+ L_2 \to L_2(\R^n;\C^m) : \tilde h^+\mapsto (\tilde h^+)_\no
$$
being an isomorphism.
From Corollary~\ref{eq:WPcompatDir}, it is equivalent to well-posedness in the class of
$\mY$-solutions. Remark that well-posedness implies that the map $u_{0} \mapsto u_{t}$ is a $C_{0}$-semigroup on $L_{2}(\R^n;\C^m)$ and this corollary also shows that the results in \cite{Au} concerning the domain of this semi-group obtained for solutions in the sense of \cite{AAM}
apply to $\mY$-solutions.

We now turn to $t$-dependent perturbations of the coefficients. 

\begin{cor}   \label{cor:dirmain}
   Assume that the Dirichlet problem for $A_0$ is well-posed.
   
   Then there exists $\epsilon>0$ such that for any $t$-dependent coefficient matrix
   $A$ with $\|\E\|_*<\epsilon$, the
   Dirichlet problem is well-posed for $A$ in the following sense.

   Given any function $\bphi \in L_2(\R^n;\C^m)$, there is a unique 
   $\mY$-solution $u$ to the divergence form equation with coefficients $A$, with boundary trace $u_0=\bphi$.
   Moreover, this solution has estimates
$$
   \|\nabla_{t,x}u\|_\mY\approx \sup_{t>0}\|u_t\|_2 \approx \|\bphi\|_2.
$$
   \end{cor}
   
\begin{proof}
Throughout the proof, we assume that $\|\E\|_*$ is small enough, 
so that $I-S_A$ is invertible on $\mY$ by Proposition~\ref{prop:TAboundsonXY}.
To solve the Dirichlet problem, we make the ansatz 
\begin{equation}   \label{eq:diransatz}
  u= \Big( (I + \tS_A (I-S_A)^{-1}D) \widetilde C_{0}^+ \tilde h^+ \Big)_\perp \end{equation}
for some $\tilde h^+\in \tE_0^+L_2$. Theorems~\ref{thm:inteqforDir} and \ref{thm:Ysols} show
that $u$ is a $\mY$-solution to the divergence form equation with coefficients $A$ and 
that all $\mY$-solutions with $L_2$ trace are of this form.
Moreover, the Dirichlet boundary condition $u_0= \bphi$ is satisfied if and only if
$\tilde h^+$ solves the equation $\widetilde \Gamma_A \tilde h^+= \bphi$, where $\widetilde \Gamma_A : \tE_0^+L_2\to L_2(\R^n;\C^m)$ is the operator
$$
  \widetilde \Gamma_A :
   \tilde h^+\mapsto  \left( \tilde h^+- \int_0^\infty e^{-s\tilde\Lambda} \tE_0^- \E_s f_s ds \right)_\no,
$$
where
$f:= (I-S_A)^{-1} D\widetilde C_{0}^+ \tilde h^+$. Since $\widetilde \Gamma_{A_0}\tilde h^+= (\tilde h^+)_\no$,  Lemma~\ref{lem:princpart} and the boundedness of $(I-S_{A})^{-1} $ on $\mY$ imply that $\|\widetilde \Gamma_A-\widetilde \Gamma_{A_0}\|_{L_2\to L_2}\lesssim \|\E\|_*$. By assumption
$\widetilde \Gamma_{A_0}$ is an invertible operator, and thus $\widetilde \Gamma_A$ remains an isomorphism
whenever $\|\E\|_*$ is sufficiently small.
Thus, in this case we can, given $\bphi$, calculate $\tilde h^+= \widetilde \Gamma_A^{-1}\bphi$ with $\|\tilde h^+\|_2\approx \|\bphi\|_2$ and find a unique solution $u$ to the Dirichlet problem.
From Theorem~\ref{thm:Ysols}, we get estimates
$$
  \|\bphi\|_2=\|u_{0}\|_{2}\le \sup_{t>0}\|u_t\|_2 \lesssim \|\nabla_{t,x}u\|_\mY \approx \|\tilde h^+\|_2 \approx \|\bphi\|_2.
$$
This proves the theorem.
\end{proof}

\begin{rem}
The tangential part $v_\ta$ of the vector-valued potential 
$$v=(I + \tS_A (I-S_A)^{-1}D)\widetilde C_{0}^+ \tilde h^+$$
can be viewed as a set of generalized conjugate functions to the Dirichlet solution $u$.
Our proof of Theorem~\ref{thm:inteqforDir} 
above eliminates the need of the technical condition at $\infty$ on these conjugate functions
which was required in \cite[Definition 3.1]{AAM}.
\end{rem}

%
%
%
%
%
\section{Further estimates}      \label{sec:further}

In Section~\ref{sec:neu}, we constructed solutions, with estimates on the modified non-tangential maximal function, to
the Neumann and regularity problems with $L_2$ boundary data, and in Section~\ref{sec:dirichlet}
we constructed solutions, with estimates on the square function, to
the Dirichlet problem with $L_2$ boundary data. 
In this section, we prove two theorems which give modified non-tangential maximal function 
estimates for the Dirichlet problem, and, upon some further regularity on the coefficients, square function estimates for the Neumann/regularity
problems.

%
%
%
%
%
\subsection{Maximal function estimates for $\mY$-solutions}    \label{sec:maxest}

\begin{thm}  \label{thm:dirNT}
  Let $A: \R^{1+n}_+\to \mL(\C^{(1+n)m})$ with $\|A\|_\infty<\infty$ and strictly accretive on $\mH$, 
  and assume that there exists $t$-independent coefficients $A_{0}$ with $\|A-A_0\|_C<\infty$.
  Then any $\mY$-solution $u$ to the divergence form equation with coefficients $A$, 
  with boundary trace $u_0\in L_2(\R^n,\C^m)$,
  has modified non-tangential maximal estimates
 $$
   \|u_0\|_2\lesssim \| \tN(u)\|_2 \lesssim  \|\nabla_{t,x}u\|_\mY.
 $$  
\end{thm}

The core of the proof reduces to the following estimate of the operator $\tS_A$.

\begin{lem}   \label{lem:dirNT}
  For any fixed $p\in [1,2)$, the operator $\tS_A$ has estimates
$$
  \| \tN^p((\tS_A f)_\no)\|_2 \lesssim \|\E\|_C \|f \|_\mY.
$$
Here 
$\tN^p(f)(x):= \sup_{t>0}  t^{-(1+n)/p} \|f\|_{L_p(W(t,x))}$ is an $L_p$ modified non-tangential 
maximal function.
\end{lem}

\begin{proof}[Proof of Theorem~\ref{thm:dirNT} modulo Lemma~\ref{lem:dirNT}]
  As in Theorems~\ref{thm:inteqforDir} and \ref{thm:Ysols}, any $\mY$-solution
  $u$ with $L_{2}$ trace (hence, vanishing at $\infty$ in $L_{2}$ sense) can be written
$$
  u_t= (e^{-t\tilde\Lambda}\tilde h^+ + \tS_A f_t)_\no,\qquad \tilde h^+\in \tE_0^+ L_2, f\in \mY.
$$
From Poincar\'e's inequality $\| u -u_{W(t,x)}\|_{L_2(W(t,x))}\lesssim t \|\nabla_{s,y}u\|_{L_2(W(t,x))}$,
where $u_{W(t,x)}$ denotes the average, we obtain the estimate
$
  \|\tN(u)\|_2 \lesssim \|\tN^1(u)\|_2 + \|\nabla_{t,x}u\|_\mY.
$
Theorem~\ref{thm:QEandNT}, Lemma~\ref{lem:dirNT} and Theorem~\ref{thm:inteqforDir} now
apply to give the estimate
$$
  \| \tN^1(u)\|_2\lesssim \|\tilde h^+\|_2+ \|f\|_\mY \approx\|\nabla_{t,x}u\|_\mY.
$$ 
To see the first estimate, write $\tilde h^+= B_0 h^+$ with $h^+\in E_0^+ \mH$, and
apply Theorem~\ref{thm:QEandNT} to get
$\|e^{-t\tilde\Lambda}B_0h^+\|_\mX= \|B_0 e^{-t\Lambda} h^+\|_\mX\lesssim \|h^+\|_2\approx \|\tilde h^+\|_2$.
The lower estimate follows from Lemma~\ref{lem:XlocL2} since
$$
  \|\tN(u)\|_2^2\gtrsim \lim_{t\to 0} t^{-1}\int_t^{2t}\|u_s\|_2^2 ds = \|u_0\|_2^2.
$$
\end{proof}

\begin{proof}[Proof of Lemma~\ref{lem:dirNT}]
Before we start, we remark that $p\mapsto \| \tN^p((\tS_A f)_\no)\|_2$ is increasing, 
so it suffices to consider $p$ close to $2$.
We shall fix the value of $p$ eventually in (iii) below, when we see how close to $2$ it need
to be.
Next it suffices to prove the inequality for $t\mapsto f_t$ compactly supported in $\R_+$.
Indeed, combining Lemma~\ref{lem:XlocL2} and Proposition~\ref{prop:endpointdir}, for all $\epsilon>0$
and $f\in\mY$ we have (since $p\le 2$)
\begin{multline*}
  \| \tN^p(\chi_{(\epsilon,\epsilon^{-1})}(t)  (\tS_A f)_\no )\|_2^2
  \le  \| \tN(\chi_{(\epsilon,\epsilon^{-1})}(t)  (\tS_A f)_\no )\|_2^2  \\
  \lesssim \int_\epsilon^{\epsilon^{-1}} \| (\tS_A f)_\no  \|_2^2 \frac {dt}t
  \lesssim \ln\epsilon\, \sup_{t>0} \| \tS_A f \|_2^2\lesssim \ln\epsilon\, \|f\|^2_\mY.
\end{multline*}
Thus, if $f_\delta:= \chi_{(\delta, \delta^{-1})}(t)f$ for $f\in\mY$, we have for fixed $\epsilon>0$
$$
\| \tN^p(\chi_{(\epsilon,\epsilon^{-1})}(t)  (\tS_A f)_\no )\|_2 \le \liminf_{\delta\to 0}
\| \tN^p(\chi_{(\epsilon,\epsilon^{-1})}(t)  (\tS_A f_\delta)_\no )\|_2.
$$
Now our assumption gives
$$
\| \tN^p(\chi_{(\epsilon,\epsilon^{-1})}(t)  (\tS_A f_\delta)_\no )\|_2
\lesssim \|\E\|_C \|f_\delta\|_\mY \lesssim \|\E\|_C \|f\|_\mY, 
$$
uniformly in $\epsilon$, so for all $f\in \mY$ and $\epsilon>0$ we obtain
$$
\| \tN^p(\chi_{(\epsilon,\epsilon^{-1})}(t)  (\tS_A f)_\no )\|_2
\lesssim  \|\E\|_C \|f\|_\mY.
$$
It remains to let $\epsilon\to 0$ and apply the monotone convergence theorem.

(i)
We now fix $t\mapsto f_t$ compactly supported in $\R_+$ and write
$$
  \tS_A f_t=\int_0^t e^{-(t-s)\tilde\Lambda} \tE_0^+ \E_s f_s ds - \int_t^\infty e^{-(s-t)\tilde\Lambda} \tE_0^- \E_s f_s ds=: I-II.
$$
Most of the time we use the pointwise inequality
$\tN^p\le \tN$.
It is only for one term, estimated in (iii) below, that we require $p<2$.

Split the integral $I$ as
$$
  I= \int_0^t e^{-(t-s)\tilde\Lambda}(I-e^{-2s\tilde\Lambda}) \tE_0^+ \E_s f_s ds
  +e^{-t\tilde\Lambda} \int_0^t e^{-s\tilde\Lambda} \tE_0^+ \E_s f_s ds = I_1+I_2.
$$
As in the proof of Proposition~\ref{prop:endpointdir}, the kernel of $I_1$ has
bounds $s/t$, giving the estimate
\begin{multline}   \label{eq:schurbound}
  \|\tN(I_1)\|_2^2\lesssim \int_0^\infty \|I_1\|_2^2 \frac {dt}t 
  \lesssim \int_0^\infty \left( \int_0^t \frac st \frac{ds}s \right) \left( \int_0^t \frac st \|\E_s f_s\|_2^2 sds \right)  \frac{dt}t \\
  \lesssim \int_0^\infty \|\E_s f_s\|_2^2 sds\le \|\E\|_\infty^2 \|f\|^2_\mY.
\end{multline}
Similarly we split $II= II_1+II_2$ by writing $e^{-(s-t)\tilde\Lambda}=e^{-(s-t)\tilde\Lambda}(I-e^{-2t\tilde\Lambda})+
e^{-t\tilde\Lambda}e^{-s\tilde\Lambda}$, and a Schur estimate similar to (\ref{eq:schurbound}) give the bound for $II_1$.
Next we write
$$
  II_2= e^{-t\tilde\Lambda} \int_0^\infty e^{-s\tilde\Lambda} \tE_0^- \E_s f_s ds
  -e^{-t\tilde\Lambda} \int_0^t e^{-s\tilde\Lambda} \tE_0^- \E_s f_s ds=: II_3-II_4.
$$
By Theorem~\ref{thm:QEandNT}, the term $II_3$ has bound
\begin{multline*}
  \left\| \tN\left( B_0 e^{-t\Lambda} B_0^{-1}P_{B_0\mH} \int_0^\infty e^{-s\tilde\Lambda} \tE_0^- \E_s f_s ds  \right) \right\|_2
  \lesssim \left\| \int_0^\infty e^{-s\tilde\Lambda} \tE_0^- \E_s f_s ds \right\|_2 \\
  = \sup_{\|h\|_2=1}  \left| \int_0^\infty( \E_s^* e^{-s\tilde \Lambda^*} (\tE_0^-)^* h, f_s )ds \right|\lesssim \|\E\|_*\|f\|_\mY.
\end{multline*}

(ii)
It remains to consider $I_2+ II_4= (\tE_0^++ \tE_0^-) e^{-t\tilde\Lambda} \int_0^t e^{-s\tilde\Lambda} \E_s f_s ds$.
Note that $(\tE_0^+ + \tE_0^-)= P_{B_0\mH}$.
Since we only consider the normal component of  $I_2+ II_4$ and $(P_{B_0\mH} h)_\no = h_\no$
for any $h$, it remains to estimate
$e^{-t\tilde\Lambda} \int_0^t e^{-s\tilde\Lambda} \E_s f_s ds$.
To make use of off-diagonal estimates (see Lemma~\ref{lem:lqoffdiag}), we need to replace $e^{-t\tilde\Lambda}$ by
the resolvents $(I+itB_0D)^{-1}$. To this end, define 
$\psi_t(z):= e^{-t|z|}-(1+itz)^{-1}$ and split the integral
\begin{multline*}
  e^{-t\tilde\Lambda} \int_0^t e^{-s\tilde\Lambda} \E_s f_s ds
  = \psi_t(B_0D) \int_0^\infty e^{-s\tilde\Lambda} \E_s f_s ds -\int_t^\infty \psi_t(B_0D) e^{-s\tilde\Lambda} \E_sf_sds \\
  +\int_0^t (I+ itB_0D)^{-1} (e^{-s\tilde\Lambda}-I) \E_sf_sds + (I+itB_0D)^{-1} \int_0^t \E_s f_s ds.
\end{multline*}
For the first term, square function estimates show that $\psi_t(B_0D): L_2\to\mY^*\subset\mX$ is continuous, and
a duality argument like for $II_3$ gives the bound.
For the second and third terms, we note the operator estimates
$$
  \|\psi_t(B_0D) e^{-s\tilde\Lambda}\|= \left\| \frac ts\frac{e^{-t|B_0D|}-(I+itB_0D)^{-1}}{tB_0D} (sB_0D) e^{-s|B_0D|} \right\|
  \lesssim t/s,
$$
and
$$
  \|(I+ itB_0D)^{-1} (e^{-s\tilde\Lambda}-I)\| \lesssim \left\| \frac st\frac{tB_0D}{I+itB_0D} \frac{e^{-s|B_0D|}-I}{sB_0D} \right\|
  \lesssim s/t.
$$
Schur estimates similar to (\ref{eq:schurbound}) give the $\tN$ bounds.

(iii)
It remains to prove the estimate
$$
   \left\| \tN^p \left( (I+itB_0D)^{-1}\int_0^t \E_s f_s ds  \right) \right\|_2 \lesssim \|\E\|_C \|f\|_\mY.
$$
To show this, fix a Whitney box $W(t_0, x_0)$, take $h\in L_q(W(t_0,x_0); \C^{(1+n)m})$, and let
$h=0$ outside $W(t_0,x_0)$. Here $1/p+1/q=1$, $p<2$ and $q>2$.
To bound the $L_p(W(t_0,x_0))$ norm, we do the duality argument
\begin{multline*}
  \frac 1{t_0} \int_{c_0^{-1}t_0}^{c_0t_0}  \left( (I+itB_0D)^{-1}\int_0^t \E_s f_s ds, h_t \right) dt \\
    =\int_0^{c_0t_0} \left( \E_s f_s,  \frac 1{t_0} \int_{\max(c_0^{-1}t_0,s)}^{c_0t_0}  (I-itDB_0^*)^{-1} h_t dt \right) ds\\
    \le\int_{\R^n}\int_0^{c_0t_0}  |\E(s,y)| |f(s,y)| H(y) dsdy,
\end{multline*}
where
$$
 H(y):= \frac 1{t_0} \int_{c_0^{-1}t_0}^{c_0t_0}  |(I-itDB_0^*)^{-1}h_t(y)| dt. 
$$
To handle the tails of $(I-itDB_0^*)^{-1} h_t$, we split the space into annular regions
$\R^n= \bigcup_{k=0}^\infty A_k$, where $A_0:= B(x_0;t_0)$ and $A_k:= (2^kA_0)\setminus (2^{k-1}A_0)$ for 
$k\ge 1$.
Define $f_k(s,y):= \chi_{(0, c_0t_0)}(s)\chi_{A_k}(y)f(s,y)$ and $H_k(y):= \chi_{A_k}(y) H(y)$.
Then Whitney averaging as in the proof of Lemma~\ref{lem:Carleson} gives
\begin{multline*}
   \int_{\R^n}\int_0^{c_0t_0}  |\E(s,y)| |f(s,y)| H(y) dsdy
   \le \sum_{k=0}^\infty \iint_{\R^{1+n}_+} |\E(s,y)| s|f_k(s,y)| H_k(y) \frac{dsdy}s \\
   \approx \sum_{k=0}^\infty\iint_{\R^{1+n}_+} \left(\frac 1{t^{1+n}} \iint_{W(t,x)}  |\E(s,y)| s|f_k(s,y)| H_k(y) dsdy\right) \frac{dtdx}t  \\
  \lesssim 
   \sum_{k=0}^\infty\iint_{\R^{1+n}_+} \sup_{W(t,x)}|\E|  \left(\frac 1{t^{1+n}} \iint_{W(t,x)} |s f_k|^2\right)^{1/2} 
   \left(\frac 1{t^{1+n}} \iint_{W(t,x)} |H_k|^2\right)^{1/2} \frac{dtdx}t    \\
  \lesssim
 \sum_{k=0}^\infty \|\E\|_C \int_{\R^n}\area \left( \tfrac 1{\sqrt{t^{1+n}}}  \|s f_k\|_{L_2(W(t,x))} 
      \tfrac 1{\sqrt{t^{1+n}}}\|H_k\|_{L_2(W(t,x))} \right)\!(z)dz \\
   \lesssim 
 \sum_{k=0}^\infty \|\E\|_C \int_{\R^n}\area \left( \tfrac 1{\sqrt{t^{1+n}}} \|s f_k\|_{L_2(W(t,x))} \right)\!(z) \,
    N_* \left( \tfrac 1{\sqrt{t^{1+n}}}\|H_k\|_{L_2(W(t,x))} \right)\!(z)dz \\
   \lesssim  
 \sum_{k=0}^\infty \|\E\|_C \| \area(s f_k) \|_{L_p(\R^n)} \| M(|H_k|^2)^{1/2}\|_{L_q(\R^n)}.
 \end{multline*}
 Here $\area$ denotes the area function $\area g(z):= (\iint_{|y-z|<cs} |g(s,y)|^2 s^{-(1+n)} dsdy)^{1/2}$ and
 $N_* g(z):=  \sup_{|y-z|<cs} |g(s,x)|$ is the non-tangential maximal function, where $c\in(0,\infty)$ is some
 fixed constant, and $M$ is the Hardy--Littlewood maximal function.
 On the fourth line we used the tent space estimate by Coifman, Meyer and Stein in \cite[Theorem 1(a)]{CMS}.
Since $M: L_{q/2}\to L_{q/2}$ is bounded, we have
\begin{multline*}
  \| M(|H_k|^2)^{1/2}\|_{L_q(\R^n)} \lesssim \|H\|_{L_q(A_k)}\le 
  \frac 1{t_0} \int_{c_0^{-1}t_0}^{c_0t_0}  \|(I-itDB_0^*)^{-1} h_t\|_{L_q(A_k)} dt \\
  \lesssim 2^{-km}\frac 1{t_0} \int_{c_0^{-1}t_0}^{c_0t_0}  \|h_t\|_{L_q(B(x_0; c_0t_0))} dt 
  \lesssim 2^{-km} t_0^{-q} \|h\|_{L_q(W(t_0,x_0))}.
\end{multline*}
The third estimate uses Lemma~\ref{lem:lqoffdiag} below, and thus is where we choose $p<2$
sufficiently close to $2$ so that $2<q<2+\delta$.
We obtain the maximal function estimate
\begin{multline*}
  \tN^p\left( (I+itB_0D)^{-1}\int_0^t \E_s f_s ds  \right)(x_0) \lesssim
  \|\E\|_C \sup_{t_0>0} \sum_{k=0}^\infty 2^{-km} t_0^{n/q-n}  \| \area(s f_k) \|_{L_p(\R^n)} \\
  \lesssim \|\E\|_C \sum_{k=0}^\infty 2^{-k(m- n/p)} \sup_{t_0>0} 
    \left( \frac 1{(2^k t_0)^n}\int_{B(x_0;(2^k+ cc_0) t_0)} |\area(sf)|^p dx \right)^{1/p} \\
    \lesssim \|\E\|_C M( \area(sf)^p )^{1/p}(x_0),
\end{multline*}
where $c$ is the constant from the definition of $\area$ and $m>n/p$.
Since $M:L_{2/p}\to L_{2/p}$ is bounded, this yields
\begin{multline*}
   \left\|\tN^p\left( (I+itB_0D)^{-1}\int_0^t \E_s f_s ds  \right)\right\|_2 \\
   \lesssim
   \|\E\|_C\|M( \area(sf)^p )^{1/p}\|_2\lesssim \|\E\|_C \|\area(sf)\|_2
   \approx \|\E\|_C \|f\|_\mY.
\end{multline*}
This completes the proof of the maximal function estimate.
\end{proof}

The following lemma, which we used above, is contained in \cite[Lemma 2.57]{AAH}.
However, we give a more direct proof here, since the algebraic setup in \cite{AAH} was quite different.

\begin{lem}  \label{lem:lqoffdiag}
  Let $B_0$ be $t$-independent coefficients, strictly accretive on $\mH= \clos{\ran(D)}$.
  Then for each positive integer $m$, there is $C_m<\infty$ and $\delta>0$ such that
$$
  \| (1+it DB_0)^{-1} f \|_{L_q(E)}\le \frac {C_m}{(1+\dist(E,F)/t)^m} \|f\|_{L_q(F)}
$$
for all $t>0$ and sets $E,F\subset \R^n$ such that $\supp f\subset F$, 
and all $q$ such that $|q-2|<\delta$. 
Here $\dist(E,F):= \inf\sett{|x-y|}{x\in E, y\in F}$.
\end{lem}

\begin{proof}
  For $q=2$, these off-diagonal estimates can be proved as in \cite[Proposition 5.1]{elAAM}, 
  using estimates on commutators with bump functions (and replacing the operator $B_0D$ there
  by $DB_0$).
  By interpolation, it suffices to estimate $\| (1+it DB_0)^{-1} f \|_{L_q(\R^n)\to L_q(\R^n)}$, uniformly for $t$
  and $q$ in a neighbourhood of $2$.
  To this end, assume that $(I+ it DB_0)\tilde f= f$.
  As in Proposition~\ref{prop:divformasODE}, but replacing $\pd_t$ by $(it)^{-1}$, this equation is equivalent to
$$
\begin{cases}
  (A_0\tilde g)_\no + it\divv_x(A_0\tilde g)_\ta = (A_0 g)_\no, \\
  \tilde g_\ta - it \nabla_x\tilde g_\no = g_\ta,
\end{cases}
$$
where $A_0, g, \tilde g$ are related to $B_0, f, \tilde f$, respectively, as in Proposition~\ref{prop:divformasODE}.
Using the second equation to eliminate $\tilde g_\ta$ in the first, shows that $\tilde g_\no$ satisfies the divergence form 
equation
$$
  L\tilde g_\no:= \begin{bmatrix} 1 & it\divv_x \end{bmatrix}
   A_0(x)
    \begin{bmatrix} 1 \\ it\nabla_x \end{bmatrix}
    \tilde g_\no =
     \begin{bmatrix} 1 & it\divv_x \end{bmatrix}
      \begin{bmatrix} A_{\no\no} g_\no \\ -A_{\ta\ta} g_\ta \end{bmatrix}.
$$
By the stability result of {\v{S}}ne{\u\i}berg~\cite{sneiberg} it follows that 
the divergence form operator $L$ is an isomorphism $L: W^1_q(\R^n)\to W^{-1}_q(\R^n)$
for $|q-2|<\delta$, giving us the desired estimate
$$
  \|\tilde f\|_q\approx \|\tilde g\|_q \lesssim \|\tilde g_\no\|_q+ t\|\nabla_x \tilde g_\no\|_q + \|g_\ta\|_q 
  \lesssim \|g\|_q\approx \|f\|_q.
$$
\end{proof}

%
%
%
%
%
\subsection{Square function estimates for $\mX$-solutions under $t$-regularity for the coefficients}    \label{sec:sqfcnests}

Looking closely at the equation $\divv_{t,x} Ag=0$, it seems unlikely that 
$\mX$-solutions $g$ would in general satisfy
the square function estimate $\int_0^\infty \|\pd_t g_t\|_2^2 tdt<\infty$, \textit{i.e.}\,  $\pd_t g_t \in \mY$,
when $A$ is $t$-dependent. More precisely, it writes $\pd_{t}(A_{\no\no}g_{\no}+A_{\no\ta}g_{\ta}) + \divv_{x}(A_{\ta\no}g_{\no}+A_{\ta\ta}g_{\ta})=0$, and as  $\pd_{t}$ and multiplication by $A$ do not commute the quantity $\pd_{t}g_{t}$ does not arise.
We show in the next result that $\pd_t g_t \in \mY$ can be obtained upon a further $t$-regularity assumption on $A$. This also improves the regularity of $g_t$ itself. We do not claim this assumption is sharp nor necessary (in particular, it could well be that this regularity on the components $A_{\no\no}, A_{\no\ta}$ suffices). 
This regularity assumption is akin to the one in \cite{KP3}, going back to \cite{D2}, towards $A_{\infty}$ property of the $L$-harmonic  measure with respect to surface measure on bounded Lipschitz domains for real elliptic operators. See also \cite{DPP, DR} for results with smallness assumptions of the derivatives of $A$. The difference wih these works is that we are imposing our coefficients $A$ to be perturbation of ``good'' $t$-independent coefficients. So our next  result neither contains or is contained in the above cited works. Besides, it is again an \textit{a priori} estimate on solutions, so it is valid independently of solvability issues.

\begin{thm}  \label{thm:neuQE}
  Let $A: \R^{1+n}_+\to \mL(\C^{(1+n)m})$ with $\|A\|_\infty<\infty$ and strictly accretive on $\mH$, 
  and assume that there exists $t$-independent coefficients $A_{0}$ with $\|A-A_0\|_*$ sufficiently small.
 If $A$ satisfies the $t$-regularity condition
$$
  \| t\pd_t A \|_* <\infty,
$$
  then any $\mX$-solution $g$ to the divergence form equation, with coefficients $A$, with boundary trace $g_0$
  has regularity $\pd_t g_t \in L_2^\loc(\R_+; L_2)$ with estimates
 $$
   \|\pd_t g_t\|_\mY\lesssim \|g\|_\mX.
 $$
  We also have estimates $\sup_{t>0}\|g_t\|_2\approx \|g\|_\mX$, and $t\mapsto g_t\in L_2$ is continuous with limits
  $\lim_{t\to 0} \|g_t-g_0\|_2=0=\lim_{t\to\infty}\|g_t\|_2$.
  The converse estimate $\|g\|_\mX\lesssim \|\pd_t g\|_\mY$ holds for all $\mX$-solutions $g$,
  provided $\| t\pd_t A \|_*$ is sufficiently small.

  If  $\max(\|t\pd_i A\|_*,\| t\pd_t A \|_*) <\infty$ for some $i=1,\ldots, n$, 
  then $\pd_i g_t \in L_2^\loc(\R_+; L_2)$ for any $\mX$-solution $g$ to the divergence form equation with coefficients $A$, 
  with estimates $\|\pd_i g_t\|_\mY\lesssim \|g\|_\mX$.
  The estimate $\|g\|_\mX\lesssim \|\nabla_x g\|_\mY$ holds for all $\mX$-solutions $g$,
  provided $\| t\nabla_{t,x} A \|_*$ is sufficiently small.
\end{thm}

We do not know whether the smallness assumptions are needed for the converse estimates to hold.
We also remark that the same conclusions hold for the conormal gradient $f$, as is clear from the proof below.

\begin{lem}   \label{lem:neuQE}
  If $h\in \mX$ has distribution derivative $\pd_t h\in \mY$, then $\pd_t (S_A h)\in \mY$ with estimates
$$
  \|\pd_t (S_A h)\|_\mY\lesssim (\|\E\|_* + \|t\pd_t \E\|_*) \|h\|_\mX + \|\E\|_\infty \|\pd_t h\|_{\mY}.
$$
\end{lem}

\begin{proof}[Proof of Theorem~\ref{thm:neuQE} modulo Lemma~\ref{lem:neuQE}]
(i)
  As in the proof of Corollary~\ref{cor:neuregmain}, any $\mX$-solution
  can be written $g= [ (Bf)_\no, f_\ta ]^t$, where
$$
 (I-S_A)f=  e^{-t\Lambda} h^+,\qquad\text{for some } h^+\in E_0^+\mH.
$$
Introduce the auxiliary Banach space 
$Z:= \sett{h\in \mX}{\pd_t h\in \mY}\subset \mX$, with norm $\|h\|_Z:= \|h\|_\mX+ a\|\pd_t h\|_\mY$.
By Proposition~\ref{prop:TAboundsonXY} and Lemma~\ref{lem:neuQE} we have estimates
$\|S_A h\|_\mX\le C\|h\|_\mX$ and $\|\pd_t (S_Ah)\|_\mY\le D \|h\|_\mX+ C\|\pd_t h\|_\mY$,
where we assume $C<1$, and we choose the parameter $a>0$ small enough so that
$$
  \|S_A\|_{Z\to Z}\le C+ aD<1.
$$
Hence $I-S_A$ is invertible on both $\mX$ and $Z$.
Since $e^{-t\Lambda}h^+\in Z$ by Theorem~\ref{thm:QEandNT}, we conclude that 
$f\in Z$ with estimates $\|\pd_t f\|_\mY\lesssim \|f\|_Z\approx \|e^{-t\Lambda}h^+\|_Z\approx \|h^+\|_2$.
For the gradient $g$, this gives the bound 
$\|\pd_t g\|_\mY\lesssim \|t\pd_t B\|_* \|f\|_\mX+ (\|B\|_\infty+1)\|\pd_t f\|_\mY \lesssim \|h^+\|_2\approx\|f\|_\mX
\approx \|g\|_\mX$.

(ii)
To prove the $\sup-L_2$ estimate and trace result for $g_t$, write
$\int_0^\infty s\eta(s) \pd_s g_s ds= \int_0^\infty (\eta(s)+ s\eta'(s)) g_s ds$,
for some $\eta\in C_0^\infty(\R_+)$.
Take the limit as $\eta$ approaches the characteristic function for $(0,t)$ to get
$$
  g_t = \frac 1t\int_0^t g_s ds + \frac 1t\int_0^t \pd_s g_s sds,\qquad \text{a.e. } t>0.
$$
The last term has bound $(\int_0^t \|\pd_s g_s\|^2 sds)^{1/2}$, whereas the first term satisfies
$$
 \left\| \frac 1t\int_0^t g_s ds -g_0 \right\|_2^2
 \le  
  \sum_{k=1}^\infty 2^{-k}\left(\frac 1{2^{-k}t}\int_{2^{-k}t}^{2^{1-k}t} \|g_s-g_0\|_2^2 ds\right)
 \to 0
$$
as $t\to 0$.
Hence the trace claims follow from the square function estimates $\|\pd_t g_t\|_\mY<\infty$.
Moreover, the estimate $\sup_{t>0}\|g_t\|_2\lesssim \|g\|_\mX+ \|\pd_t g\|_\mY\lesssim \|g\|_\mX$ follows.
The converse estimate follows from Theorem~\ref{cor:inteqforNeu}.

An integration by parts, similar to above, shows that
$$
  2 g_{2t}= g_t + \frac 1t \int_t^{2t} g_s ds + \frac 1t \int_t^{2t} \pd_s g_s sds, \qquad \text{a.e. } t>0.
$$
Taking $\limsup_{t\to\infty}$ of both sides, shows $2\limsup_{t\to\infty} \|g_t\|_2= \limsup_{t\to\infty}\|g_t\|_2$.
Since $\|g_t\|_2$ is bounded, we conclude that $\lim_{t\to\infty} \|g_t\|_2=0$. 

(iii)
 To show $\|g\|_\mX\lesssim \|\pd_t g\|_\mY$, consider $f$ satisfying 
 $e^{-t\Lambda}h^+= f_t- S_A f_t$.
 Theorem~\ref{thm:QEandNT} and Lemma~\ref{lem:neuQE} give
$$
  \|h^+\|_2\approx \|\pd_t e^{-t\Lambda}h^+\|_\mY \lesssim \|\pd_t f\|_\mY+
  (\|\E\|_* + \|t\pd_t A\|_*) \|f\|_\mX + \|\E\|_\infty \|\pd_t f\|_{\mY},
$$
 where by Theorem~\ref{thm:inteqforNeu} we have $\|f\|_\mX\approx \|h^+\|_2$
as $\|\E\|_*$ is assumed small enough. 
If in addition $\| t\pd_t A \|_*$ is sufficiently small, then we obtain $\|f\|_\mX\lesssim \|\pd_t f\|_\mY$. 
 As in (i), again using smallness of $\| t\pd_t A \|_*$, this implies $\|g\|_\mX\lesssim \|\pd_t g\|_\mY$.

(iv)
To prove the $x$-regularity result, consider the equation $\pd_t f+DBf=0$, which implies
$$
  \|\pd_t f\|_\mY = \|D P_\mH B f\|_\mY\approx \sum_{i=1}^n \| (P_\mH B) (\pd_i f) + P_\mH (\pd_i B)f\|_\mY
$$
since $D= D P_\mH$ and the operator $D$ has estimates $\|Dh\|_2\approx \sum_{i=1}^n\|\pd_i h\|_2$ for all
$h\in\dom(D)\cap \mH$. (The latter is straightforward to verify with the Fourier transform.)
Here $P_\mH$ denotes orthogonal projection onto $\mH$; it commutes with $\pd_i$.
This yields the bound
$$
  \|\pd_i f\|_{\mY} \approx \| (P_\mH B)\pd_i f\|_{\mY}
  \lesssim \|\pd_t f\|_\mY + \|t(\pd_i B) f\|_{\mY^*} \lesssim (1+ \|t\pd_i B\|_*) \|f\|_\mX\lesssim \|f\|_\mX
$$
if $\max(\|t\pd_i A\|_*,\| t\pd_t A \|_*) <\infty$, where we used that $P_\mH B_t:\mH\to \mH$ is an isomorphism in the 
first comparison.
Conversely, if $\| t\pd_t A \|_*$ is sufficiently small, then
$$
  \|f\|_\mX\lesssim \|\pd_t f\|_\mY \lesssim \sum_{i=1}^n (\|\pd_i f\|_\mY + \|t\pd_i B\|_* \|f\|_\mX),
$$
where the first estimate is by (iii). Using next that $\sum_{i=1}^n \|t\pd_i B\|_*$ is small enough, this implies 
$\|f\|_\mX\lesssim \|\nabla_x f\|_\mY$.

As in (i) above, these estimates translate to $\|\pd_i g\|_\mY\lesssim \|g\|_\mX$ and $\|g\|_\mX\lesssim \|\nabla_x g\|_\mY$
respectively.
\end{proof}

\begin{proof}[Proof of Lemma~\ref{lem:neuQE}]
  Assume that the coefficients $A$ satisfy $\|A-A_0\|_*<\infty$ and has distribution 
  derivative $\pd_t A\in L_\infty^\loc(\R^{1+n}_+; \mL(\C^{(1+n)m}))$ such that
  $\|t\pd_t A\|_* <\infty$.
  Fix $h\in\mX$ with distribution derivative $\pd_t h\in \mY$.
  By Theorem~\ref{prop:TAboundsonXY}, $\int_a^b \| S_A h_t- S_A^\epsilon h_t \|_2^2 dt\to 0$
  as $\epsilon\to 0$, where
$$
  S_A^\epsilon h_t:= \int_0^t \eta_\epsilon^+(t,s) \Lambda e^{-(t-s)\Lambda} \hE_0^+ \E_s h_s ds -
  \int_t^\infty \eta_\epsilon^-(t,s) \Lambda e^{-(s-t)\Lambda} \hE_0^- \E_s h_s ds =I-II.
$$
Hence it suffices to bound $\|\pd_t (S_A^\epsilon h)\|_\mY$, uniformly for $\epsilon>0$.

(i)
Differentiate $I$ and write
\begin{multline*}
  t\pd_t(I)=  \int_0^t (t\pd_t \eta_\epsilon^+) \Lambda e^{-(t-s)\Lambda} \hE_0^+ \E_s h_s ds 
   -  \int_0^t \eta_\epsilon^+ (t-s) \Lambda^2 e^{-(t-s)\Lambda} \hE_0^+ \E_s h_s ds \\
   -  \int_0^t \eta_\epsilon^+ (\pd_s \Lambda e^{-(t-s)\Lambda})  \hE_0^+(s \E_s h_s) ds 
   = \int_0^t (t\pd_t \eta_\epsilon^++ s\pd_s\eta_\epsilon^+) \Lambda e^{-(t-s)\Lambda} \hE_0^+ \E_s h_s ds \\
   -  \int_0^t \eta_\epsilon^+ (t-s) \Lambda^2 e^{-(t-s)\Lambda} \hE_0^+ \E_s h_s ds   
   + \int_0^t \eta_\epsilon^+ \Lambda e^{-(t-s)\Lambda}  \hE_0^+ \pd_s(s \E_s h_s) ds= I_1-I_2+I_3. 
\end{multline*}
Note that in $I_3$ the distribution derivative $\pd_s(s \E_s h_s)$ extends its action to test functions
$s\mapsto (\eta^+_\epsilon(t,s)\lambda e^{-(t-s)\Lambda} \hE_0^+ )^*\phi$, for any 
$\phi\in\mH$.
Theorem~\ref{thm:weakmaxreg} and Lemma~\ref{lem:Carleson} give the estimate
$$
  \|I_3\|_{\mY^*}\lesssim \| \pd_t(t\E_t h_t)\|_{\mY^*}\lesssim ( \|\E\|_* + \|t\pd_t \E \|_*) \|h\|_\mX + \|\E\|_\infty \| \pd_t h \|_\mY.
$$
To bound $I_2$, we apply Lemma~\ref{lem:psiasintop}, using the bounds
$$
   \int_0^t |(t-s)\lambda^2 e^{-(t-s)\lambda} |s ds\lesssim t
   \qquad\text{and}\qquad \int_s^\infty |(t-s)\lambda^2 e^{-(t-s)\lambda}|dt\lesssim 1,
$$
which shows $\|I_2\|_{\mY^*}\lesssim\|\E h\|_{\mY^*}\lesssim \|\E\|_* \|h\|_\mX$.
To estimate $I_1$, we calculate
$$
  (t\pd_t + s\pd_s) \eta_\epsilon^+(t,s)= \tfrac{t-s}\epsilon (\eta^0)'(\tfrac{t-s}\epsilon)\eta_\epsilon(t)\eta_\epsilon(s)
  +  \eta^0(\tfrac{t-s}\epsilon) (t\eta_\epsilon'(t)\eta_\epsilon(s)+ s \eta_\epsilon(t)\eta_\epsilon'(s)).
$$
From this, we verify that $| (t\pd_t + s\pd_s) \eta_\epsilon^+|\lesssim \chi_{\supp \nabla\eta_\epsilon^+}\le 1$.
Hence an estimate as in the proof of Theorem~\ref{thm:weakmaxreg} shows that
$\|I_1\|_{\mY^*}\lesssim \|\E\|_* \|h\|_\mX$.

(ii)
Next we differentiate $II$ and write
\begin{multline*}
  t\pd_t(II)=  \int_t^\infty (t\pd_t \eta_\epsilon^-) \Lambda e^{-(s-t)\Lambda} \hE_0^- \E_s h_s ds 
   -  \int_t^\infty t\eta_\epsilon^- (\pd_s \Lambda e^{-(s-t)\Lambda})  \hE_0^-\E_s h_s ds  \\
    =  \int_t^\infty t(\pd_t\eta_\epsilon^- +\pd_s \eta_\epsilon^-) \Lambda e^{-(s-t)\Lambda} \hE_0^- \E_s h_s ds 
   +  \int_t^\infty \eta_\epsilon^- \tfrac ts \Lambda e^{-(s-t)\Lambda}  \hE_0^-s\pd_s(\E_s h_s) ds \\
   = II_1+ II_2.
\end{multline*}
To bound $II_2$, we apply Lemma~\ref{lem:psiasintop} using the bounds
$$
   \int_t^\infty |(t/s)\lambda e^{-(s-t)\lambda} |s ds\lesssim t
   \qquad\text{and}\qquad \int_0^s |(t/s)\lambda e^{-(s-t)\lambda}|dt\lesssim 1,
$$
which shows $\|II_2\|_{\mY^*}\lesssim \|t\pd_t \E\|_* \|h\|_\mX+ \|\E\|_\infty \|\pd_t h\|_\mY$.
To estimate $II_1$, we calculate
$$
  t(\pd_t + \pd_s) \eta_\epsilon^+(t,s)=
  t\eta^0(\tfrac{t-s}\epsilon) (\eta_\epsilon'(t)\eta_\epsilon(s)+  \eta_\epsilon(t)\eta_\epsilon'(s)).
$$
The last term is supported on $s\in(1/(2\epsilon), 1/\epsilon)$, $t\in(\epsilon, s-\epsilon)$, where 
it is bounded by $\epsilon t\lesssim t/s$. Thus estimates as for $II_2$ apply.
The first term is supported on $t\in(\epsilon, 2\epsilon)$, $s\in(t+\epsilon, 1/\epsilon)$ 
(and another component which can be taken together with the last term) and is bounded 
by $1$.
Splitting this remaining term as in (\ref{eq:splitofbadtinftyterm}), it suffices to estimate
\begin{multline*}
  \left\| \chi_{(\epsilon, 2\epsilon)}(t) t\eta'_\epsilon(t) e^{-t\Lambda}\int_0^\infty \eta_\epsilon(s)
  \Lambda e^{-s\Lambda} \hE_0^- \E_s h_s ds  \right\|_{\mY^*} \\
  \lesssim\left( \frac 1\epsilon\int_\epsilon^{2\epsilon}\left\| e^{-t\Lambda}\int_0^\infty \eta_\epsilon(s)
  \Lambda e^{-s\Lambda} \hE_0^- \E_s h_s ds \right\|_2^2 dt \right)^{1/2} \\
  \lesssim \left\| \int_0^\infty \eta_\epsilon(s) \Lambda e^{-s\Lambda} \hE_0^- \E_s h_s ds \right\|_2 
  \lesssim \|\E h\|_{\mY^*} \lesssim \|\E\|_* \|h\|_\mX,
\end{multline*}
using the uniform boundedness of $e^{-t\Lambda}$ and Lemma~\ref{lem:princpart}.
This completes the proof.
\end{proof}

%
%
%
%
%
\section{Miscellaneous remarks and open questions}    \label{sec:misc}

(i)
The condition $\tN(\nabla_{t,x} u)\in L_2$ implies that Whitney averages $\frac 1{|W(t,y)|}\iint_{W(t,y)}u$
converge non-tangentially for almost every $x$, \textit{i.e.}\,  with $|y-x|<\alpha t$ for some $\alpha<\infty$,
to some $u_0(x)$ with $u_0$ belonging to the closure of $C_0^\infty(\R^n)$ with respect to $\|\nabla_x f\|_2<\infty$.
Furthermore, $t^{-1}\int_t^{2t} \nabla_x u_s ds$ converges weakly to $\nabla_x u_0$ in $L_2$ as $t\to 0$
(compare Theorem~\ref{thm:NeuLip}(i)). In particular $\|\nabla_x u_0\|_2\lesssim \|\tN(\nabla_{t,x}u)\|_2$.
This is essentially in \cite[p. 461-462]{KP}, where it is done on the unit ball instead of the upper half space,
and with pointwise values instead of averages, working with $u$'s solving a real symmetric equation.
However, the result has nothing to do with BVPs, but is a result on a function space.

(ii)
Assume that $A\in L_\infty(\R^{1+n}_+; \mL(\C^{(1+n)m}))$ and that $\tN(\nabla_{t,x}u)\in L_2$ with $u$ satisfying
(\ref{eq:divform}) in $\R_+^{1+n}$ distributional sense.
Then there exists $g\in \dot H^{-1/2}(\R^n;\C^m)$ such that 
\begin{equation}   \label{eq:Green}
  \iint_{\R^{1+n}_+} (A\nabla_{t,x} u, \nabla_{t,x} \phi) dtdx = (g, \phi|_{\R^n}),\qquad
\text{for all } \phi\in C_0^\infty(\R^{1+n};\C^m).
\end{equation}
If $\pd_{\nu_A} u(s,x):= (A\nabla_{t,x}u(s,x))_\no$ for all $s>0$, $x\in \R^n$, then $t^{-1}\int_t^{2t} \pd_{\nu_A}u_s ds$
converges weakly to $-g$ in $L_2$ as $t\to 0$.
In particular $\|g\|_2\lesssim \|\tN(\nabla_{t,x}u)\|_2$.
This is again essentially \cite{KP} for the unit ball. See \cite[Lemma 4.3(iii)]{AAAHK} for an argument in $\R^{1+n}_+$.
The equality (\ref{eq:Green}) justifies that $g$ is called the Neumann data. 
This result has nothing to do with accretivity of $A$, boundedness suffices.
Compare again Theorem~\ref{thm:NeuLip}(i).

(iii)
Theorem~\ref{thm:DirLip}(i) contains {\em a priori} estimates on $\mY$-solutions. 
A natural question is to reverse the {\em a priori} estimates for such systems.
Does a weak solution to (\ref{eq:divform})  with $\|A-A_0\|_C<\infty$ and $\tN(u)\in L_2$
satisfy $\|\nabla_{t,x}u\|_\mY\lesssim \|\tN(u)\|_2$?
Same question replacing $\tN(u)\in L_2$ with $\sup_{t>0}\|u_t\|_2<\infty$.
The smallness of $\|A-A_0\|_C$, which implies well-posedness of the Dirichlet problem for $\mY$-solutions,
yields {\em a posteriori} such estimates.
It would be interesting to have positive answers {\em a priori} (\textit{i.e.}\,  independently of well-posedness)
when $\|A-A_0\|_C<\infty$.

(iv)
Is there existence of $\mX$-solutions to the Neumann and regularity problems with $L_2$ data under 
$\|A-A_0\|_C<\infty$ (or even under the stronger $\int_0^\infty \omega_A(t)^2 dt/t<\infty$, where 
$\omega_A(t):= \sup_{0<s<t} \|A_s-A_0\|_\infty$)?
Is there uniqueness under the same constraint on $A$, provided existence
holds? Recall that tools such as Green's functions are not available here.

(v)
Same questions for $\mY$-solutions and the Dirichlet problem with $L_2$ data.

(vi)
It is likely that $\mY$-solutions have the a.e. non-tangential convergence property for averages:
$\frac 1{|W(t,y)|}\iint_{W(t,y)}u\to u_0(x)$ for a.e. $x\in \R^n$ and $(t,y)\to (0,x)$ in $|y-x|<\alpha t$.
This requires an argument which we leave open.

\bibliographystyle{acm}

\end{document}